\documentclass[reqno]{amsart}
\usepackage{times,amssymb,amsmath,amsfonts,float,nicefrac,subfigure,float,accents,color,stmaryrd}
\usepackage{hyperref}

\usepackage{euscript,graphics}
\usepackage[dvips]{epsfig}

\evensidemargin=\oddsidemargin
\newtheorem{theorem}{Theorem}[section]
\newtheorem{lemma}[theorem]{Lemma}
\newtheorem{proposition}[theorem]{Proposition}

\newtheorem{definition}{Definition}

\renewenvironment{proof}[1]{\noindent {\it Proof~:} #1}
{\ \rule{1mm}{2mm}\medskip}

\renewcommand\qed{\rule{1mm}{2mm}\medskip}

\newcommand{\remove}[1]{}
\renewcommand{\tilde}{\widetilde}

\newcommand\nd{\noindent}

\newlength{\dhatheight}

\def\supp{\qopname\relax{no}{supp}}

\def\rank{\qopname\relax{no}{rk}}

\def\card{\qopname\relax{no}{card}}

\newcommand\nc\newcommand
\nc\bfa{{\boldsymbol a}}\nc\bfA{{\bf A}}\nc\cA{{\mathcal A}}
\nc\bfb{{\boldsymbol b}}\nc\bfB{{\bf B}}\nc\cB{{\mathcal B}}
\nc\bfc{{\boldsymbol c}}\nc\bfC{{\bf C}}\nc\cC{{\mathcal C}}
\nc\bfd{{\boldsymbol d}}\nc\bfD{{\bf D}}\nc\cD{{\mathcal D}}
\nc\bfe{{\boldsymbol e}}\nc\bfE{{\bf E}}\nc\cE{{\mathcal E}}
\nc\bff{{\boldsymbol f}}\nc\bfF{{\bf F}}\nc\cF{{\mathcal F}}
\nc\bfg{{\boldsymbol g}}\nc\bfG{{\bf G}}\nc\cG{{\mathcal G}}
\nc\bfh{{\boldsymbol h}}\nc\bfH{{\bf H}}\nc\cH{{\mathcal H}}
\nc\bfi{{\boldsymbol i}}\nc\bfI{{\bf I}}\nc\cI{{\mathcal I}}
\nc\bfj{{\boldsymbol j}}\nc\bfJ{{\bf J}}\nc\cJ{{\mathcal J}}
\nc\bfk{{\boldsymbol k}}\nc\bfK{{\bf K}}\nc\cK{{\mathcal K}}
\nc\bfl{{\boldsymbol l}}\nc\bfL{{\bf L}}\nc\cL{{\mathcal L}}
\nc\bfm{{\boldsymbol m}}\nc\bfM{{\bf M}}\nc\cM{{\mathcal M}}
\nc\bfn{{\boldsymbol n}}\nc\bfN{{\bf N}}\nc\cN{{\mathcal N}}
\nc\bfo{{\boldsymbol o}}\nc\bfO{{\bf O}}\nc\cO{{\mathcal O}}
\nc\bfp{{\boldsymbol p}}\nc\bfP{{\bf P}}\nc\cP{{\mathcal P}}
\nc\bfq{{\boldsymbol q}}\nc\bfQ{{\bf Q}}\nc\cQ{{\mathcal Q}}
\nc\bfr{{\boldsymbol r}}\nc\bfR{{\bf R}}\nc\cR{{\mathcal R}}
\nc\bfs{{\boldsymbol s}}\nc\bfS{{\bf S}}\nc\cS{{\mathcal S}}
\nc\bft{{\boldsymbol t}}\nc\bfT{{\bf T}}\nc\cT{{\mathcal T}}
\nc\bfu{{\boldsymbol u}}\nc\bfU{{\bf U}}\nc\cU{{\mathcal U}}
\nc\bfv{{\boldsymbol v}}\nc\bfV{{\bf V}}\nc\cV{{\mathcal V}}
\nc\bfw{{\boldsymbol w}}\nc\bfW{{\bf W}}\nc\cW{{\mathcal W}}
\nc\bfx{{\boldsymbol x}}\nc\bfX{{\bf X}}\nc\cX{{\mathcal X}}
\nc\bfy{{\boldsymbol y}}\nc\bfY{{\bf Y}}\nc\cY{{\mathcal Y}}
\nc\bfz{{\boldsymbol z}}\nc\bfZ{{\bf Z}}\nc\cZ{{\mathcal Z}}
\nc\od{{\bar d}}\nc\ow{{\bar w}}\nc\odelta{{\bar\delta}}
\nc\ox{{\bar x}}\nc\oy{{\bar y}}\nc\ou{{\bar u}}
\nc\oh{{\bar h}}

\newcommand\reals{{\mathbb R}}
\newcommand\complexes{{\mathbb C}}
\newcommand\ff{{\mathbb F}}

\newcommand\integers{{\mathbb Z}}
\newcommand\naturals{{\mathbb N}}
\renewcommand\epsilon{\varepsilon}

\newcommand{\I}{\Upsilon}
\newcommand{\R}{{\mathfrak R}}
\nc\ellone{{\ell_1}}
\nc\elltwo{{\ell_2}}
\nc\ellinf{{{\ell_\infty}}}
\nc\ip[2]{\langle #1,#2\rangle}

\newcommand{\beeq}{\begin{eqnarray*}}
\newcommand{\eneq}{\end{eqnarray*}}

\newcommand{\half}{{\nicefrac12}}
\numberwithin{equation}{section}
\nc{\1}{{\mathbbold{1}}}

\newcommand{\bfit}{\bfseries\itshape}

\DeclareSymbolFont{bbold}{U}{bbold}{m}{n}
\DeclareSymbolFontAlphabet{\mathbbold}{bbold}

\usepackage{glossaries}

\makeglossaries
\newglossaryentry{I}{  name={\ensuremath{\I}},  description={index set of classes},  sort=I}
\newglossaryentry{cR}{name={\ensuremath{\cR=\{R_i,i\in \I\}}}, description={partition of $X\times X$}, sort=R}
\newglossaryentry{cX} {name={\ensuremath{\cX(X,\mu,\cR)}}, description={association scheme}, sort=X}
\newglossaryentry{mui} {name={\ensuremath{\mu_i}}, description={valency of the scheme}, sort=m}
\newglossaryentry{Ri}  {name={\ensuremath{R_i, i\in\I}}, description={classes of the scheme}, sort=R}
\newglossaryentry{frakA} {name={\ensuremath{{\mathfrak A}(\cX)}}, description={adjacency algebra of $\cX$}, sort=A}
\newglossaryentry{J} {name={\ensuremath{J}}, description={operator with kernel $j(x,y)=1$ for all $x,y\in X$}, sort=J}
\newglossaryentry{chii} {name={\ensuremath{\chi_i(x,y)}}, description={indicator function of $R_i$}, sort=x}
\newglossaryentry{Ai}  {name={\ensuremath{A_i}}, description={adjacency operator, Eq.~\eqref{eq:Ai}}, sort=A}
\newglossaryentry{pij} {name={\ensuremath{p_{i}(j)}}, description={eigenvalues of adjacency operators}, sort=p2}
\newglossaryentry{mj} {name={\ensuremath{m_j}}, description={multiplicities of $\cX$}, sort=m}
\newglossaryentry{Fsim} {name={\ensuremath{\cF^\sim}}, description={Fourier transform $L_2(X,\mu)\to L_2(\hat X,\hat \mu)$},sort=F}
\newglossaryentry{Fhat} {name={\ensuremath{\cF^\natural}}, description={Fourier transform $L_2(\hat X,\hat \mu) \to L_2(X,\mu)$},
  sort=F1}
\newglossaryentry{N} {name={\ensuremath{\cN}}, description={partition of the group $X$}, sort=N}
\newglossaryentry{Ni} {name={\ensuremath{N_i}}, description={block of the partition}, sort=N1}
\newglossaryentry{I0} {name={\ensuremath{\I_0}}, description={$\{i\in\I: \mu(N_i)>0\}$}, sort=I1}
\newglossaryentry{P} {name={\ensuremath{P}}, description={first eigenvalue matrix},sort=P}
\newglossaryentry{Q} {name={\ensuremath{Q}}, description={second eigenvalue matrix},sort=Q}
\newglossaryentry{Ek} {name={\ensuremath{E_k}}, description={orthogonal projector in $L_2(X,\mu)$},sort=E}
\newglossaryentry{pijk} {name={\ensuremath{p_{ij}^l}}, description={intersection number of $\cX$}, sort=p}
\newglossaryentry{Xj} {name={\ensuremath {X_j}}, description={subgroup of $X$}, sort=X}
\newglossaryentry{nu} {name={\ensuremath{\nu(x)}}, description={discrete valuation}, sort=N10}
\newglossaryentry{rho} {name={\ensuremath{\rho(x)}}, description={metric on $X$}, sort=r}
\newglossaryentry{omegaj} {name={\ensuremath{\omega(j)}}, description={$|X/X_j|$}, sort=o}
\newglossaryentry{ni} {name={\ensuremath{n_i}}, description={$|X_{i-1}/X_i|$,\nopostdesc},sort=n,}
\newglossaryentry{Xjbot}{name={\ensuremath{X_j^\bot}}, description={annihilator subgroup of $X_j$}, sort=X}
\newglossaryentry{Br}{name={\ensuremath{B(r)}},description={ball of radius r},sort=B}
\newglossaryentry{barr}{name={\ensuremath{\bar r}},description={maximum radius in $X$}, sort=r}
\newglossaryentry{tau-}{name={\ensuremath{\tau_-(s)}},description={$\max\{r:r<s\}$}, sort=t}
\newglossaryentry{tau+}{name={\ensuremath{\tau_+(s)}},description={$\min\{r:r>s\}$}, sort=t}
\newglossaryentry{S}{name={\ensuremath{S(r)}},description={sphere of radius $r$ in $X$},sort=S}
\newglossaryentry{pr1r2r3}{name={\ensuremath{p_{r_1,r_2}^{r_3}}},description={intersection numbers of metric schemes},sort=p}
\newglossaryentry{frakR}{name={\ensuremath{{\mathfrak R}}}, description={set of values of the radius in $X$},sort=R}
\newglossaryentry{omegaij}{name={\ensuremath{\omega_{ij}(r)}}, description={},sort=o}
\newglossaryentry{alpha0}{name={\ensuremath{\alpha_0(x)}},description={indicator function of $B(\bar r)$},sort=a}
\newglossaryentry{alphali}{name={\ensuremath{\alpha_{l,i}(x)}}, description={indicator functions of the balls},sort=a}
\newglossaryentry{cA}{name={\ensuremath{\cA}}, description={adjacency algebra of $\cX$},sort=A}
\newglossaryentry{cAm}{name={\ensuremath{\cA_m}}, description={adjacency algebra of $\cX$},sort=A}
\newglossaryentry{cXm}{name={\ensuremath{\cX_m}},description={finite subscheme of $\cX$},sort=X}
\newglossaryentry{cAsph}{name={\ensuremath{\cA^{\text{\rm (sph)}}}},description={adjacency algebra of metric scheme},sort=A}
\newglossaryentry{cErf}{name={\ensuremath{\cE_r f}},description={average value of $f$ on the ball},sort=E}
\newglossaryentry{Deltar}{name={\ensuremath{\Delta_r}},description={increment of $\cE_r$}, sort=D}
\newglossaryentry{psi}{name={\ensuremath{\psi_{r,j,z}(x)}},description={wavelets on $X$},sort=p3}

\makeindex

\begin{document}
\title[Association schemes on general measure spaces]
{Association schemes on general measure spaces and zero-dimensional Abelian groups}
\thanks{{\em Date}\/: October 4, 2013, revised May 1, 2014.\/ }%

\author[A. Barg]{Alexander Barg$^\ast$}\thanks{$^\ast$
Institute for Systems
Research, University of Maryland, College Park, MD 20742,
and Institute for Information Transmission Problems,
Russian Academy of Sciences, Moscow, Russia. Email: abarg@umd.edu. Research
supported in part by NSA grant H98230-12-1-0260 and NSF grants DMS1101697,
CCF1217245, and CCF1217894.}
\author[M. Skriganov]{Maxim Skriganov$^\dag$} \thanks{$^\dag$
St.~Petersburg Department of Steklov Institute of Mathematics, Russian Academy of Sciences, nab. Fontanki 27, St.~Petersburg, 191023, Russia. Email: maksim88138813@mail.ru.}

\begin{abstract} Association schemes form one of the main objects of algebraic
combinatorics, classically defined on finite sets. At the same time, direct extensions of this concept to infinite sets encounter
some problems even in the case of countable sets, for instance, countable discrete Abelian groups. 
In an attempt to resolve these difficulties, we define association schemes on arbitrary, possibly uncountable sets
with a measure. We study operator realizations of the adjacency algebras of schemes and derive simple properties of 
these algebras. However, constructing a complete theory in the general case faces a set of obstacles related to the properties
of the adjacency algebras and associated projection operators.
To develop a theory of association schemes, we focus on schemes on topological Abelian groups where we can employ duality theory
and the machinery of harmonic analysis. Using the language of spectrally dual partitions, we prove that such groups support
the construction of general Abelian (translation) schemes and establish properties of their spectral parameters (eigenvalues).

Addressing the existence question of spectrally dual partitions, we show that they arise naturally on topological zero-dimensional Abelian groups, for instance, Cantor-type groups or the groups of $p$-adic numbers. This enables us to construct large
classes of examples of dual pairs of association schemes on zero-dimensional groups with respect to their Haar measure, and to compute their
eigenvalues and intersection numbers (structural constants). 
We also derive properties of infinite metric schemes, connecting them with the properties of the non-Archimedean
metric on the group.

Next we focus on the connection between schemes on zero-dimensional groups and harmonic analysis. We show that the eigenvalues
have a natural interpretation in terms of Littlewood-Paley wavelet bases, and in the (equivalent) language of martingale theory.
For a class of nonmetric schemes constructed in the paper, the eigenvalues coincide with values of orthogonal function systems
on zero-dimensional groups. We observe that these functions, which we call Haar-like bases, have the properties of wavelet
bases on the group, including in some special cases the self-similarity property. This establishes a seemingly new
link between algebraic combinatorics and (non-Archimedean) harmonic analysis.

We conclude the paper by studying some analogs of problems of classical coding theory related to the theory of association schemes.
\end{abstract}
\maketitle
{\small{\tableofcontents}}

\section{Introduction}
\subsection{Motivation of our research} Association schemes form a fundamental object in algebraic combinatorics. 
They were defined in the works of Bose
and his collaborators \cite{Bose52,bos59} and became firmly established after the groundbreaking work of Delsarte \cite{del73a}. Roughly speaking, an association scheme
$\cX$ is a partition of the Cartesian square $X\times X$ of a finite set $X$ into subsets, or classes, whose incidence matrices generate a complex commutative algebra, called the adjacency algebra of the scheme $\cX.$ Properties of the adjacency algebra provide numerous insights into the structure of combinatorial objects related to the set $X$ such as distance regular graphs and error correcting codes, and find applications in other areas of discrete mathematics such as distance geometry, spin models, experimental design, to name a few.
The theory of association schemes is presented from several different perspectives in the books 
by Bannai and Ito \cite{ban84}, Brouwer et al. \cite{bro89}, and Godsil \cite{god93}. A recent survey of the 
theory of commutative association schemes was given by Martin and Tanaka \cite{mar09}. Applications of association schemes in 
coding theory are summarized in the survey of Delsarte and Levenshtein \cite{del98}. The approach to association schemes on 
finite groups via permutation groups and Schur rings is discussed in detail by Evdokimov and Ponomarenko \cite{EvdPon09}. 

Classically, association schemes are defined on finite sets. Is it possible to define them on infinite sets? In the case
of countable discrete sets, such a generalization was developed by Zieschang \cite{zie05}. However, this extension
does not include an important part of the classical theory, namely, duality of association schemes.  
Indeed, our results imply existence of translation-invariant association schemes on Abelian groups that are countable, discrete and periodic. 
Such schemes can of course be described in usual terms \cite{zie05}. However their duals are defined on uncountable, compact and zero-
dimensional groups, so the classical definition of association schemes does not apply: in particular, the intersection numbers of the dual
scheme are not well defined. 
The above discussion shows that the notion of association schemes on infinite sets cannot be restricted
to the case of countable discrete sets. Such a theory would not include the important concept of the dual scheme
and would therefore be incomplete.

Association schemes on arbitrary measure spaces are also important in applications, notably, in harmonic analysis and approximation
theory. In particular, we became interested in these problems while discussing combinatorial aspects of 
papers \cite{skr06,skr11} devoted to the theory of uniformly distributed point sets. 
The connection of association schemes to harmonic analysis is well known in the finite case: in particular, 
there are classes of the so-called $P$- and $Q$-polynomial association schemes, i.e., schemes whose eigenvalues
coincide with the values of classical orthogonal polynomials of a discrete
variable \cite{ban84}. The most well-known example is given by the Hamming scheme for which the
polynomials belong to the family of Krawtchouk polynomials; see \cite{del73a}. Generalizing this link to the infinite
case is another motivation of this work. 
{Of course, these generalizations rely on general methods of harmonic analysis: for instance,
the Littlewood-Paley theory, martingale theory, and the theory of Haar-like wavelets arise naturally while studying
association schemes on measure spaces.}

\subsection{Overview of the paper} In an attempt to give a general definition of the association scheme, 
we start with a measure space (a set equipped with some fixed $\sigma$-additive
measure) and define intersection numbers as measures of the corresponding subsets. 
It is possible to deduce several properties of such association schemes related to their adjacency 
algebras. These algebras are generated by bounded commuting operators whose kernels
are given by the indicator functions of the blocks of $\cX.$ Common eigenspaces of these operators play an important role
in the study of the parameters and properties of the scheme. The set of projectors on the eigenspaces in the finite case
forms another basis of the adjacency algebra, and provides a starting point for the study of duality theory of schemes.
At the same time, in the general case, proving that the projectors are contained in the algebra and computing the
associated spectral parameters of the scheme becomes a difficult problem.

Specializing the class of spaces considered, we focus on the case of translation association 
schemes defined on topological Abelian groups. A scheme defined on a group $X$ is said to have the translation property 
if the partition of $X\times X$ into classes is invariant under the group operation. Translation schemes on Abelian groups 
come in dual pairs that follow the basic duality theory for groups themselves. Formally, the definition of the dual scheme 
in the general case is analogous to the definition for the finite Abelian group, see \cite{del73a,bro89}. At the same time,  unlike the finite case, for infinite topological Abelian groups the structure of the group and the structure of its group of characters can be totally different.
This presents another obstacle in the analysis of the adjacency algebras and their spectral parameters. To overcome it,
we define association schemes in terms of ``spectrally dual partitions" of $X$ and the dual group $\hat X.$ 
Roughly speaking, a spectrally dual partition is a partition of $X$ and $\hat X$ into blocks such that the Fourier 
transform is an isomorphism between the spaces of functions on $X$ and $\hat X$ that are constant on the blocks. 
In the finite case such partitions constitute an equivalent language in the description of translation schemes on Abelian groups \cite{zin96},
\cite{zin09}. We show that in the general case, 
spectrally dual partitions form a sufficient condition for the projectors to be contained in the adjacency algebra.
Using the language of partitions, we develop a theory of translation association schemes in the general case 
of infinite, possibly uncountable topological Abelian groups. Of course, in the finite or countably 
infinite case with the counting measure, our definition coincides with the original definition of the scheme.

While the above discussion motivates the general definitions given in the paper, 
the main question to be answered before moving on is whether this generalization is of interest, i.e.,
whether there are informative examples of generalized association schemes on uncountable sets. 
We prove that spectrally dual partitions do not exist if either $X$ or $\hat X$ is connected.
This observation suggests that one should study totally disconnected (zero-dimensional) Abelian groups. 
Indeed, we construct a large class of examples of translation 
association schemes on topological zero-dimensional Abelian groups with respect to their Haar measure.
These schemes occur in dual pairs, including in some cases self-dual schemes.

In classical theory, many well-known examples of translation schemes, starting with the Hamming scheme, are metric, in the sense that the partitions of the group $X$ are defined by the distance to the identity element. In a similar way, we construct
classes of metric schemes defined by the distance on zero-dimensional groups. The metric on such groups is non-Archimedean,
which gives rise to some interesting general properties of the metric schemes considered in the paper.
We also construct classes of nonmetric translation schemes on zero-dimensional groups and compute their parameters.

One of the important results of classical theory states that a finite association scheme is metric if and only if it is
$p$-polynomial, i.e., if its eigenvalues coincide with values of orthogonal polynomials of a discrete variable; see \cite[Sect.~3.1]{ban84},
\cite[Sect.~2.7]{bro89}.
This result establishes an important link between algebraic combinatorics and harmonic analysis and is the source of a large number of fundamental combinatorial theorems. {In the finite case the metric on $X$ is a graphical distance, which implies that the triangle inequality can be satisfied with equality. This condition can be taken as an equivalent definition of the metric scheme. At the same time, in the non-Arcimedean case this condition is not satisfied
because of the ultrametric property of the distance, and so the schemes are not polynomial (an easy way to see
this is to realize that the coefficients in the three-term relation for the adjacency matrices turn into zeros).} Therefore we are faced with the question of describing the functions whose values coincide with eigenvalues of metric schemes with non-Archimedean distances. We note that even in the finite case this question is rather nontrivial; see, e.g., \cite{mar99,bar09b} for more about this. At the same time, the characterization of metric schemes is of utmost importance for our study because zero-dimensional groups are metrizable precisely by non-Archimedean metrics. 

In order to resolve this question, we note that the chain of nested subgroups of $X$ defines a sequence of increasingly 
refined partitions of the group.
Projection operators on the spaces of functions that are constant on the blocks of a given partition play an important
role in our analysis: namely, we show that eigenvalues of the scheme on $X$ coincide with the values of the kernels of these operators.
This enables an interpretation of the eigenvalues in terms of the Littlewood-Paley theory \cite{edw77}, connecting the
eigenvalues of metric schemes and orthogonal systems known as Littlewood-Paley wavelets \cite[p.115]{Daub92}. We also discuss briefly an interpretation of these results in terms of martingale theory.

Another observation in the context of harmonic analysis on zero-dimensional topological groups relates to
the uncertainty principle. We note that the Fourier transforms of the indicator functions of compact subgroups of $X$ 
are supported on the annihilator subgroups which are compact as well. Developing this observation, we note that
there exist functions on $X$ that ``optimize'' the uncertainty principle, in stark contrast to the Archimedean case.

Perhaps the most interesting result in this part concerns eigenvalues of nonmetric schemes on zero-dimensional groups.
{We observe that the eigenvalues coincide with the values of orthogonal functions on zero-dimensional groups defined in terms of {\em multiresolution analyses}, a basic concept in wavelet theory \cite{Wojta99,NPS11}.}
 We introduce a new class of orthogonal functions on zero-dimensional
groups, calling them Haar-like wavelets. We also isolate a sufficient condition for these wavelets to have self-similarity 
property. While there is a large body of literature on self-similar wavelets on zero-dimensional groups, e.g., \cite{Lang96,Benedetto04,Lukom10}, to the best of our knowledge their connection to algebraic combinatorics so far has not been observed.
{
Concluding this discussion, we would like to stress that the choice of zero-dimensional groups and the associated 
wavelet-like functions is naturally suggested by the logic of our study and is by no means arbitrary. This construction arises naturally as the main example of the abstract theory developed in the paper.}

\vspace*{.05in}{\em Outline of the paper:} We begin with the definition of an association scheme on a general measure space. In Section \ref{sect:BM} we derive simple
properties of the adjacency (Bose-Mesner) algebra of the scheme. Then in Section \ref{sect:Abelian} we consider translation schemes
on topological Abelian groups. Assuming existence of spectrally dual partitions of the
group $X$ and its dual group $\hat X,$ we prove the main results
of duality theory for schemes, including the fact that orthogonal projectors on common eigenspaces are contained
in the adjacency algebra, and perform spectral analysis of the adjacency operators.
The main results of this part of the paper are contained in Section \ref{sect:vil} where we show that
spectrally dual partitions and dual pairs of translation schemes exist for the case of compact and 
locally compact Abelian zero-dimensional groups with the second countability axiom such as the additive
group of $p$-adic integers or groups of the Cantor type (countable direct products of cyclic groups). 
{In Sect.~\ref{sect:metric} we study metric schemes from
the geometric viewpoint and prove that they are nonpolynomial.}
In Sect.~\ref{sect:nonmetric} we construct classes of nonmetric schemes.
The question of characterizing the adjacency algebras of the constructed schemes turns out to be nontrivial. It is addressed
in Section \ref{sect:aa} where we construct these algebras as algebras of functions closed with respect to multiplication and convolution
(Schur rings), addressing both the metric and nonmetric schemes.
Section \ref{sect:eigenvalues} offers several different viewpoints of the eigenvalues of the schemes constructed in the paper 
in the framework of harmonic analysis.
In Section \ref{sect:coding} we consider analogs of some basic results of coding theory related to the theory of association schemes.
To make the paper accessible to a broad mathematical audience, we have included some background information on zero-dimensional Abelian groups; see Sect.~\ref{sect:zero}. 

\vspace*{.05in}{\em Further directions:} Further problems related to the theory developed in this paper include in particular, a general study of infinite association schemes
in terms of Gelfand pairs and spherical functions on homogeneous spaces, an extension of the construction of the paper to 
noncommutative zero-dimensional groups, a study of the connection with (inductive and projective limits) of Schur rings 
outlined in Section~\ref{sect:aa}, and a more detailed
investigation of the new classes of orthogonal bases constructed in the paper.

\vspace*{.05in}{\em Remarks on notation and terminology:}
Throughout the paper we denote by $X$ a second-countable topological space that is endowed with a countably additive measure $\mu.$ A partition 
of $X\times X$ is written as $\cR=\{R_i\}$ where the $R_i$ denote the blocks (classes) of the partition. An association scheme on $X$ defined 
by $\cR$ is denoted by $\cX=\cX(X,\mu,\cR).$ For a subset $D\subset X$ we denote by $\chi[D;x]=\1\{x
\in D\}$ the indicator function of $D$ in $X$, 
and use the notation $\chi_i(x,y):=\chi[R_i; (x,y)]$ as a shorthand for the indicators of the classes. The notation $\naturals_0$ refers
to nonnegative integers. The cardinality of a finite set $Y$ is denoted by $\card(Y)$ or $|Y|.$

Constructing schemes on groups, we consider compact and locally compact Abelian groups. 
When the group $X$ is compact, we explicitly say so, reserving the term
``locally compact'' for noncompact locally compact groups.

\section{Association schemes on measure spaces}\label{sect:scheme}
In this section we define association schemes on an arbitrary set with a measure. For reader's convenience
we begin with the standard definition in the finite case \cite{del73a,ban84,bro89}.

\subsection{The finite case}
\addtocounter{definition}{-1}
\begin{definition}\label{def0}
Let $\I=\{0,1,\dots,d\},$ where
$d$ is some positive integer.
Let $X$ be a finite set and let $\cR=\{R_i\subset X\times X, i\in \I\}$ be a family of disjoint subsets that have the following properties:\\
($i_0$) $R_0=\{(x,x),x\in X\},$\\
($ii_0$) $X\times X=R_0\cup R_1\cup\dots\cup R_d;$ $R_i\cap R_j=\emptyset \text{ if } i\ne j$\\
($iii_0$) $^t R_i=R_{i'},$ where $i'\in \I$ and $^t R_i=\{(y,x)\mid (x,y)\in R_i\}.$\\
($iv_0$) For any $i,j\in \I$ and $x,y\in X$ let
$$
p_{ij}(x,y)=\card\{z\in X: (x,z)\in R_i, (z,y)\in R_j\}.
$$
For any $(x,y)\in R_k$, the quantities $p_{ij}(x,y)=p_{ij}^k$ are constants that depend only on $k.$
Moreover, $p_{ij}^k=p_{ji}^k.$

The configuration $\cX=(X,\cR)$ is called a {\em commutative association scheme}.
\index{association scheme!finite}
The quantities $p_{ij}^k$ are called the intersection numbers, \index{intersection numbers}
and the quantities $\mu_i=p_{ii}^0, i\in \I$ are called the valencies
of the scheme.
If $i=i'$, then $\cX$ is called {\em symmetric.}
\end{definition}
The adjacency matrices $A_i$ of an association scheme are defined by
$$
(A_i)_{xy}=\begin{cases} 1&\text{if }(x,y)\in R_i\\
0&\text{ otherwise}.\end{cases}
$$
The definition of the scheme implies that
\begin{equation}\label{eq:A}
(i) A_0=I,\quad (ii) \sum_{i=1}^d A_i=J, \quad(iii) A_i^T=A_{i'},\quad (iv) A_iA_j=\sum_{k=0}^d p_{ij}^k A_k,
\end{equation}
where $J$ is the all-one matrix.
The matrices $A_i$ form a complex $(d+1)$-dimensional commutative algebra $\mathfrak A(\cX)$ called the {\em adjacency} ({\em Bose-Mesner}) {\em algebra} \cite{bos59}. \index{adjacency algebra}
The space $\complexes^{\card(X)}$ decomposes into $d+1$ common eigenspaces of $\mathfrak A(\cX)$ of multiplicities
$m_i, i\in \I.$
This algebra has a basis of primitive idempotents $\{E_i, i\in \I\}$ given by projections on the eigenspaces.
of the matrices $A_i.$ We have $\rank E_i=m_i, i\in \I.$ The adjacency algebra is closed with respect
to matrix multiplication as well as with respect to the element-wise (Schur, or Hadamard) multiplication $\circ$. We have
\begin{equation}\label{eq:shur}
E_i\circ E_j=\frac 1{\card(X)}\sum_{k\in\I}q_{ij}^k E_k
\end{equation}
where the real numbers $q_{ij}^k$ are called the {\em Krein parameters} of $\cX.$ 
\index{Krein parameters}
If two association schemes have the property
that the intersection numbers of one are the Krein parameters of the other, then the converse is also true. Two such
schemes are called {\em formally dual}. A scheme that is isomorphic to its dual is called {\em self-dual}.
In the important case of schemes on Abelian groups, there is a natural way to construct
dual schemes. This duality will be the subject of a large part of our work.

Finally, since $E_i\in \mathfrak A(\cX)$ for all $i$, we have
\begin{align}
A_i&=\sum_{j\in \I}p_i(j)E_j, \quad i\in\I\label{eq:AE}\\
E_j&=\sum_{i\in\I} q_j(i)A_i \quad j\in \I\label{eq:EA}
\end{align}
(we have changed the normalization slightly from the standard form of these relations). The matrices $P=(p_i(j))$ and
$Q=(q_j(i))$ are called the first and the second {\em eigenvalue matrices} of the scheme. 
\index{eigenvalues}
They satisfy the relations $PQ=QP=I.$

An association scheme $\cX=(X,\cR)$ is called {\em metric} 
\index{association scheme!metric}
if it is possible to define a metric $\rho$ on $X$ so that
any two points $x,y\in X$ satisfy $(x,y)\in R_i$ if and only if $\rho(x,y)=f(i)$ for some strictly monotone function $f.$ Equivalently, $\cX$ is metric if for some ordering of its classes we have $p_{ij}^k\ne 0$ only if $k\le i+j.$ 
Metric schemes have the important property that their eigenvalues $p_i(j)$ are given by (evaluations of) some discrete
orthogonal polynomials; see \cite{del73a,ban84,bro89}. 

An association scheme $\cX=(X,\cR)$ is noncommutative if it satisfies Definition 
\ref{def0} without the condition $p_{ij}^k=p_{ji}^k.$ 
If the definition is further
relaxed so that the diagonal $\{(x,x),x\in X\}$ is a union of some classes $R_i\in\cR,$
then $\cX$ is called a coherent configuration \cite{Higman75}. \index{coherent configuration}

Before moving to the general case of uncountable sets $X$ we comment on the direction of our work. Once we give the definition
of the scheme (Def.~\ref{def1} below), it is relatively easy to construct the corresponding adjacency algebra. The main
problem arises in describing duality, in particular, in finding conditions under which the relations \eqref{eq:AE}
can be inverted to yield relations of the form \eqref{eq:EA}. While a general answer proves elusive, we find classes
of schemes for which this can be accomplished, thereby constructing an analog of the classical theory in the infinite case.

\subsection{The general case}
Let us extend the above definition to infinite, possibly uncountable sets with a measure.

\begin{definition}\label{def1} Let $X$ be an arbitrary set equipped with a $\sigma$-additive measure $\mu$ and let \gls{I} be
finite or countably infinite set. Consider the direct product $X\times X$ with measure $\mu\times\mu.$ Let \gls{cR} be a
collection of measurable sets in $X\times X.$
Assume that the following conditions are true.

\vspace*{.05in}\nd (i) For any $i\in \I$ and any $x\in X$ the set
\begin{equation}\label{eq:set}
\{y\in X: (x,y)\in R_i\}
\end{equation}
is measurable, and its measure is finite. If $\mu(X)<\infty,$ then the last condition can be omitted.

\nd$(ii)$
\begin{equation}\label{eq:R0}
R_0:=\{(x,x): x\in X\}\in \cR
\end{equation}

\nd$(iii)$ The set $\{R_i,i\in \I\}$ forms a partition of $X\times X,$ i.e.,
\begin{equation}\label{eq:Ri}
X\times X=\cup_{i\in \I}R_i, \quad R_i\cap R_j=\emptyset \text{ if } i\ne j
\end{equation}

\nd$(iv)$ $^t R_i=R_{i'},$ where $i'\in \I$ and $^t R_i=\{(y,x)\mid (x,y)\in R_i\}$ is the transpose of $R_i.$

\nd$(v)$ For any $i,j\in \I$ and $x,y\in X$ let
\begin{equation}\label{eq:pijk}
p_{ij}(x,y)=\mu\big(\{z\in X: (x,z)\in R_i, (z,y)\in R_j\}\big).
\end{equation}
For any $(x,y)\in R_k, k\in \I,$ the quantities $p_{ij}(x,y)=p_{ij}^k$ are constants that depend only on $k$.
Moreover, $p_{ij}^k=p_{ji}^k.$

The configuration $\gls{cX}$ is called a {\em commutative association scheme} 
\index{association scheme!general} (or simply a scheme) on
the set $X$ with respect
to the measure $\mu$. The sets $R_i,i \in \I$ are called {\em classes} of the scheme and the nonnegative numbers $p_{ij}^k$
are called {\em intersection numbers} of the scheme. The notion of symmetry is unchanged from the finite case. 
\end{definition}

We will not devote special attention to the 
noncommutative schemes and coherent configurations restricting ourselves to the
remark that the corresponding definitions carry over to the general case without
difficulty. It is also straightforward to define metric schemes, which will be studied in more detail in Sect.~\ref{sect:metric} below.

For the time being it will be convenient not to specialize the index set $\I$, leaving it to be an abstract set.
We note that any scheme $\cX$ can be symmetrized by letting $\tilde\cX$ to be a scheme with
$\tilde R_i=R_i\cup R_{i'},i\in \I.$
The intersection numbers of $\tilde \cX$ can be expressed via the intersection numbers of $\cX$; see \cite[p.57]{ban84}.

Define the numbers
\begin{equation}\label{eq:mu_i}
\mu_i=p_{ii'}^0=\mu\big(\{y\in X:(x,y)\in R_i\}\big), \quad i\in\I.
\end{equation}
These quantities are finite because of condition (i) and do not depend on $x\in X$ because of (v). Call $\mu_i$ the {\em valency}
of the relation $R_i.$ Clearly
\begin{equation}\label{eq:mu}
\gls{mui}=\mu_{i'} \quad\text{and}\quad \sum_{i\in \I}\mu_i=\mu(X).
\end{equation}
The intersection numbers and valencies of finite schemes satisfy a number of well-known relations;
see \cite[Prop.2.2]{ban84} or \cite[Lemma 2.1.1]{bro89}.
All these relations can be established for schemes on sets with a measure without difficulty. In particular,
the following statement, which is analogous to \cite[Prop.2.2(vi)]{ban84}, will be used below in the paper.
\begin{lemma}
\begin{equation}\label{eq:iij}
\mu_k p_{ij}^k=\mu_{i}p_{kj'}^i=\mu_j p_{i'k}^j.
\end{equation}
For symmetric schemes this means that the function
\begin{equation}\label{eq:sigma}
\sigma(i,j,k)=\mu_k p_{ij}^k
\end{equation}
is invariant under permutations of its arguments.
\end{lemma}
\begin{proof} Let us prove the first equality in \eqref{eq:iij}.
Let $x\in X$ and consider the measurable subsets $\cE_{ij}^k=\cE_{ij}^k(x)\subset X\times X$ and
$\cE_k=\cE_k(x)\subset X$:
\begin{align}
\cE_{ij}^k&=\big\{ (z,y)\in X\times X: (x,z)\in R_i, (y,z)\in R_j, (x,y)\in R_k\big\} \label{eq:eijk}\\
\cE_k&=\{y\in X: (x,y)\in R_k\} \label{eq:ek}
\end{align}
and let $\chi[\cE_{ij}^k;\cdot]$ and $\chi[\cE_k;\cdot]$ be their indicator functions.
The definition of the scheme implies that
\begin{align}
\int_X\chi[\cE_{ij}^k;(y,z)]d\mu(z)&=p_{ij}^k\chi[\cE_k;y]\label{eq:a4}\\
\int_X\chi[\cE_{ij}^k;(y,z)]d\mu(y)&=p_{kj'}^i \chi[\cE_i;z]\label{eq:a5}
\end{align}
as well as (see \eqref{eq:mu_i})
\begin{equation}\label{eq:a6}
\int_X\chi[\cE_l;y]d\mu(y)=\mu_l.
\end{equation}
Since the indicator function of the subset $\cE_{ij}^k$ is nonnegative and measurable, we can
use the Fubini theorem to write
\begin{align}
\iint_{X\times X} \chi[\cE_{ij}^k;(y,z)]d\mu(y)d\mu(z)&=\int_Xd\mu(y)\int_X\chi[\cE_{ij}^k;(y,z)]d\mu(z)\nonumber\\
&=\int_Xd\mu(z)\int_X\chi[\cE_{ij}^k;(y,z)]d\mu(y).
\end{align}
Substituting in this equation expressions \eqref{eq:a4} and \eqref{eq:a5} and using \eqref{eq:a6}
with $l=k$ and $l=i$, we obtain the first equality in \eqref{eq:iij}.
The remaining equalities can be proved by a very similar argument or derived from
the first one using commutativity.
\end{proof}

Note the following important difference between schemes on countable and uncountable sets.
Consider the valency $\mu_0$ of the diagonal relation $R_0.$ It equals the measure of a point:
$\mu_0=\mu(\{x\}),$ and is the same for all $x\in X.$ Thus, if $\mu_0>0,$ then $X$ is at most
countably infinite and $\mu(\cdot)=\mu_0 \card\{\cdot\},$ while if $\mu_0=0,$ then $X$ is uncountable 
and the measure $\mu$ is non-atomic.

In accordance with the above we can introduce the following

{\bfit Classification of association schemes $\cX=(X,\mu,\cR):$} \index{association schemes!classification}
\begin{enumerate}
\item[$(S_1)$] $\mu(X)<\infty$ and $\mu_0>0.$ In this case $\cX$ is a classical scheme on the finite set given by Definition \ref{def0}
\cite{Bose52,del73a}. This is the most studied case.
\item[$(S_2)$] $\mu(X)=\infty$ and $\mu_0>0.$ In this case $\cX$ is a scheme on a countable discrete set.
Such schemes are studied in \cite{zie05}.
\item[$(S_3)$] $\mu(X)<\infty$ and $\mu_0=0.$ Examples of such schemes can be constructed on uncountable compact zero-dimensional
Abelian groups, see Sect.~\ref{sect:dualpairs}. Note that their duals are schemes of type $(S_2).$
\item[$(S_4)$] $\mu(X)=\infty$ and $\mu_0=0.$ Schemes of this kind can be constructed on locally compact zero-dimensional Abelian groups. Examples will be considered in Sect.~\ref{sect:dualpairs} below. We note that in this case, similarly to the case $(S_1)$, a scheme can be self-dual.
\end{enumerate}

%
%

\section{Adjacency algebras}\label{sect:BM}
Generalization of the Bose-Mesner algebras to the infinite case is nontrivial. Here we consider only
the main features of such generalized adjacency algebras that follow directly from the definition
of the scheme on an arbitrary measure set. 
For a given scheme $\cX=(X,\mu,\cR)$ consider the indicator functions of its relations:
\begin{equation}\label{eq:bm}
\gls{chii}=\begin{cases} 1&\text{if }(x,y)\in R_i\\
0&\text{otherwise }.\end{cases}
\end{equation}
Definition \ref{def1} immediately implies the following properties of the indicators:

\begin{lemma}\label{lemma:ai} (i) For a fixed $x$, the functions $\chi_i(x,y)$ are measurable and integrable functions of $y;$

(ii) $\chi_0(x,y)=\chi_0(y,x),$ and for each $x$, $\chi_0(x,y)$ is the indicator of a single point $y=x;$

(iii) $\chi_i(x,y)=\chi_{i'}(y,x),$ and if $\cX $is symmetric then $\chi_i(x,y)=\chi_i(y,x);$

(iv) For any $x,y\in X,$
$$
\sum_{i\in \I} \chi_i(x,y)=j(x,y),
$$
where $j(x,y)\equiv 1$ for all $x,y.$

(v) The following equality holds true:
\begin{equation}\label{eq:conv}
\int_X \chi_i(x,z)\chi_j(z,y)d\mu(z)=\sum_{k\in\I}p_{ij}^k \chi_k(x,y).
\end{equation}
In particular, we have
\begin{equation*} 
\int_X \chi_i(x,z)\chi_j(z,y)d\mu(z)=\int_X \chi_j(x,z)\chi_i(z,y)d\mu(z).
\end{equation*}

(vi) The valencies of $\cX$ satisfy the relation
\begin{equation}\label{eq:val}
\mu_i=\int_X \chi_i(x,y)d\mu(y)=\int_X \chi_i(x,y)d\mu(x).
\end{equation}
\end{lemma}
These properties parallel the finite case; see the relations in \eqref{eq:A}.

Consider the set \gls{frakA} of finite linear combinations
\begin{equation}\label{eq:lc}
a(x,y)=\sum_{i\in \I} c_i\chi_i(x,y),
\end{equation}
where $c_i\in \complexes$
for all $i.$ This is a linear space of functions $f:X\times X\to\complexes$ piecewise constant on the classes $R_i,i\in \I$. We have $\dim ({\mathfrak A}(\cX))
=\card(\I).$ Multiplication on this space can be introduced in two ways. Clearly, ${\mathfrak A}(\cX)$ is closed
with respect to the usual product of functions $a(x,y)\cdot b(x,y)$ (because $\chi_i(x,y) \chi_j(x,y)=\delta_{i,j}$).
Define the convolution of functions $a$ and $b$ as follows:
\begin{equation}\label{eq:conv0}
(a\ast b)(x,y)=\int_X a(x,z)b(z,y)d\mu(z).
\end{equation}
By \eqref{eq:conv} and the commutativity condition of $\cX$, convolution is commutative on ${\mathfrak A}(\cX).$
We conclude that linear space ${\mathfrak A}(\cX)$ is a complex commutative algebra with respect to the product of functions.
By \eqref{eq:conv} this algebra is
also closed with respect to convolution. It is called the \emph{adjacency algebra} of the scheme $\cX.$

The question of multiplicative identities of ${\mathfrak A}(\cX)$ with respect to each of the product operations
deserves a separate discussion.

In the classical case $(S_1),$ the adjacency algebra contains units for both operations. For the usual product of functions
(Schur product of matrices) the identity is the function $j(x,y),$ while for convolution (matrix product) the identity is given by $\mu_0^{-1}\chi_0(x,y).$
Clearly, both these functions are contained in ${\mathfrak A}(\cX).$

In the case $(S_2)$ the identity for convolution is given by $\mu_0^{-1}\chi_0(x,y);$ however, the identity for the usual multiplication
$j(x,y)$ is not contained in ${\mathfrak A}(\cX)$ because the convolution $j\ast j$ is not well defined.

In the case $(S_3)$ there is an identity for the usual multiplication, but no identity for convolution (this should be
an atomic measure, viz., the Dirac delta-function $\delta(x,y),$ but the product $\delta\cdot\delta$ cannot be given any meaning).

Finally, in the case $(S_4)$ the algebra ${\mathfrak A}(\cX)$ generally contains neither the usual multiplicative identity nor the
identity for convolution.

{\em Remark:} Note that if an algebra has only one multiplication operation, we can always adjoin to it its identity element.
At the same time, if there is more than one multiplication, it is generally impossible to adjoin several identities in a
coordinated manner.

\subsection{Operator realizations of adjacency algebras}\label{sect:or}
Consider the space $L_2(X,\mu)$ of square-integrable functions on $X$.
Given a measurable function $a(x,y),x,y\in X$, define the integral operator with the kernel $a(x,y):$
\begin{equation}\label{eq:op}
A f(x)=\int_X a(x,y)f(y)d\mu(y).
\end{equation}
Let $\cX=(X,\mu,\cR)$ be an association scheme on $X$. With every class $R_i,i\in\I$ associate an integral
operator with the kernel $\chi_i(x,y)$ \eqref{eq:bm} defined by \eqref{eq:op}:
\begin{equation}\label{eq:Ai}
\gls{Ai} f(x)=\int_X \chi_i(x,y)f(y)d\mu(y).
\end{equation}
Linear combinations \eqref{eq:lc} give rise to operators of the form
\begin{equation}\label{eq:AA}
A=\sum_{i\in\I} c_iA_i.
\end{equation}
Operators of this kind will be used to describe association schemes and their adjacency algebras, therefore
we will devote some space to the study of their basic properties.
\begin{lemma}\label{lemma:bound} For any scheme $\cX=(X,\mu,\cR),$ operators \eqref{eq:op} with 
$a(x,y)\in {\mathfrak A}(\cX)$ are bounded in $L_2(X,\mu).$
\end{lemma}
\begin{proof} According to the Schur test of boundedness, if the kernel $a(x,y)$ of an integral operator
satisfies the conditions
\begin{align*}
\alpha_1&:=\text{ess\,sup}_{x\in X} \int_X|a(x,y)|d\mu(y)<\infty\\
\alpha_2&:=\text{ess\,sup}_{y\in X} \int_X |a(x,y)|d\mu(x)<\infty
\end{align*}
then the operator is bounded in $L_2(X,\mu)$ and its norm $\|A\|\le (\alpha_1\alpha_2)^{\half}$ \cite[p.~22]{Halmos78}.
For the operators $A_i$ we obtain, on account of \eqref{eq:val},
$$
\alpha_1=\int_X \chi_i(x,y)d\mu(y)=\mu_i,
$$
and $\alpha_2=\alpha_1$. We conclude that
\begin{equation}\label{eq:Am}
\|A_i\|\le\mu_i
\quad\text{for all }i.
\end{equation}
Since the sums in \eqref{eq:AA} are finite, the proof is complete.
\end{proof}

In fact, integral operators \eqref{eq:op} with kernels $a(x,y)\in {\mathfrak A}(\cX)$ belong to a special class of operators
called Carleman operators \cite{Halmos78}. Recall that $A$ is called a Carleman operator if
$$
\xi(A,x)=\Big(\int_X |a(x,y)|^2d\mu(y)\Big)^{\half}<\infty
$$
almost everywhere on $X$. For operators in \eqref{eq:AA} we obtain
\begin{align*}
\xi(A,x)&=\Big(\int_X\Big|\sum_{i\in \I} c_i\chi_i(x,y)\Big|^2d\mu(y)\Big)^{\half}
=\Big(\int_X\sum_{i\in\I}|c_i|^2 \chi_i(x,y)d\mu(y)\Big)^{\half}\\
&=\Big(\sum_{i\in\I}|c_i|^2\mu_i\Big)^{\half}
\end{align*}
where the last equality is obtained using \eqref{eq:val}. Thus, for finite sums in \eqref{eq:AA} these functions are finite constants.

If in addition $\mu(X)<\infty,$  then operators \eqref{eq:op} with
$a(x,y)\in {\mathfrak A}(\cX)$ are compact Hilbert-Schmidt.
Indeed, the  
Hilbert-Schmidt norm of $A_i$ is estimated as follows: 
\begin{equation}\label{eq:Fb}
\|A_i\|^2_{HS}=\iint_{X\times X} |\chi_i(x,y)|^2 d\mu(x)d\mu(y)=\int_X d\mu(x)\int_X \chi_i(x,y)d\mu(y)=\mu(X)\mu_i,
\end{equation}
where we have used Fubini's theorem and \eqref{eq:val}.

Let us introduce some notation. Define the set
\begin{equation}\label{eq:I0}
\gls{I0}=\{i\in\I: \mu_i>0\}
\end{equation}
and note that $A_i\ne 0$ only if $i\in \I_0.$ Let $\mu(X)<\infty$ and
let $J$ be an integral operator in $L_2(X,\mu)$ with kernel $j(x,y)\equiv 1$, and let $P$ be the orthogonal
projector on the subspace of constants.
Let us list basic properties of the operators $A_i$.
\begin{lemma}
(i) $A_0=\mu_0 I$, where $I$ is the identity operator in $L_2(X,\mu)$. In particular, for schemes of type
$(S_3)$ and $(S_4),$ $A_0$ is the zero operator.

(ii) \begin{equation}\label{eq:t*}
^tA_i=A_i^\ast=A_{i'}
\end{equation}
where $^tA$ is the transposed operator and $A^\ast$ is the adjoint operator of $A$.

(iii) $A_i A_j=A_j A_i;$ in particular $A_i A_i^\ast=A_i^\ast A_i$. Thus, the operators $A_i$ are normal,
and if the scheme $\cX$ is symmetric, they are self-adjoint.

(iv) Let $\mu(X)<\infty$ and let $P$ be the orthogonal projector on constants. Then
\begin{align}
\sum_{i\in\I_0} A_i&=\gls{J}=\mu(X) P \label{eq:J}\\
A_iA_j&=\sum_{k\in\I_0}p_{ij}^k A_k\label{eq:AiAj}
\end{align}
where both the series converge in the operator norm.
\end{lemma}
\nd{\em Proof: }
Part (i) is immediate from the definitions, Part (ii) follows from Lemma \ref{lemma:ai}(iii), Part (iii) follows
from \eqref{eq:conv0}, and equations \eqref{eq:J} and \eqref{eq:AiAj} are implied by parts (iv) and (v) of Lemma \ref{lemma:ai},
respectively. The convergence of the series in \eqref{eq:J} follows from \eqref{eq:mu} and \eqref{eq:Am}:
$$
\big\|\sum_{i\in\I_0} A_i\big\|\le\sum_{i\in\I_0} \|A_i\|\le \sum_{i\in\I_0} \mu_i=\mu(X)
$$
As for the series in \eqref{eq:AiAj}, Eq. \eqref{eq:pijk} implies that $p_{ij}^k\le\min\{\mu_i,\mu_j\}\le\mu(X),$ and so
$$
\big\|\sum_{k\in\I_0}p_{ij}^k A_k\big\|\le\mu(X)^2.\eqno \qed
$$

{\em Remark:} Let us make an important remark about the definition of the adjacency algebras. Since the adjacency algebra
of the scheme $\cX$ generally is infinite-dimensional, the notion of the basis as well as relations of the form \eqref{eq:Ai}, \eqref{eq:AiAj} 
become rigorous only upon defining a topology on the algebra that supports the needed convergence of the series. So far we have understood
convergence in sense of the operator (Hilbert-Schmidt) norm, but this norm is generally not closed with respect to the Schur product of operators: for instance, in part (iv) of the previous lemma we need the compactness assumption for convergence.
Thus, in general, our arguments in this part are of somewhat heuristic nature. We make them fully rigorous for the case of association
schemes on zero-dimensional groups; see Sect.~\ref{sect:aa}.


\subsection{Spectral decomposition} In the previous subsection we established that the operators $A_i, i\in \I_0$ are bounded in $L_2(X,\mu),$ commuting normal operators
(in the symmetric case, even self-adjoint). By the spectral theorem \cite[p.895]{Dunford63}, they can be simultaneously diagonalized.
 The analysis of spectral decomposition of $\mathfrak A(\cX)$
is simple in the case $\mu(X)<\infty.$ In this case, all the operators $A_i,i\in\I_0$ are compact Hilbert-Schmidt, and the situation resembles the most the classical case of finite sets when the $A_i$s are finite-dimensional
matrices. Namely, if $\mu(X)<\infty,$ then the space $L_2(X,\mu)$ contains a complete orthonormal system of functions $f_j, j\in\I_1$ that are simultaneous
eigenfunctions of all $A_i,i\in \I_0,$ viz.
$$
A_i f_j=\lambda_i(j)f_j;
$$
here $\I_1$ is some set of indices. For every $i\in \I_0$ the nonzero eigenvalues $\lambda_i(j)$ have finite multiplicity. The sequence $\lambda_i(j),
j\in \I_1$ has at most one accumulation point $\lambda=0$. (These two statements follow from 
general spectral theory, e.g. \cite{Dunford63}, Cor. X.4.5.) Moreover,
\begin{align}
\sum_{j\in\I_1}|\lambda_i(j)|^2&=\mu(X)\mu_i, \quad i\in \I_0 \label{eq:l2}\\
\overline{\lambda_i(j)}&=\lambda_{i'}(j).\label{eq:conj}
\end{align}
Indeed, Eq. \eqref{eq:l2} follows from \eqref{eq:Fb}, and Eq. \eqref{eq:conj} is a consequence of \eqref{eq:t*}.

Our next goal is to define the minimal idempotents (cf. \eqref{eq:shur}). Let
\begin{equation}\label{eq:ds}
L_2(X,\mu)=\bigoplus_{j\in \I_2} V_j
\end{equation}
be the expansion of $L_2(X,\mu)$ into an orthogonal direct sum of common eigenspaces of all the operators $A_i,i\in\I_0,$ so that
\begin{equation}\label{eq:max}
A_i f=p_i(j)f \quad\text{for all }f\in V_j
\end{equation}
where $V_j$ is the maximal eigenspace in the sense that for any $V_{j_1},V_{j_2}, j_1\ne j_2$ there exists
an operator $A_i,i\in\I_0$ such that $p_i(j_1)\ne p_i(j_2).$ 
Now let $E_j,j\in \I_2$ be the orthogonal
projectors on the subspaces $V_j.$ Then we can write
\begin{equation}\label{eq:pij}
A_i=\sum_{j\in\I_2}p_i(j)E_j, \quad i\in\I_0,
\end{equation}
where $\gls{pij}$ are the {\em eigenvalues} of the operators $A_i$ on the subspace $V_j$ (cf. Eq.\eqref{eq:AE}).
Call the quantities
\begin{equation}\label{eq:mult}
\gls{mj}=\dim V_j, \quad j\in\I_2
\end{equation}
the {\em multiplicities} of the scheme $\cX.$ \index{association scheme!multiplicities}
In accordance
with \eqref{eq:l2}-\eqref{eq:conj}, taking into account the multiplicities, we have
\begin{align}\label{eq:dim}
\sum_{j\in\I_2}m_j|p_i(j)|^2&=\mu(X)\mu_i\\
\overline{p_i(j)}&=p_{i'}(j).
\end{align}
These relations have their finite analogs; see e.g., \cite[pp.59,63]{ban84}.

The projectors $E_j$ clearly satisfy the relation
\begin{equation}\label{eq:I}
\sum_{j\in\I_2}E_j=I,
\end{equation}
where $I$ is the identity operator in $L_2(X,\mu).$ Recall that for schemes of type $(S_3)$, the operator $I$
is not an integral operator.

We will return to relations \eqref{eq:dim}-\eqref{eq:I} in the next section in the context of duality theory.
This theory is well developed in the finite case, where relations \eqref{eq:pij} can be inverted so that
the projectors are written in terms of the adjacency operators \eqref{eq:EA} \cite[p.60]{ban84}.
These relations and their corollaries form one of the main parts of the classical theory of association schemes.
Unfortunately, in the general case of measure spaces we did not manage to prove the invertibility of relations \eqref{eq:pij}
even in the case $\mu(X)<\infty.$ In the case of $\mu(X)=\infty$ the situation becomes even more complicated
because the operators $A_i,i\in I_0$ can have continuous spectrum. Thus, in the general case it is not known whether
the spectral projectors $E_j,j\in \I_2$ are contained in the adjacency algebra $\mathfrak A(\cX).$

In the next section we show that relations \eqref{eq:pij} can be inverted in the case of schemes on topological
Abelian groups, leading to relations \eqref{eq:EA}. This will enable us to introduce Krein parameters
of schemes and develop a duality theory in the infinite case.

%
%

\section{Association schemes and spectrally dual partitions on topological Abelian groups}\label{sect:Abelian}

\subsection{Harmonic analysis on topological Abelian groups} We begin with reminding the reader the basics about topological Abelian groups.
Details can be found, e.g., in \cite{pon66,hew6370}.

Let $X$ be a second countable topological compact or locally compact Abelian group written additively. 
Let $\hat X$ be the character group of $X$ (the group of continuous characters of $X$) written multiplicatively.
Just as $X$, $\hat X$ is a topological compact or locally
compact Abelian group. By Pontryagin's duality theorem \cite[Thm.~39]{pon66}, its character group $\hat{\hat X}$ is topologically canonically isomorphic to $X.$

For any $x,y\in X$ and $\phi,\psi\in \hat X$ we have
\begin{equation}\label{eq:+.}
\begin{matrix}\phi(x+y)=\phi(x)\phi(y), \quad(\phi\cdot\psi)(x)=\phi(x)\psi(x)\\[.05in]
\overline{\phi(x)}=\phi(-x).\end{matrix}
\end{equation}

Let $\mu$ and $\hat\mu$ be the Haar measures on $X$ and $\hat X$, respectively. 
Define the {\em Fourier transforms} $\cF^\sim:L_2(X,\mu) \to L_2(\hat X,\hat\mu)$ and $\cF^\natural: L_2(\hat X,\hat\mu)\to L_2(X,\mu)$ as follows: \index{Fourier transform}
\begin{align}
&\gls{Fsim}: f(x) \to \tilde f(\xi)=\int_X\xi(x)f(x)d\mu(x), \quad\xi\in \hat X \label{eq:ft}\\
&\gls{Fhat}: g(\xi)\to g^\natural(x)=\int_{\hat X}\overline{\xi(x)}g(\xi)d\hat\mu(\xi), \quad x\in X. \label{eq:ift}
\end{align}
Assume for the moment that $f\in L_2(X,\mu)\cap L_1(X,\mu)$ and similarly, $g\in L_2(\hat X,\hat\mu)\cap L_1(\hat X,\hat\mu).$
It is known that the Haar measures can be normalized to fulfill the Parseval identities and equalities for the inner products
    \begin{align}
       \int_X |f(x)|^2d\mu(x)&=\int_{\hat X} |\tilde f(\xi)|^2 d\hat\mu(\xi), \quad\;\; 
          \int_{\hat X} |g(\xi)|^2 d\hat\mu(\xi)=\int_X |g^\natural(x)|^2 d\mu(x)
    \label{eq:Parseval}\\[.1in]
     &\begin{array}{c}{\displaystyle \int_X f_1(x)\overline{f_2(x)}d\mu(x)=\int_{\hat X} \tilde f_1(\xi) \overline{\tilde f_2(\xi)}d\hat\mu(\xi)}
     \\ {\displaystyle \int_{\hat X} g_1(\xi)\overline{g_2(\xi)}d\hat\mu(\xi)=\int_X g_1^\natural(x)\overline{g_2^\natural(x)}d\mu(x)}
     \end{array}
    \label{eq:ip}     
    \end{align}
In what follows we always assume that equalities \eqref{eq:Parseval}, \eqref{eq:ip} are fulfilled. The corresponding normalizations
of the Haar measures $\mu$ and $\hat\mu$ will be given below.
Equalities \eqref{eq:Parseval}, \eqref{eq:ip} imply that the mappings $\cF^\sim$ and $\cF^\natural$ can be extended to mutually inverse isometries of
the Hilbert spaces $L_2(X,\mu)$ and $L_2(\hat X,\hat\mu)$ that preserve inner products:
$$
\cF^\sim\cF^\natural=I,\quad \cF^\natural\cF^\sim=\hat I,
$$
where $I$ and $\hat I$ are the identity operators in $L_2(X,\mu)$ and $L_2(\hat X,\hat\mu),$ respectively, and
so $\cF^\natural=(\cF^\sim)^{-1}.$

Recall the formulas for convolutions and their Fourier transforms. Let
\begin{align}
(f_1\ast f_2)(x)=\int_X f_1(x-y)f_2(y)d\mu(y), \quad x\in X\label{eq:conv1}\\
(g_1\ast g_2)(\phi)=\int_{\hat X} g_1(\phi\xi^{-1})g_2(\xi)d\hat\mu(\xi)\quad \phi\in\hat X\label{eq:conv2}
\end{align}
According to the Young inequality \cite[p.157]{Edwards82}, for any functions $f_1\in L_p, f_2\in L_q$
  \begin{equation}\label{eq:Young}
\|f_1\ast f_2\|_r\le \|f_1\|_p\|f_2\|_q,
   \end{equation}
where $p,q\in[1,\infty]$ and $\frac 1r=\frac 1p+\frac 1q-1\ge0.$ Thus, the convolutions \eqref{eq:conv1},\,\eqref{eq:conv2}
are in $L_2$ if one of the functions is in $L_2$ and the other in $L_1.$ We have
    \begin{equation} \begin{array}{c}
(\tilde{f_1\ast f_2})(\xi)=\tilde f_1(\xi)\tilde f_2(\xi), \quad \xi\in\hat X \\
(g_1\ast g_2)^\natural(x)=g_1^\natural(x)g_2^\natural(x), \quad x\in X.\end{array}\label{eq:Fc}
    \end{equation}
These formulas are useful for spectral analysis of integral operators that commute with translations on the groups $X$ and $\hat X.$
Consider the integral operators on the spaces $L_2(X,\mu)$ and $L_2(\hat X,\hat\mu)$ given by
\begin{align*}
Af(x)&=\int_X a(x-y)f(y)d\mu(y)\\
Bg(\xi)&=\int_{\hat X} b(\xi\eta^{-1})g(\eta)d\hat\mu(\eta),
\end{align*}
where $a(x), x\in X$ and $b(\xi),\xi\in \hat X$ are $L_1$ kernels. We have
\begin{align*}
\tilde{Af}(\xi)&=\tilde a(\xi)\tilde f(\xi)\\
(Bg)^\natural(x)&=b^\natural(x)g^\natural(x),
\end{align*}
which shows that the operators $\cF^\sim A\cF^\natural$ and $\cF^\natural B\cF^\sim$ are diagonal (i.e., are multiplication operators).
Their spectra are given by the values of the functions $\tilde a(\xi),\xi\in\hat X$ and $b^\natural(x), x\in X.$ In particular,
if $X$ is an infinite compact group, then the spectrum of $A$ is discrete and the spectrum of $B$ of continuous. If $X$ is locally
compact, then both the spectra of $A$ and $B$ are continuous. This happens, for instance, when the graph of the function $b^\natural(x),
x\in X$ has constant segments supported on sets of positive measure, which corresponds to the continuous spectrum in spectral theory. We will encounter this situation later in the paper.

\remove{{\em Remark on notation:} The symmetry of Fourier transforms \eqref{eq:ft}-\eqref{eq:ift} is more apparent if the dual group $\hat X$ is written additively, similarly to $X.$ Then one can define a bilinear map $\langle\phi,x\rangle: X\times \hat X\to \reals/\integers$ with the natural
properties
   \begin{align*}
      \langle\phi,x_1+x_2\rangle=\langle \phi,x_1\rangle+\langle\phi,x_2,\rangle, \quad \langle\phi,-x\rangle=-\langle\phi,x\rangle\\
      \langle\phi_1+\phi_2,x\rangle=\langle \phi_1,x\rangle+\langle\phi_2,x,\rangle, \quad \langle-\phi,x\rangle=-\langle\phi,x\rangle
  \end{align*}
Then for a fixed $\phi\in\hat X,$ the function $\langle\phi,x\rangle$ is a character on $X$;  for a fixed $x\in X$, the function
$\langle\phi,x\rangle$ is a character on $\hat X$, and all the characters of $X$ and $\hat X$ are of this form.
Define the pairing $\{\phi,x\}:X\times \hat X\to \complexes^\ast$ by $\{\phi,x\}=\exp(2\pi\sqrt{-1}\langle\phi,x\rangle),$ then the Fourier
transforms \eqref{eq:ft}-\eqref{eq:ift} are written as follows
  \begin{align*}
  \tilde f(\xi)= \int_X\{\xi,x\}f(x)d\mu(x), \quad g^{\sharp}(x)=\int_{\hat X}\overline{\{\xi,x\}}g(\xi) d\hat\mu(\xi)
  \end{align*}
Even though this notation makes some formulas of this work to look more natural and also more symmetric, we have opted
for the multiplicative notation which is more standard in the current combinatorial literature.}


\subsection{Translation association schemes}
\begin{definition}\label{def2} Let $X$ be an Abelian group. A scheme $\cX(X,\mu,\cR)$ is called \emph{translation invariant} if for every $R_i,i\in \I$ and every \index{association scheme!translation {\em or} Abelian}
pair $(x,y)\in R_i,$ the pair $(x+z,y+z)\in R_i$ for all $z\in X.$ If the group $\hat X$ is written multiplicatively, then
the scheme $\hat X(\hat X,\hat\mu,\hat\cR)$ is called translation invariant if for every relation $R_i,i\in \hat\I$ and
every pair $(\phi,\psi)\in \hat R_i,$ the pair $(\phi\xi,\psi\xi)\in \hat R_i$ for all $\xi\in \hat X.$
\end{definition}
This definition is the same as in the classical case \cite{bro89}. Following it, we note that the partition $\{R_i,i\in \I\}$
of $X\times X$ can be replaced with the partition $\{N_i,i\in \I\}$ of the set $X$ itself: indeed,
letting
\begin{equation}
\gls{Ni}:=\{x\in X:(x,0)\in R_i\}, \quad i\in \I,
\end{equation}
we note that $(x,y)\in R_i$ if and only if $x-y\in N_i.$ The indicator function of the set $R_i\subset X\times X$ has the
form
\begin{equation}\label{eq:sim}
\chi_i(x,y)=\chi_i(x-y)
\end{equation}
where $\chi_i(x)$ is the indicator function of the set $N_i\subset X.$ The adjacency operators \eqref{eq:Ai} have the form
\begin{equation}\label{eq:Aii}
A_i f(x)=\int_X \chi_i(x-y)f(y)d\mu(y).
\end{equation}
Similarly, let
\begin{equation}\label{eq:hatNi}
\hat N_i:=\{\phi\in \hat X: (\phi,1)\in \hat R_i\}, \quad i\in \hat\I,
\end{equation}
where $1$ is the unit character. We obtain that $(\phi,\psi)\in\hat R_i$ if and only if $\phi\psi^{-1}\in \hat N_i,$ and the
indicator function of the relation $\hat R_i\in \hat X\times\hat X$ has the form
$$
\hat \chi_i(\phi,\psi)=\hat \chi_i(\phi\psi^{-1}),
$$
where $\hat \chi_i(\phi)$ is the indicator function of the set $\hat N_i\subset \hat X.$ The corresponding adjacency operators
have the form
\begin{equation}\label{eq:hatAi}
\hat A_ig(\phi)=\int_{\hat X} \hat \chi_i(\phi\psi^{-1})g(\psi)d\mu(\psi).
\end{equation}
We also easily obtain the following expressions for the valencies of $\cX$ and $\widehat\cX:$
\begin{align*}
\mu_i=\int_X \chi_i(x)d\mu(x),\quad i\in \I\\
\hat\mu_i=\int_{\hat X}\hat \chi_i(\phi)d\hat\mu(\phi), \quad i\in \hat \I.
\end{align*}
The classical counterparts of these expressions are found in \cite[Sect.2.10]{bro89}.

The general approach to infinite schemes developed above applies to translation schemes defined on infinite Abelian groups. However, following this
path, we would again face the question of inverting relations \eqref{eq:pij} and developing the duality theory.
For general translation-invariant schemes we could not find conditions that would enable us to express spectral
projectors $E_j$ via the adjacency
operators \eqref{eq:Ai}. Instead, we will formulate a sufficient condition for invertibility in terms of the partitions $\{N_j\}$ and
$\{\hat N_i\}.$ This condition will enable us to obtain a rather complete duality theory of translation invariant schemes on the
groups $X$ and $\hat X.$

\subsubsection{Spectrally dual partitions} Let $\cN=\{N_i,i\in\I\}$ and $\widehat\cN=\{\hat N_i, i\in \hat\I\}$ be finite or
countable partitions of the groups $X$ and $\hat X.$ Assume that the blocks of these partitions are measurable with respect to the Haar
measures $\mu$ and $\hat\mu,$ and their measure is finite. 
For every $N_i\in\cN$ and every $\hat N_j\in\widehat\cN$ define
\begin{equation}\label{eq:N}
N_{i'}=\{-x: x\in N_i\}, \quad \hat N_{j'}=\{\phi^{-1}: \phi\in \hat N_j\}.
\end{equation}
If $i'=i$ and $j'=j,$ we call such partitions {\em symmetric}.
\index{group!partition, symmetric}
 We also assume that $N_0=\{0\}\in\cN$ and $\hat N_0=\{1\}\in\widehat\cN$
(here 1 is the trivial, or unit character), and that only these subsets can have measure zero, i.e.,
$$
\mu(N_i)>0 \;\text{for }i\ne 0\quad\text{and}\quad \hat\mu(\hat N_i)>0\;\text{for }i\ne 0.
$$
Introduce the following notation (cf. \eqref{eq:I0}):
\begin{align*}
\gls{I0}&=\{i\in \I: \mu(N_i)>0\}=\begin{cases} \I\backslash \{0\} &\text{if }\mu(N_0)=0\\
\I&\text{if }\mu(N_0)>0
\end{cases}\\
\hat \I_0&=\{i\in \hat\I: \hat\mu(\hat N_i)>0\}=\begin{cases} \hat\I\backslash \{0\} &\text{if }\hat\mu(\hat N_0)=0\\
\hat\I&\text{if }\hat\mu(\hat N_0)>0.
\end{cases}
\end{align*}
The partitions $\cN$ and $\widehat\cN$ give rise to partitions of the sets $X\times X$ and $\hat X\times \hat X$ given by
\begin{align}
\cR=\{R_i, i\in\I\}, \quad R_i=\{(x,y)\in X\times X: x-y\in N_i\}\label{eq:NR}\\
\hat\cR=\{\hat R_i: i\in\hat I\}, \quad \hat R_i=\{(\phi,\psi)\in \hat X\times\hat X: \phi\psi^{-1}\in\hat N_i\}.\label{eq:hNR}
\end{align}
Denote by $\chi_i(x)$ and $\hat \chi_i(\xi)$ the indicator functions of the subsets $N_i, i\in \I$ and $\hat N_i, i\in\hat \I,$
respectively. Define the space $\Lambda_2(\cN)\subset L_2(X,\mu)$ of functions piecewise constant on the partition $\cN.$
In other words, a function $f:X\to\complexes$ is contained in $\Lambda_2(\cN)$ if and only if
\begin{equation}\label{eq:fi}
f(x)= \sum_{i\in\I_0}f_i \chi_i(x)
\end{equation}
and
$$
\int_X|f(x)|^2d\mu(x)=\sum_{i\in\I_0}|f_i|^2\mu(N_i)<\infty
$$
where $f_i$ is the value of $f$ on $N_i.$ Note that the sum in \eqref{eq:fi} contains just one nonzero term, so the issue of convergence
does not arise.
In a similar way, let us introduce the space $\Lambda_2(\widehat\cN)\subset L_2(\hat X,\hat\mu)$ of functions
piecewise constant on the partition $\widehat\cN$, letting $g\in\Lambda_2(\widehat\cN)$ if and only if
\begin{equation}\label{eq:if}
g(\xi)= \sum_{i\in\hat\I_0} g_i\hat \chi_i(\xi)
\end{equation}
and
$$
\int_{\hat X}|g(\xi)|^2d\hat\mu(\xi)=\sum_{i\in\hat\I_0}|g_i|^2\hat\mu(\hat N_i)<\infty.
$$
Obviously, relations \eqref{eq:fi} and \eqref{eq:if} hold pointwise because the sums in these expressions involve 
indicator functions of disjoint sets.

\begin{definition}\label{def:sdp} Let $\cN=\{N_i,i\in\I\}$ and $\widehat\cN=\{\hat N_i, i\in \hat\I\}$ be partitions of mutually
dual topological Abelian groups $X$ and $\hat X,$ respectively. The partitions $\cN$ and $\widehat\cN$ are called
\emph{spectrally dual} \index{group!partitions, spectrally dual}
if the Fourier transform $\cF^\sim$ is an isomorphism of the subspaces $\Lambda_2(\cN)$ and $\Lambda_2(\widehat\cN)$:
$$
\cF^\sim \Lambda_2(\cN)=\Lambda_2(\widehat \cN), \quad \Lambda_2(\cN)=\cF^\natural\Lambda_2(\widehat \cN).
$$
In other words, the Fourier transform \eqref{eq:ft} of any function of the form \eqref{eq:fi} is a function of
the form \eqref{eq:if} and conversely, the Fourier transform \eqref{eq:ift} of any function of the form \eqref{eq:if}
is a function of the form \eqref{eq:fi}.
\end{definition}

We remark that in the finite case, spectrally dual partitions were introduced by Zinoviev and Ericson \cite{zin96}
(these papers used the term {\em Fourier-invariant}, while Gluesing-Luerssen in a recent work \cite{hgl12}, calls them {\em Fourier-reflexive} partitions).
In particular,  \cite{zin96} showed that such partitions can be used to prove a Poisson summation formula
for subgroups of finite Abelian groups. Forney \cite{Forney98} discussed a generalization of this result for
locally compact Abelian groups, however he made no attempt to study spectrally dual partitions and association schemes
in the infinite case.
Later, Zinoviev and Ericson showed \cite{zin09} that the existence of spectrally dual partitions of a finite Abelian group and its dual group
is equivalent to the existence of a pair of dual translation schemes on these groups.
As a consequence, the existence of such partition forms a necessary and sufficient condition for the existence of a dual pair of
relations \eqref{eq:AE}-\eqref{eq:EA} between the bases of the Bose-Mesner algebra.

While in the finite case, the duality theory of translation schemes can be derived independently of the language of spectrally dual
partitions, we find this language quite useful for infinite schemes. In the remainder of this section,
we show that the existence of such partitions is sufficient for the invertibility of
the relation \eqref{eq:pij} in the general case of infinite groups.

\vspace*{.1in}We begin with deriving some results implied by the definition of spectrally dual partitions based on the fact that the Fourier transform
is an isometry that preserves the inner products; see \eqref{eq:ip}, \eqref{eq:Parseval}.
Let $\tilde \chi_i(x)$ and $\hat \chi_i^\natural(x)$ be the Fourier transforms of the indicator functions of the blocks:
\begin{align}
\tilde \chi_i(\phi)&=\int_X \chi_i(x)\phi(x)d\mu(x)=\int_{N_i}\phi(x)d\mu(x) \label{eq:fa}\\
\hat \chi_i^\natural(x)&=\int_{\hat X}\hat \chi_i(x)\overline{\phi(x)}d\hat\mu(\phi)=\int_{\hat N_i}\overline{\phi(x)}d\hat\mu(\phi).\label{eq:ifa}
\end{align}
By definition, we obtain
\begin{align}
\tilde \chi_i(\phi)&\simeq \sum_{k \in \hat\I_0} p_i(k)\hat \chi_k(\phi), \quad i\in \I_0 \label{eq:pik1}\\
\hat \chi_i^\natural(x)&\simeq \sum_{k\in\I_0} q_i(k) \chi_k(x), \quad i\in\hat \I_0,\label{eq:qik1}
\end{align}
where $p_i(k)$ and $q_i(k)$ are some complex coefficients.
Define the matrices
\begin{equation}\label{eig}
\gls{P}=(p_i(j)), i\in \hat \I_0, j\in\I_0; \quad \gls{Q}=(q_i(j)), i\in \I_0, j\in\hat \I_0
\end{equation}
Here we use notation $q_i(j)$ instead of more logical $\hat p_i(j)$ to conform with the classical case.
Note that by the definition of spectrally dual partitions, the matrices $P$ and $Q$ have equal ``dimensions," i.e.,
the sets $\I_0$ and $\hat\I_0$ are equicardinal. 
%
%
\begin{lemma} The coefficients $p_i(k)$ and $q_i(k)$ satisfy the relations
\begin{align}
\sum_{k\in \hat \I_0} p_i(k)\overline {p_j(k)}\hat\mu(\hat N_k)&=\delta_{ij}\mu(N_i)\label{eq:pipj}\\
\sum_{k\in \I_0} q_i(k)\overline{q_j(k)} \mu(N_k)&=\delta_{ij}\hat\mu(\hat N_i)\label{eq:qiqj}\\
p_i(j)\hat \mu(\hat N_j)&=\overline{q_j(i)}\mu(N_i)\label{eq:mv}\\
\sum_{k\in\I_0} \frac 1{\mu(N_k)} p_k(i)\overline{p_k(j)}&=\delta_{ij}\frac 1{\hat\mu(\hat N_i)}\label{eq:o2}\\
\sum_{k\in\hat\I_0}\frac 1{\hat\mu(\hat N_k)} q_k(i)\overline{q_k(j)}&=\delta_{ij}\frac 1{\mu(N_j)}.\label{eq:o3}
\end{align}
The matrices $P$ and $Q$ satisfy
\begin{equation}\label{eq:PQ}
PQ=\hat I, \;QP=I
\end{equation}
where $I$ and $\hat I$ are the identity operators in the spaces of sequences indexed by $\I$ and $\hat \I$,
respectively.
\end{lemma}
{\em Remark:} The finite analogs of these relations are well known in the classical theory; see \cite{ban84}, Theorem 2.3.5 or \cite{bro89}, Lemma 2.2.1(iv).

\begin{proof} To prove \eqref{eq:pipj}, observe that the isometry property of the Fourier transform \eqref{eq:ip} implies that
\begin{align*}
\int_{\hat X} \tilde \chi_i(\phi)\overline{\tilde \chi_j(\phi)}d\hat\mu(\phi)&=\int_X \chi_i(x)\overline{\chi_j(x)}d\mu(x)
=\delta_{ij}\int_X \chi_i(x)d\mu(x)\\
&=\delta_{ij}\mu(N_i).
\end{align*}
Substituting \eqref{eq:pik1} and using the fact that $\hat \chi_k$ is
$\{0,1\}$-valued and that $\hat \chi_k\hat \chi_{k'}=0$ if $k\ne k',$ we obtain
\begin{align*}
\int_{\hat X}\tilde \chi_i(\phi)\overline{\tilde \chi_j(\phi)}d\hat\mu(\phi)&=
\sum_{k\in\hat\I_0}p_i(k)\overline{p_i(k)}\int_{\hat X}\hat \chi_k(\phi)d\hat\mu(\phi)\\
&=\sum_{k\in\hat\I_0}p_i(k)\overline{p_j(k)}
\hat\mu(\hat N_k).
\end{align*}
The proof of \eqref{eq:qiqj} is completely analogous and will be omitted.
To prove \eqref{eq:mv}, multiply both sides of \eqref{eq:pik1} by $\hat \chi_i(\phi)$ and integrate on $\phi$. We obtain
$$
\int_{\hat X} \hat \chi_j(\phi)\tilde \chi_i(\phi)d\hat \mu(\phi)=\sum_{k\in \hat \I_0} p_i(k)
\int_{\hat X}\hat \chi_j(\phi) \hat \chi_k(\phi)d\hat\mu(\phi)=p_i(j)\hat \mu(\hat N_j).
$$
On the other hand, using \eqref{eq:ip} and \eqref{eq:qik1}, we obtain
\begin{align*}
\int_{\hat X} \hat \chi_j(\phi)\tilde \chi_i(\phi)d\hat \mu(\phi)&=\int_{X} \overline{\hat \chi_j^\natural(x)} \chi_i(x)d\mu(x)
=\sum_{k\in\I_0}\overline{q_j(k)}\int_X \chi_i(x)\chi_k(x)d\mu(x)\\
&=\overline{q_j(i)}\mu(N_i).
\end{align*}
The last two relations imply \eqref{eq:mv}. Relations \eqref{eq:o2} and \eqref{eq:o3} follow from \eqref{eq:mv} and
\eqref{eq:qiqj}-\eqref{eq:pipj}.
Finally, \eqref{eq:PQ} follows directly from \eqref{eq:pik1}, \eqref{eq:qik1}.
\end{proof}

We can also consider the matrices
\begin{align*}
T&=\bigg[\Big(\frac{\hat\mu(\hat N_i)}{\mu(N_i)}\Big)^{\half} p_i(j)\bigg], i\in \I_0, j\in\hat \I_0;
& U&=\bigg[\Big(\frac{\mu( N_i)}{\hat\mu(\hat N_i)}\Big)^{\half} q_i(j)\bigg], i\in \hat \I_0, j\in \I_0,\\
\end{align*}
then \eqref{eq:pipj}, \eqref{eq:qiqj}, and \eqref{eq:mv} are equivalently written as
\begin{equation}\label{eq:TU}
T T^\ast=I, \quad UU^\ast=\hat I, \quad U=T^\ast.
\end{equation}
All these relations express in different ways the fact (implied by the definition of spectrally dual partitions)
that the Fourier transform is an isometric isomorphism of the subspaces $\Lambda_2(\cN)$ and $\Lambda_2(\widehat\cN)$
of functions constant on the blocks of the partitions.

\vspace*{.2in} So far it sufficed to assume that relations \eqref{eq:pik1}, \eqref{eq:qik1} hold as equalities
of functions in $L_2(X,\mu)$ and $L_2(\hat X,\hat\mu),$ respectively.
With a minor adjustment, relations \eqref{eq:pik1} and \eqref{eq:qik1} can be shown to hold pointwise.
\begin{lemma}\label{lemma:pointwise} Let $p_i(0)=\mu(N_i)$ and $q_i(0)=\hat\mu(\hat N_i)$. Then
\begin{align}
\tilde \chi_i(\phi)&=\sum_{k\in \I} p_i(k)\hat \chi_i(\phi) \quad \text{for all } \phi\in \hat X \label{eq:pw1}\\
\hat \chi_i^\natural(x)&=\sum_{k\in \I} q_i(k)\chi_k(x) \quad \text{for all } x\in X.\label{eq:pw2}
\end{align}
\end{lemma}
\begin{proof}
For instance, let us prove \eqref{eq:pw1}. Since the function $\tilde \chi_i(\phi)$ is piecewise constant, the two sides of
\eqref{eq:pik1} can be different only on a set of measure 0. At the same time, by our assumption, $\hat\mu(\hat N_i)>0$ for $i\ne 0,$ so
for $\phi\ne 1$, Eq.~\eqref{eq:pik1} holds pointwise on $\hat X.$ Thus, the two sides of this equation can be different on the set $\hat N_0=\{1\}$ if
and only if $\hat\mu(\hat N_0)>0.$ At the same time, $\tilde \chi_i(1)=\mu(N_i)$ by \eqref{eq:ft}, so our definition of
$p_i(0)$ ensures that \eqref{eq:pw1} holds for all $\phi\in \hat X.$ The proof of \eqref{eq:pw2} follows by the same
arguments.
\end{proof}


\subsubsection{Spectral decomposition}
Recall that we defined adjacency operators of the schemes $\cX$ and $\widehat\cX$ in \eqref{eq:Aii} and \eqref{eq:hatAi}, with
kernels $\chi_i(x-y)$ and $\hat \chi_i(\phi\psi^{-1}),$ respectively.
Following our plan of developing a duality theory, let us also introduce orthogonal projectors (cf. \eqref{eq:pij}).
\index{orthogonal projector}
 Apply
the Fourier transform \eqref{eq:ift} on both sides of \eqref{eq:pik1} and the Fourier transform \eqref{eq:ft} on both
sides of \eqref{eq:qik1}. We obtain
\begin{align}
\chi_i(x)&\simeq \sum_{k\in\hat\I_0} p_i(k)\hat \chi_k^\natural(x), \quad i\in \I_0\label{eq:pik2}\\
\hat \chi_i(\phi)&\simeq \sum_{k\in\I_0} q_i(k)\tilde \chi_k(\phi), \quad i\in\hat\I_0. \label{eq:qik2}
\end{align}
Define the operator $E_k$ with the kernel $\hat \chi_k^\natural(x-y), k\in \hat\I_0$:
\begin{equation}\label{eq:Ek}
\gls{Ek} f(x)=\int_X \hat \chi_k^\natural(x-y)f(y)d\mu(y)
\end{equation}
and the operator $\hat E_k$ with the kernel $\tilde \chi_k(\phi\psi^{-1}), k\in\I_0:$
\begin{equation}\label{eq:hatEk}
\hat E_k g(\phi)=\int_{\hat X} \tilde \chi_k(\phi\psi^{-1})g(\psi)d\hat\mu(\psi).
\end{equation}
Then relations \eqref{eq:qik1} and \eqref{eq:pik2} can be expressed as the following operator relations in $L_2(X,\mu)$:
\begin{align}
E_i&=\sum_{k\in\I_0} q_i(k) A_k, \quad i\in\hat \I_0\label{eq:Ei2}\\
A_j&=\sum_{k\in\hat\I_0} p_j(k)E_k, \quad j\in\I_0.   \label{eq:Aj2}
\end{align}
Likewise, relations \eqref{eq:pik1} and \eqref{eq:qik2} can be written as operator equalities in $L_2(\hat X,\hat\mu)$ as follows:
\begin{align}
\hat E_i=\sum_{k\in\hat\I_0}p_i(k)\hat A_k, \quad i\in \I_0\label{eq:Ei3}\\
\hat A_j=\sum_{k\in\I_0}q_j(k)\hat E_k, \quad j\in \hat\I_0.\label{eq:Aj3}
\end{align}
The pairs \eqref{eq:Ei2}-\eqref{eq:Aj2} and \eqref{eq:Ei3}-\eqref{eq:Aj3} are mutually inverse; cf \eqref{eq:PQ}, \eqref{eq:TU} and also
\eqref{eq:AE}-\eqref{eq:EA}.
\begin{lemma}\label{lemma:idem}
The operators from the families $\{E_k, k\in\I_0\}$ and $\{\hat E_k, k\in\I_0\}$ are self-adjoint commuting orthogonal projectors in their respective
$L_2$ spaces.
Furthermore,
\begin{equation}\label{eq:spectral}
\left.
\begin{array}{c}
V_k:=E_k L_2(X,\mu)=\cF^\natural L_2(\hat N_k,\hat\mu)\\
\hat V_k:=\hat E_k L_2(\hat X,\hat \mu)=\cF^\sim L_2(N_k,\mu)\\
L_2(X,\mu)=\bigoplus_{k\in\hat \I_0} V_k\\
L_2(\hat X,\hat\mu)=\bigoplus_{k\in\I_0} \hat V_k
\end{array}\right.
\end{equation}
\end{lemma}
\begin{proof} Using \eqref{eq:+.},
the kernels of the operators $E_k$ and $\hat E_k$ can be written as
\begin{align}
\hat \chi_k^\natural(x-y)=\int_{\hat N_k} \overline{\phi(x)}\phi(y)d\hat\mu(\phi)  \label{eq:Ei4}\\
\tilde \chi_k(\phi\psi^{-1})=\int_{N_k}\phi(x)\overline{\psi(x)}d\mu(x)\label{eq:^Ei4}
\end{align}
implying that
\begin{equation}\label{eq:ad}
\hat \chi_k^\natural(x-y)=\overline{\hat \chi_k^\natural(y-x)}, \quad \tilde \chi_k(\phi\psi^{-1})=
\overline{\tilde \chi_k(\psi\phi^{-1})},
\end{equation}
i.e.,
$$
E_k=E_k^\ast, k\in \hat\I_0,\quad \hat E_k=\hat E_k^\ast, k\in \I_0.
$$
Using the isometry conditions \eqref{eq:ip}, we obtain the relations
\begin{align*}
\int_X \hat \chi_k^\natural(x-z)\hat \chi_l^\natural(z-y)d\mu(z)&=\delta_{kl}\hat \chi_k^\natural( x-y)\\
\int_{\hat X} \tilde \chi_k(\phi \xi^{-1})\tilde \chi_l(\xi\psi^{-1})d\hat \mu(\xi)&=\delta_{kl}\tilde \chi_k(\phi\psi^{-1}),
\end{align*}
which in the operator form are expressed as
$$
E_k E_l=\delta_{kl}E_k, \quad \hat E_k\hat E_l=\delta_{kl}\hat E_k.
$$
Now the claims in \eqref{eq:spectral} are immediate.
\end{proof}

\vspace*{.1in}{\em Remark:} Formulas \eqref{eq:Ei4} and \eqref{eq:^Ei4} generalize the following well-known expressions for the idempotents of finite translation schemes \cite[Eq.(2.21)]{bro89}:
   $$
   E_k=\frac 1{\card(X)}\sum_{\phi\in \hat N_k} \phi\phi^\dag,\;k\in\hat\I; \qquad \hat E_k=\frac 1{\card(\hat X)}\sum_{x\in N_k} \bfx \bfx^\dag,\;k\in\I,
  $$
where $\phi=\{\phi(x), x\in X\}$ and $\bfx=\{\hat{\hat x}(\phi), \phi\in\hat X\}$ (here the coordinates of the vector $\bfx$ are the values
of the character $\hat{\hat x}\in \hat{\hat X}$, viz., $\hat{\hat x}(\phi)=\phi(x), x\in X$).

As a conclusion, the spectral decomposition of families of commuting normal operators $\{A_i, i\in\hat\I_0\}$ and
$\{\hat A_i, i\in\I_0\}$ is given by Equations \eqref{eq:Aj2} and \eqref{eq:Aj3}. The coefficients $\{p_j(k), j\in \I_0,k\in \hat\I_0\}$ and $\{q_j(k), j\in \hat\I_0, k\in\I_0\}$ give the
eigenvalues of the operators $A_j$ and $\hat A_j,$ respectively. An important related observation is that for infinite
groups, eigenvalues in at least one of these series have infinite multiplicity. Indeed, \eqref{eq:spectral} implies that
\begin{align*}
\text{mult}\, p_j(k)&=\dim V_k=\dim \cF^\natural L_2(\hat N_k,\hat\mu)=\dim L_2(\hat N_k,\hat\mu)\\
\text{mult}\, q_j(k)&=\dim \hat V_k=\dim\cF L_2(N_k,\mu)=\dim L_2(N_k,\mu).
\end{align*}
Therefore, the multiplicity of the eigenvalues $p_j(k)$ is finite if and only if the group $\hat X$ is discrete, and
the multiplicity of $q_j(k)$ is finite if and only if $X$ is discrete. However, if both $X$ and $\hat X$ are discrete,
Pontryagin's duality theory implies that they are both compact and therefore finite (see \cite{hew6370}, Theorem 23.17).
We obtain the following alternative for infinite groups.
\begin{lemma} If the group $X$ is compact and $\hat X$ is discrete then
$$
\text{\rm mult}\, p_j(k)<\infty,\quad \text{\rm mult}\, q_j(k)=\infty
$$
If both $X$ and $\hat X$ are noncompact (and therefore, not discrete), then
$$
\text{\rm mult}\, p_j(k)=\infty, \quad \text{\rm mult}\, q_j(k)=\infty.
$$
\end{lemma}
By a convention in spectral theory, eigenvalues of infinite multiplicity account for a continuous spectrum. Thus, according to
this lemma, in the case of infinite groups at least one of the sequences of operators $A_j,\hat A_j$ necessarily has continuous spectrum.

\subsubsection{Maximality of eigenspaces}
\begin{lemma} Spectral decomposition \eqref{eq:Aj2} has the property that for any $k_1,k_2\in \hat\I_0, k_1\ne k_2$ there exists
an operator $A_j,j\in\I_0$ such that $p_j(k_1)\ne p_j(k_2).$ The decomposition \eqref{eq:Aj3} has an analogous property
with respect to the eigenvalues $q_j(k)$ and operators $\hat A_j.$ Such decompositions are called \emph{maximal}; see
\eqref{eq:max}.
\end{lemma}
\begin{proof} Let
\begin{equation}\label{eq:mat}
\begin{array}{ll}
f(x)\simeq \sum_{k\in\I_0} f_k \chi_k(x)\in\Lambda_2(\cN), &g(\xi)\simeq \sum_{k\in\hat\I_0} g_k\hat \chi_k(\xi)\in\Lambda_2(\widehat\cN)\\
\cF^\sim f=g& f=\cF^\natural g.\end{array}
\end{equation}
Then \eqref{eq:pik1} and \eqref{eq:qik1} imply that the coefficients $f_k$ and $g_k$ are related as follows:
\begin{align}
g_k=\sum_{i\in\I_0} p_i(k) f_i, \quad k\in \hat\I_0\\
f_k=\sum_{i\in\hat\I_0} q_i(k)g_i, \quad k\in \I_0. \label{eq:fg}
\end{align}
Now suppose that for some $k_1,k_2\in\hat\I_0, k_1\ne k_2$ the equality $p_j(k_1)=p_j(k_2)$ is valid for all $j\in\I_0.$
This means that $g_{k_1}=g_{k_2}$ for all functions $f$ in \eqref{eq:mat}. But then the Fourier transform maps
$\Lambda_2(\cN)$ on the proper subspace of $\Lambda_2(\widehat\cN)$ defined by the condition $g_{k_1}=g_{k_2}$ rather than
on the entire space $\Lambda_2(\widehat\cN)$. This contradiction proves maximality of the spectral decomposition \eqref{eq:Aj2}.
Maximality of \eqref{eq:Aj3} follows in a similar way from \eqref{eq:fg}.
\end{proof}

\subsubsection{Intersection numbers} In the finite case, the product of adjacency matrices can be expanded into a linear combination
of these matrices. The coefficients of this expansion are nonnegative and are called the {\em intersection numbers} of the scheme, see
Def.~\ref{def0} and Eq.~\eqref{eq:A}(iv). In this section we establish similar relations in the general case.
\begin{theorem}\label{thm:int} We have \index{intersection numbers}
\begin{equation}\label{eq:AAij}
A_i A_j=\sum_{l\in \I_0} \gls{pijk} A_l,
\end{equation}
where
\begin{align}
p_{ij}^l=\sum_{k\in \hat \I_0} p_i(k)p_j(k)q_k(l)&=\frac 1{\mu(N_l)}\sum_{k\in \hat \I_0} p_i(k)p_j(k)\overline {p_l(k)}\hat\mu(\hat N_k),
\quad i,j,l\in\I_0.\label{eq:ser1}
\end{align}
Similarly,
\begin{equation}\label{eq:hatAAij}
\hat A_i\hat A_j=\sum_{l\in \I_0} \hat p_{ij}^l \hat A_l,\
\end{equation}
where
\begin{align}
\hat p_{ij}^l=\sum_{k\in \hat I_0} q_i(k)q_j(k)p_k(l)&
=\frac 1{\hat \mu(\hat N_l)} \sum_{k\in\I_0} q_i(k)q_j(k)\overline {q_l(k)}\mu(N_k),
\quad i,j,l\in\hat\I_0.\label{eq:ser2}
\end{align}
The series in \eqref{eq:ser1},\eqref{eq:ser2} converge absolutely.

We also have $p_{ij}^l,\hat p_{ij}^l\ge 0$ and
\begin{equation}\label{eq:commute}
p_{ij}^l=p_{ji}^l,\quad \hat p_{ij}^l=\hat p_{ji}^l
\end{equation}
for all $i,j,l\in \I_0$ or $\hat\I_0$ as appropriate.
\end{theorem}
\begin{proof} Let us first prove absolute convergence of the series in \eqref{eq:ser1}. Transformation between the
two forms of this series is performed using \eqref{eq:mv}, so it suffices to prove that one of them, say the one
on the right-hand side of \eqref{eq:ser1}, converges. The numbers $p_j(k)$ are contained in the spectrum of the operator $A_i$,
therefore, $|p_i(k)|\le\|A_i\|$. To estimate the norm of $A_i$ we proceed as in the proof of Lemma \ref{lemma:bound}. Using
the Schur test, we obtain
$$
|p_i(k)|\le\|A_i\|\le \bigg(\int_X \chi_i(x)d\mu(x)\bigg)^{\half}=\mu(N_i)^{\half}.
$$
At the same time, using the orthogonality relation \eqref{eq:pipj}, we obtain
$$
\sum_{k\in\hat\I_0} |p_j(k)|^2\hat\mu(\hat N_k)=\mu(N_j).
$$
By the Cauchy-Schwartz inequality
\begin{align*}
\sum_{k\in\hat\I_0}|p_j(k)\overline{p_l(k)}|\hat\mu(\hat N_k)&\le
\bigg(\sum_{k\in\hat\I_0}|p_j(k)|^2\hat\mu(\hat N_k)\bigg)^{\half}
\bigg(\sum_{m\in\hat\I_0}|p_l(m)|^2\hat\mu(\hat N_{m})\bigg)^{\half}\\
&=\Big(\mu(N_j)\mu(N_l)\Big)^{\half}.
\end{align*}
We obtain
\begin{align*}
\sum_{k\in\hat\I_0}|p_i(k)p_j(k)\overline {p_l(k)}|\hat\mu(\hat N_k)
\le  \Big(\mu(N_i)\mu(N_j)\mu(N_l)\Big)^\half
\end{align*}
where $i,j,l\in\I_0.$ Likewise, we obtain
$$
\sum_{k\in\I_0} |q_i(k)q_j(k)\overline {q_l(k)}|\mu(N_k)\le \Big( {\hat \mu(\hat N_i)\hat \mu(\hat N_j)}{\hat \mu(\hat N_l)}\Big)^\half.
$$
Thus, all the series in \eqref{eq:ser1},\eqref{eq:ser2} converge absolutely. 

Now let us prove \eqref{eq:AAij}-\eqref{eq:ser1}. Using \eqref{eq:Aj2}, \eqref{eq:Ei2}, and orthogonality of the projectors (Lemma~\ref{lemma:idem}), we find
$$
A_iA_j=\sum_{k\in\I_0} p_i(k)p_j(k)\sum_{l\in\I_0}q_k(l)A_l=\sum_{l\in\I_0}\Big(\sum_{k\in\hat\I_0} p_i(k)p_j(k)q_k(l)\Big)A_l.
$$

The proof of \eqref{eq:hatAAij}-\eqref{eq:ser2} is completely analogous. Finally, the commutativity conditions \eqref{eq:commute}
follow from \eqref{eq:ser1} and \eqref{eq:ser2}.
\end{proof}

The results of this theorem can be also expressed in terms of the kernels. They are summarized in the next lemma whose proof is immediate.
\begin{lemma}\label{lemma:kernel}
\begin{align*}
\int_X \chi_i(x-z) \chi_j(z-y) d\mu(z)\simeq \sum_{l\in\I_0} p_{ij}^l \chi_l(x-y)\\
\int_{\hat X} \hat \chi_i(\phi\xi^{-1})\hat \chi_j(\xi\psi^{-1})d\hat\mu(\xi) \simeq \sum_{l\in \hat\I_0} \hat p_{ij}^l \hat \chi_l(\phi\psi^{-1})
\end{align*}
\end{lemma}

Note that the integrals in this lemma can be easily computed. Indeed, we have
\begin{align*}
\int_X \chi_i(x-z)\chi_j(z-x)d\mu(z)&=\int_{X} \chi_i(x-z)\chi_{j'}(x-z)d\mu(z)=\delta_{ij'}\int_X \chi_i(t)d\mu(t)\\
&=\delta_{ij'}\mu(N_i).
\end{align*}
where we used Lemma~\ref{lemma:ai}(iii) together with \eqref{eq:sim}. Similarly, we obtain
$$
\int_{\hat X} \hat \chi_i(\phi\xi^{-1})\hat \chi_j(\xi\phi^{-1}) d\hat\mu(\xi)=\delta_{ij'}\hat \mu(\hat N_i).
$$
In analogy with Lemma \ref{lemma:pointwise}, we can make small modifications so that $L_2$ equalities in Lemma \ref{lemma:kernel}
hold pointwise.
\begin{lemma} \label{lemma:inter}
For all $x,y\in X$
$$
\int_X \chi_i(x-z)\chi_j(z-y)d\mu(z)=\sum_{i\in\I} p_{ij}^l \chi_l(x-y),
$$
where
$
p_{ij}^0=\delta_{ij'}\mu(N_i).
$
Similarly, for all $\phi,\xi\in\hat X$
$$
\int_{\hat X} \hat \chi_i(\phi\xi^{-1}) \hat \chi_j(\xi\psi^{-1})d\hat\mu(\xi)=\sum_{l\in\hat\I} \hat p_{ij}^l
\hat \chi_l(\phi\psi^{-1}),
$$
where
$
p_{ij}^0=\delta_{ij'}\hat \mu(\hat N_i).
$
Also,
\begin{equation}\label{eq:sym}
p_{ij}^0=p_{ji}^0, \quad \hat p_{ij}^0=\hat p_{ji}^0.
\end{equation}
\end{lemma}
\begin{proof} Pointwise equalities follow by the same arguments as Lemma \ref{lemma:kernel}. The symmetry conditions \eqref{eq:sym} follows
because $\mu(N_i)=\mu(N_{i'})$ and $\hat\mu(N_i)=\hat\mu(N_{i'})$ by \eqref{eq:N}.
\end{proof}

\subsubsection{Dual pairs of translation schemes} In the following theorem, which summarizes the results of this section, we define mutually dual translation association schemes.
\begin{theorem}\label{thm:SDP}
Let $\gls{N}=\{N_i,i\in\I\}$ and $\widehat N=\{\hat N_i, i\in\hat\I\}$ be spectrally dual partitions of mutually dual
topological Abelian groups $X$ and $\hat X.$
Let the partitions $\cR=\{R_i,i\in\I\}$ on $X\times X$ and $\hat \cR=\{R_i,i\in \hat\I\}$
on $\hat X\times\hat X$ be given by \eqref{eq:NR}, \eqref{eq:hNR}. Then the triples $\cX(X,\mu,\cR)$ and 
$\widehat\cX(\hat X,\hat\mu,\hat\cR)$ form translation invariant association schemes in the sense of Definition \ref{def1}. The intersection numbers $p_{ij}^l$ and $\hat p_{ij}^l$ of the schemes $\cX$ and $\widehat\cX$ are related to the spectral parameters $p_i(k)$ and $\hat p_i(k)$ of the partitions $\cN$ and $\widehat\cN$
according to \eqref{eq:AAij}-\eqref{eq:ser2}. If the partitions $\cN$ and $\widehat\cN$
are symmetric then the schemes $\cX$ and $\widehat\cX$ are symmetric.
\end{theorem}
\begin{proof} The proof follows from the arguments given earlier in this section and the definitions of the association scheme and
translation scheme, Defns.\ref{def1},\ref{def2}. Namely, parts (i)-(iv) in Definition \ref{def1} follow immediately from the way we
defined the partitions $\cR,\hat \cR.$ The condition for the intersection numbers, Def.~\ref{def1}(v), follows from Lemma \ref{lemma:inter};
cf. also \eqref{eq:conv}.
The final claim (about symmetry) is implied by the definition of symmetric partitions.\end{proof}

\vspace*{.1in} 
It is interesting to
note that our definition of dual schemes does not directly generalize the classical definition for finite Abelian groups
\cite{del73a,bro89}. Indeed, in the finite case the classes $\hat R_i,i\in\hat \I$ of $\widehat\cX$ are defined as follows:
\begin{equation}\label{eq:sm}
\hat R_i=\{(\phi,\psi)\in\hat X\times\hat X: \phi\psi^{-1}\in V_i\}, i\in \hat \I,
\end{equation}
where $V_i\subset L_2(X,\mu)$ are the maximal eigenspaces of all the operators $A_j,j\in\I.$ Then the blocks $\hat N_i, i\in\hat \I$ of
the dual partition $\widehat\cN$ are given by \eqref{eq:hatNi}:
$$
\hat N_i=\{\phi\in X: \phi\in V_i\}, \quad i\in \I.
$$
For the finite case the two definitions are equivalent. However, for infinite groups, Eq. \eqref{eq:sm} loses its meaning because
for locally compact groups $X$, e.g., $\reals$, characters $\phi\in\hat X$ are not $L_2$ functions. Of course, if $X$ is infinite but compact,
we still can use definition \eqref{eq:sm}. At the same time, for the dual discrete group $\hat X$, \eqref{eq:sm} is not well defined.
Therefore, adopting this definition, we would not be able to claim that the dual of the dual scheme $\hat{\hat \cX}$ is isomorphic to $\cX$, while for the finite case the schemes $\hat{\hat \cX}$ and $\cX$ are canonically isomorphic.
Thus, out of several possibilities we chose the definition of duality that extends without difficulty to the case of infinite groups.

\section{Spectrally dual partitions and association schemes on zero-dimensional Abelian groups}\label{sect:vil}
In the previous section we developed a theory of translation invariant schemes on Abelian groups which relies
on spectrally dual partitions. In this section we investigate the question whether such partitions exist
on topological Abelian groups. Our main result, given in Theorems \ref{thm:m1},~\ref{thm:m2}, will be that such partitions arise naturally on zero-dimensional groups
and their duals.

We begin with a simple but important result that identifies an obstruction to the existence of spectrally dual partitions.
\begin{proposition}\label{prop:nos} Let $X$ and $\hat X$ be a pair of dual topological Abelian groups. If at least one of the groups $X$ and $\hat X$
is connected, then the pair $(X,\hat X)$ does not support spectrally dual partitions.
\end{proposition}
\begin{proof} Suppose that $X$ and $\hat X$ support a pair of spectrally dual
partitions $\cN=\{N_i, i\in\I\}$ and $\widehat\cN=\{\hat N_i, i\in \I\}$. Assume toward a contradiction that $X$ is connected. Note that $\hat\mu(\hat X)=\infty$ because otherwise
$\hat X$ is compact which implies that $X$ is discrete. This implies that the set $\I$ is infinite because by definition,
all the subsets $\hat N_i, i\in \hat \I$ have finite measure.
Consider the indicator functions $\hat \chi_i(\xi)=\mathbbold{1}\{\xi\in \hat N_i\}, i\in \hat \I$ and their Fourier transforms
$\hat \chi_i^\natural(x)$ defined in \eqref{eq:ifa}. By conditions \eqref{eq:ip}, we have
$$
\int_X \chi_i^\natural(x) \overline{\hat \chi_j^\natural(x)} d\mu(x)=\int_{\hat X} \hat \chi_i(\phi)\overline{\hat \chi_j(\phi)}
d\hat\mu(\phi)=\delta_{ij}\hat\mu(\hat N_i).
$$
By definition, the functions $\hat \chi_i^\natural(x), i\in\I$ are piecewise constant on $X$: they are constant on the blocks 
$N_i, i\in \I,$
and therefore, take at most countably many values. Now observe that the functions $\hat \chi_i(\phi),
i\in\hat\I$ are absolutely integrable, and therefore, their Fourier transforms $\hat \chi_j^\natural(x)$ are continuous functions on $X.$

Observe that a piecewise constant function on $X$ can be continuous only if it is identically a constant.
Indeed, the set of values of a continuous function $f:X\to \complexes$ is closed in the natural topology. The set $E_a:=\{x\in X: f(x)=a\}$
is a union of several blocks $N_i$ and thus is also closed in $X.$ Since the sets $E_a$ are disjoint for different $a,$
this defines a partition of $X$ into several (at most countably many) disjoint closed sets, in contradiction to our assumption.

Thus, $\hat \chi_i^\natural(x)=c_i$ are constant for all $x\in X$ and $i\in\hat \I$, and $c_i\ne 0$ for all $i\in\hat\I_0$ because $\hat \chi_i^\natural(x)$ is a Fourier transform of a
nonzero function $\hat \chi_i(\phi).$ Now we obtain
$$
\mu(X)c_i\overline{c}_j=\delta_{ij}\hat\mu(\hat N_i).
$$
However if $\mu(\hat X)<\infty,$ these equalities do not hold for $i\ne j, i,j\in\hat \I_0$, and if
$\mu(X)=\infty,$ then they do not hold for every $i,j\in\hat\I_0.$
\end{proof}

This implies that spectrally dual partitions on connected Abelian groups do not exist. There are not so many such groups: their list is exhausted by the pairs $(X=\reals^d,\hat X=\reals^d)$ and $(X=(\reals/\integers)^d,\hat X=\integers^d);$ see 
\cite[Thm.~9.14]{hew6370} which also classifies groups formed of more than
one connected component.
At the same time there are vast classes of topological Abelian groups such that both the groups $X$ and $\hat X$ are zero-dimensional, and at least
one of them is uncountable and non-discrete. 
For such groups, the arguments in the proof of Proposition \ref{prop:nos} do not hold, and it becomes possible to define spectrally dual
partitions and translation association schemes.

\subsection{Zero-dimensional Abelian groups}\label{sect:zero}

A topological group $X$ is called {\em zero-dimensional} if the connected component
of the identity element $e$ is formed of $e$ itself. In this case all of its connected components are points, and
it is {\em totally disconnected} as a topological space. Conversely, if $X$ is locally compact, Hausdorff and totally disconnected,
then it is zero-dimensional \cite[Thm.3.5]{hew6370}. For this reason the terms zero-dimensional
and totally disconnected are often used interchangeably. 
\index{group!zero-dimensional}
\index{group!totally disconnected}
Examples of zero-dimensional groups include the Cantor set, groups of the Cantor type, i.e., countable products of
finite Abelian groups such as $\{0,1\}^\omega,$ as well as additive groups of the rings and fields of $p$-adic numbers. 
These examples also typify the general situation that includes two kinds of zero-dimensional groups: namely, the group $X$ can be 
periodic, in which case it contains finite subgroups, or non-periodic, e.g., the additive group of $p$-adic integers. 
In regards to the structure of the dual group $\hat X$, it is known that 
if $X$ is compact, then $\hat X$ is discrete, and if $X$ is locally compact, then $\hat X$ is also locally compact \cite[Thm.36]{pon66}.
Aspects of the general theory of zero-dimensional groups are found in \cite{pon66,hew6370,aga81,Edwards82}.

A systematic study of harmonic analysis on zero-dimensional Abelian groups was initiated by the observation of I.M.~Gelfand
who noticed that Walsh functions are precisely the continuous
characters of the group $\{0,1\}^\omega$ (see \cite{aga81}). Vilenkin \cite{Vilenkin47} generalized this
result to other Cantor-type groups (independently these results were also obtained by Fine \cite{Fine49}).
Currently zero-dimensional groups form the subject of a
large body of literature in harmonic analysis
(e.g., \cite{aga81,Edwards82,Lang96,Benedetto04}) and approximation theory \cite{skr06}, while their finite analogs have been studied in connection with a problem in combinatorial coding theory \cite{mar99,skr01}.

In the first part of this section we remind the reader the basics about compact and locally compact zero-dimensional Abelian groups. 
We include some details to make the paper self-contained and accessible to combinatorialists working on association schemes.
Moreover, the calculations performed below are not immediately available in the literature, and lay the groundwork for the analysis of association schemes later in this section. Detailed treatment of zero-dimensional Abelian groups is contained in Hewitt and Ross \cite{hew6370}.  A good reference source on such groups is the book
by Agaev at al.~\cite{aga81} which unfortunately is not available in English.

\subsubsection{Compact groups} We begin with the compact case which will also be useful in describing the locally compact case.
In this case the topology on $X$ defined by a countable chain of decreasing subgroups:
\index{group!zero-dimensional!compact}
\begin{equation}\label{eq:chain}
X=X_0\supset X_1\supset\dots\supset X_j\supset\dots\supset\{0\},
\end{equation}
where $X_j$ are subgroups of finite index $|X/X_j|$, and $\cap_{j\ge 0}X_j=\{0\}.$ The embeddings	 in \eqref{eq:chain} are strict,
so the index $|X_j/X_{j+1}|\ge 2, j=1,2,\dots.$ We note that generally, there are many different ways of forming the chain \eqref{eq:chain}
that give rise to the same topology on $X.$ 

By definition, the subgroups $\{X_j,j=0,1,\dots\}$ are open sets that form a countable base of neighborhoods of zero,
and the cosets $\{X_j+z, z\in X/X_j,j=0,1,\dots\}$ are open sets that form a countable base of the topology on $X.$
Thus, the topology in $X$ satisfies the second countability axiom. The converse is also true: if $X$ is locally compact and
satisfies the second countability axiom, then the topology on it is defined by a decreasing chain of subgroups \cite{pon66}.

It is easy to see that $X$ with topology defined by \eqref{eq:chain} is totally disconnected.
Indeed, for each $j$ we have the following partition of $X$ into cosets
\begin{equation}\label{eq:d}
\begin{array}{c}
X=\displaystyle{\bigcup_{z\in X/X_j}}(X_j+z)\\
(X_j+z_1)\cap(X_j+z_2)=\emptyset, \quad z_1,z_2\in X/X_j, z_1\ne z_2.
\end{array}
\end{equation}
Likewise, for every $j\ge 0$ we have
\begin{equation}\label{eq:dj}
\begin{array}{c}
X_j=\displaystyle{\bigcup_{z\in X_j/X_{j+1}}}(X_{j+1}+z)\\
(X_{j+1}+z_1)\cap(X_{j+1}+z_2)=\emptyset, \quad z_1,z_2\in X_j/X_{j+1}, z_1\ne z_2.
\end{array}
\end{equation}
This implies the following for the topology of $X.$ First,
$$
X_j=X\backslash Y_j, \text{ where }Y_j=\bigcup_{z\in X/X_j, z\ne 0} (X_j+z).
$$
The set $Y_j$ is a union of open sets $X_j+z$ and therefore is itself open. Thus, $X_j$ is closed, and so
all of the $X_j+z$ are both closed and open in the topology given by \eqref{eq:chain} (such sets are sometimes aptly called clopen).

Further, the group $X$ as well as all the subgroups $X_j$ are unions of disjoint open sets. This means that the groups $X_j,j\ge 0$
are disconnected, and for every point $x$ its connected component is $x$ itself. Decompositions \eqref{eq:d},\eqref{eq:dj}
also imply that $X$ affords arbitrarily fine coverings with multiplicity 1. By definition of topological dimension \cite[p.~I.15]{hew6370},
we obtain that $\dim X=0.$

The group $X$ is metrizable, i.e., the topology on $X$ can be defined by a metric. Let $\nu(0)=\infty$ and
\begin{equation}\label{eq:vl}
\gls{nu}=\max\{j: x\in X_j\}, \quad x\ne 0.
\end{equation}
We have
\begin{equation}\label{eq:um}
\begin{array}{l}
\nu(x+y)\ge\min\{\nu(x),\nu(y)\}, \quad x,y\in X\\
\nu(x+y)=\min\{\nu(x),\nu(y)\} \quad\text{ if }\nu(x)\ne \nu(y).
\end{array}
\end{equation}
We see that $\nu(x)$ is a discrete valuation on $X$ and defines on it a non-Archimedean metric. For instance we can put\footnote{Strictly speaking, $\rho$ is a norm that induces a metric on $X$. By abuse of terminology we use the term ``metric'' in both cases.}
\begin{equation}\label{eq:metric}
\gls{rho}=2^{-\nu(x)},
\end{equation}
then $\rho(0)=0$ and
\begin{equation}\label{eq:um1}
\begin{array}{l}
\rho(x+y)\le\max\{\rho(x),\rho(y)\},\quad x,y\in X\\
\rho(x+y)=\max\{\rho(x),\rho(y)\} \quad \text{if }\rho(x)\ne\rho(y)
\end{array}
\end{equation}
We conclude that $\rho(x-y)$ is a non-Archimedean metric, and the balls in this metric coincide with the subgroups $X_j$: 
\begin{equation}\label{eq:ball}
X_j=\{x\in X:\;\rho(x)\le 2^{-j}\}=\{x\in X: \rho(x)<2^{-j+1}\},\; j=0,1,\dots.
\end{equation}
This again shows that the balls are both open and closed. Further, two balls of the same radius in the Non-Archimedean metric are either disjoint or
coincide completely, and every point of the ball is the center.

A well-known example of a non-Archimedean metric arises in the construction of $p$-adic integers. A less standard example is provided by a problem in coding theory in which a metric on finite-dimensional vectors over
$\ff_q$ is defined by a (finite) chain of decreasing subgroups \cite{mar99,skr01}. This metric is an instance of {\em poset distances}\index{poset distance}
(metrics defined by partial orders of the coordinates) which will appear again below when we construct association schemes.
Zero-dimensional groups also arise in the context of multiplicative systems of functions such as the Walsh or Haar functions \cite{aga81, Edwards82,Golubov91}. 

The following classical fact holds true \cite[pp.28-30]{aga81}:
\begin{proposition}\label{prop:52} Let $X$ be a compact zero-dimensional group. Then $X$
can be identified with the set of all infinite sequences
$$
x=(z_1,z_2,\dots), \quad z_i\in X_{i-1}/X_i, i\in \naturals.
$$
\end{proposition}
\begin{proof} For $j=1,2,\dots $ let us fix a set of representatives $z_i(j), 0\le i\le n_j-1$ of the cosets
$X_{j-1}/X_j,$ so that
\begin{equation}\label{eq:cosets}
z_0(j)=0, \;z_1(j),\dots,z_{n_j-1}(j)\in X_{j-1}\backslash X_j,
\end{equation}
where $n_j$ is the index of $X_j$ in $X_{j-1}.$
Every element $x\in X$ can be represented uniquely as $x=z_1+y_1,$ where $z$ is one of the coset representatives of $X/X_1$ and $y\in X_1.$
Likewise, $y_1=z_2+y_2,$ and generally,
$$
x=z_1+\dots+z_j+y_j
$$
for all $j\ge 1,$ where $z_j$ are fixed according to \eqref{eq:cosets}. Note that the representatives found in earlier steps are not
changed in later steps, and that $y_j(x)\to 0$ in the non-Archimedean norm on $X.$ Therefore, every $x\in X$ can be written as
a convergent series
$$
x=\sum_{i\ge 0} z_i, \quad \text{where } z_i=z_i(x)\in X_{i-1}/X_i, i\in \naturals.
$$
Conversely, fixing arbitrary elements $z_i\in X_{i-1}/X_i, i\ge 1$, define $x_j=z_1+\dots +z_j, j\ge 1.$ Then
$$
x_{j+k}-x_j=z_{j+1}+\dots+z_{j+k}\in X_j,
$$
i.e., $x_{j+k}-x_j\to 0$ for $j\to\infty$ and every $k\in\naturals.$ We conclude that $(x_j,j\ge 1)$
is a Cauchy sequence $x_j$, and since $X$ is compact, the series
$\sum_{i\ge 1} z_i$ converges to a point $x\in X$. \end{proof}

The result of this proposition amounts to describing every point $x\in X$ as a sequence of nested balls that contain it:
$$
x=\bigcap_{j\ge 1}(X_j+x_j)=\bigcap_{j\ge 1}(X_j+z_1+\dots+z_j).
$$
Also, $X$ is a set of all such infinite sequences and therefore, clearly, is uncountable.

Using the result of Proposition~\ref{prop:52}, we can write the valuation \eqref{eq:vl} as follows: $\nu(0)=0$ and
$$
\nu(x)=\min\{j-1:z_j\ne 0\}, \quad x\in X\backslash\{0\}.
$$
Note that the metric on $X$ that gives rise to the same topology can be 
introduced in more than one way: for instance,  if $t(j),j\in\naturals_0$ is a strictly decreasing function on the set of nonnegative integers with $\lim_{j\to\infty} t(j)=0,$ then $t(\nu(x-y)),\, x,y\in X$ also defines a {\em non-Archimedean} metric on $X$ for which the balls are the same subgroups $X_j$ (the fact that $t(\cdot)$ defines a metric is specific to the non-Archimedean case). \index{metric!non-Archimedean}
The following distance will be useful below:
\begin{equation}\label{eq:rho0}
\rho_0(x)=|X/X_{\nu(x)}|^{-1}
\end{equation}
The function $t(\cdot)$ in this case is given by
$$
t(j)=\frac 1{\omega(j)}, \text{ where } \gls{omegaj}=|X/X_j|.
$$
\begin{lemma} The functions $t$ and $\omega$ have the following properties:
\begin{equation}\label{eq:omega}
\omega(0)=1; \quad\omega(j)=\prod_{i=1}^j n_i,  \;\text{where }\gls{ni}:=|X_{i-1}/X_i|
\end{equation}
\begin{equation}\label{eq:n}
\omega(j+1)=n_{j+1}\omega(j), \quad t(j+1)=\frac 1{n_{j+1}}t(j)
\end{equation}
\begin{align}
\sum_{i=j+1}^\infty (n_i-1)t(i)&=t(j), \quad j =0,1,\dots \label{eq:ti}\\
\sum_{i=1}^\infty (n_i-1)t(i)&=1\label{eq:1}
\end{align}
\end{lemma}
{\em Proof:} Equalities \eqref{eq:omega} and \eqref{eq:n} are immediate from \eqref{eq:chain}.
Relations \eqref{eq:ti}-\eqref{eq:1} now follow from \eqref{eq:n}:
\begin{align*}
\sum_{i=j+1}^\infty(n_i-1)t(i)&=\sum_{i=j+1}^\infty n_it(i)-\sum_{i=j+1}^\infty t(i)=\sum_{i=j+1}^\infty t(i-1)-\sum_{i=j+1}^\infty t(i)\\
&=t(j) \hspace*{3.5in} \qed
\end{align*}

This lemma and Proposition \ref{prop:52} imply that the group $X$ can be mapped on the segment $[0,1].$ 
Let us number the coset representatives of $X_{j-1}/X_j$ from $0$ to $n_j-1$ starting from $z=0$ in an arbitrary way
and write $N(z)=N_j(z)$ for the number of $z$ (thus $N(0)=0$). Define a mapping $\lambda:X\to[0,1]$ as follows:
   \begin{equation}\label{eq:map}
     x=(z_1,z_2,\dots)\;\mapsto\; \lambda(x)=\sum_{j\ge 1} t(j)N(z_j).
   \end{equation}
Since $t(j+1)/t(j)\le 1/2$ by \eqref{eq:n}, the series $\lambda(x)$ converges, and its value lies in $[0,1]$ because of \eqref{eq:1}.
The mapping $\lambda$ is not injective because there is a countable subset of points in $a\in [0,1]$ that can be written
in two ways, viz.,
  $$
  a=\sum_{j=1}^m t(j)N(z_j)=\sum_{j=1}^{m-1}t(j)N(z_j)+ (N(z_m')-1)+\sum_{j=m+1}t(j)N(z_j').
  $$
The preimages of the first and the second expressions above are two different points in $X,$ namely $x=(z_1,\dots,z_{m-1},z_m,0,0,\dots)$
and $x=(z_1,\dots,z_{m-1},z_m',z_{m+1}',\dots).$ To resolve this, the point $a$
is split into two points, written symbolically as $a-0$ and $a+0,$ whereupon $\lambda$ becomes one-to-one. It is possible to define a
topology on such modified segment $[0,1]$ so that if addition is inherited from $X$, it becomes a topological Abelian group
isomorphic to $X$.

\vspace*{.1in}
We will also need the {\em Haar measure} on $X.$ First define measures of the cosets $X_j+z$ by putting
\index{Haar measure}
   \begin{equation}\label{eq:Hm}
\mu(X_j+z)=t(j), \quad z\in X/X_j, j=0,1,\dots.
   \end{equation}
For a countable union $\cE$ of pairwise disjoint cosets $X_j+z$ define the measure by
$$
\mu(\cE)=\sum_{j,z}\mu(X_j+z),
$$
where the convergence follows from the convergence of the series \eqref{eq:1}. On account of \eqref{eq:n} we
also have
$$
\mu(X_j)=\mu\big(\textstyle{\bigcup_{z\in X_j/X_{j+1}}}\{X_{j+1}+z\}\big)=n_{j+1}\mu(X_{j+1}),
\quad j=0,1,\dots.
$$
These relations imply $\sigma$-additivity of the measure. Finally, we extend the measure to the set $\cP$
of all Borel subsets of $X$ and note that this extension is unique. The resulting measure is $\sigma$-additive and
is invariant with respect to translations and symmetries:
$$
\mu(\cE+x)=\mu(\cE),\;\mu(\cE)=\mu(-\cE), \quad\cE\in\cP.
$$
Details of the construction of the Haar measure are found in \cite{hew6370}.

The {\em character group} of $X$ is easily described. Let
\index{group!zero-dimensional!character group of}
$$
X_j^\bot:=\{\phi\in \hat X: \phi(x)=1 \text{ for all }x\in X_j\}
$$
		be the {\em annihilator} \index{annihilator subgroup} of the subgroup $X_j\subset X.$ Clearly, $X_j^\bot$ is a subgroup of $\hat X.$ Since $X_j$ is a
closed subgroup of $X$, the group $X_j^\bot$ is topologically isomorphic to the character group of the quotient $X/X_j$;
see \cite[Thm.23.25, p.365]{hew6370}. Since $X/X_j$ is finite, the annihilator is also finite and
\begin{equation}\label{eq:ann}
|X_j^\bot|=|X/X_j|=\omega(j)
\end{equation}
(cf. \eqref{eq:omega},\eqref{eq:n}). Further, from \eqref{eq:chain} we obtain the following reverse chain for the annihilators:
\begin{equation}\label{eq:rc}
\{1\}=X_0^\bot\subset X_1^\bot\subset\dots\subset X_j^\bot\subset\dots\subset X,
\end{equation}
and $\cup_{j\ge 0} X_j^\bot=\hat X.$ Thus, the character group is obtained as an increasing chain of nested finite groups.
The characters are easily found from the characters of the finite groups $X^\bot_j, j\in\naturals_0.$

The group $\hat X$ is countable, discrete, and periodic (i.e., every element has a finite order, which holds because it is
contained in some finite group in the chain \eqref{eq:rc}). In fact, the group $\hat X$ is discrete if and only if the group $X$ is compact, and it is periodic if and only if $X$ is zero-dimensional. These claims form a part of the general duality
theory of topological Abelian groups \cite{pon66, hew6370}.

The group $\hat X$ is metrizable. Indeed, put 
\begin{equation}\label{eq:dmd}
\hat\rho(\phi)=\min\{j:\phi\in X_j^\bot\}.
\end{equation}
Clearly, $\hat\rho(\phi)\ge 0, $ $\hat\rho(\phi)=0$ iff $\phi=1$, and
$$
\hat\rho(\phi\psi^{-1})\le\max\{\hat\rho(\phi),\hat\rho(\psi)\}.
$$
Thus, $\hat\rho(\cdot)$ is a non-Archimedean metric on $\hat X,$
and the subgroups $X_j^\bot$ are the balls in this metric:
\begin{equation}\label{eq:balls1}
\hat X_j=\{\phi\in X: \hat\rho(\phi)\le j\}, \quad j\ge 0.
\end{equation}
The dual statement for Proposition \ref{prop:52} has the following form.
\begin{proposition}\label{prop:52d}
The countable discrete topological space $\hat X$ can be identified with the set of infinite sequences
$$
\phi=( \pi_1, \pi_2,\dots), \quad \pi_j\in X_j^\bot/X_{j-1}^\bot
$$
where only a finite number of entries $\pi_j\ne 1.$
\end{proposition}
\begin{proof} Fix a set of representatives $\Theta(j)=\{\theta_i(j), 0\le i\le n_j-1\}$ of the cosets $X_j^\bot/X_{j-1}^\bot$ in the
group $X_j, j=1,\dots.$ Let us agree that $\theta_0(j)=1$ (the unit character of the subgroup $X_j^\bot$) for all $j$.
Note that the numbers $n_j$ are the same as in
\eqref{eq:omega} because
\begin{align}
n_j&=|X_{j-1}/X_j|=|X_j^\bot/X_{j-1}^\bot|, j=1,2,\dots \label{eq:sam}\\
n_0&=|X/X_0|=|X_0^\bot|=1.\nonumber
\end{align}
Once the coset representatives are fixed, any character $\phi\in X_j^\bot$ can be written uniquely as $\phi= \pi_j\psi_{j-1},$
where $ \pi_j\in X_j^\bot/X_{j-1}^\bot, \psi_{j-1}\in X_{j-1}^\bot.$ Continuing this process, we obtain
$$
\phi=\prod_{i=1}^j \pi_i, \quad \pi_i\in X_i^\bot/X_{i-1}^\bot,
$$
where $ \pi_j\in \Theta(j).$
\end{proof}

Using this result, we can write the metric $\hat\rho$ as
$$
\hat\rho(\phi)=\max\{j: \pi_j\ne\theta_0(j)\}, \phi\ne 1,
$$
and $\hat\rho(1)=0$ (this follows because $1=(\theta_0(1),\theta_0(2),\dots)$).
The Haar measure is just the counting measure: $\hat\mu(\cE)=|\cE|, \cE\in \hat X.$ In particular
$\hat\mu(X_j^\bot)=|X_j^\bot|=\omega(j);$ cf. \eqref{eq:ann}. Recall that we chose the normalization
$\mu(X)=1,\hat\mu(X^\bot)=1$ to satisfy the Parseval identities \eqref{eq:Parseval},\eqref{eq:ip}.
Note that these equalities hold for any normalization that satisfies $\mu(X)\hat\mu(X^\bot)=1$.

\vspace*{.1in}
{\em Remark:} The zero-dimensional Abelian groups considered here belong 
to the class of the so-called {\em profinite groups} that have important applications in algebra and number theory; see \cite{Ribes00}. \index{group!profinite} These groups are conveniently 
described in the
language of projective (inverse) and inductive (direct) limits of topological spaces. For instance, the group $X$ together with its
chain of nested subgroups \eqref{eq:chain} is a projective, and the group $\hat X$ with its chain \eqref{eq:rc} an inductive
limit: 
\index{limit!inductive (direct)}
\index{limit!projective (inverse)}
$$
X=\lim_{\stackrel{\longleftarrow}{i}} X_i, \quad \hat X=\lim_{\stackrel{\longrightarrow}{i}} X_i^\bot.
$$
However we prefer to avoid this specialized language to make our paper accessible not just to algebraists and tolopogists, but also to a broader mathematical audience.
\vspace*{.1in}

\subsubsection{Locally compact groups} Let us briefly
outline the changes that are needed in the setting of the previous section in order to include the locally compact case
in our considerations. Let $X$ be a locally compact uncountable zero-dimensional Abelian group. $X$ contains a doubly infinite
chain of nested compact subgroups
\index{group!zero-dimensional!locally compact}
\begin{equation}\label{eq:dc}
X \supset\dots\supset X_{j-1}\supset X_j\supset\dots\supset\{0\}, \quad j\in \integers,
\end{equation}
where
$$
\bigcup_{j\in \integers} X_j=X,\quad\bigcap_{j\in\integers} X_j=\{0\}
$$
and $X_{j-1}/X_j, j\in\integers$ are finite Abelian groups. The inclusions are strict, so $|X_{j-1}/X_j|\ge 2.$
The chain \eqref{eq:dc} defines a topology on $X$ in which the subgroups $X_j, j\in \integers$ form the base of neighborhoods
of zero and are both open and closed. This topology is metrizable. The corresponding discrete valuation and non-Archimedean metric
are defined similarly to \eqref{eq:vl}, \,\eqref{eq:metric}:
\begin{align}
\gls{nu}&=\max\{j: x\in X_j\}, \quad j\in\integers\\
\gls{rho}&=2^{-\nu(x)}, \quad x\in X, \label{eq:dis}
\end{align}
except that in this case $\nu(\cdot)$ can be any integer. All the subgroups $X_j,j\in\integers$ are balls in the metric $\rho.$
The topological space $X$ can be identified with the space of doubly infinite sequences
\begin{equation}\label{eq:di}
x=(\dots,z_j,z_{j+1},\dots), \quad z_j\in X_{j-1}/X_j, j\in \integers
\end{equation}
such that $z_j=0$ for all $j<\nu(x).$ As before in \eqref{eq:cosets},
let us assume that the coset representatives $z_j,j\in\integers$ are fixed.

Using expansion \eqref{eq:di}, we can map the group $X$ to the interval $[0,\infty).$ 
We proceed analogously to the compact case \eqref{eq:map}, defining a map $\lambda:X\to[0,\infty)$ by 
  \begin{equation}\label{eq:map-lc}
  x=(\dots,z_j,z_{j+1},\dots)\; \mapsto\; \lambda(x)=\sum_{j=-\nu(x)}^\infty t(j)N(z_j)
  \end{equation}
where $N(z_j), 0\le N(z_j)\le n_j-1$ is the index of the coset representative $z\in X_{j-1}/X_j,$ $N(0)=0,$
$n_j=|X_{j-1}/X_j|,$ and 
   $$
t(j)=\begin{cases}\prod_{i=0}^{j+1} n_j, &\text{if }-\nu(x)\le j\le -1\\
    1 &\text{if }j=0\\
    \prod_{i=1}^j \frac1{n_j} &\text{if } j\ge 1.
   \end{cases}
   $$
Convergence of the series in \eqref{eq:map-lc} to a point in $[0,\infty)$ again follows from \eqref{eq:1}. The mapping $\lambda$ is not injective but can be made such using the arguments following \eqref{eq:map}.
   
\vspace*{.1in}
Let us take one of the subgroups, say $X_0,$ in the chain \eqref{eq:dc}, and consider the group $H=X/X_0.$ We see that
\begin{equation}
H\supset\dots\supset H_{j+1}\supset H_j\supset\dots\supset H_1\supset H_0=\{0\},
\end{equation}
where $H_j=X_{-j}/X_0, j=0,1,\dots$ are finite Abelian groups and $H=\cup_{j\ge 0} H_j.$ The group $H$ is countably infinite,
discrete, and periodic.

\remove{Note that if $x=y+h=0$ then $y=0$ and $h=0.$ This implies the following well-known result.
\begin{lemma}\label{lemma:factor}
The locally compact group $X$ is a direct product of a compact group $X_0$
and a countable discrete group $H$:
\begin{equation}\label{eq:dp}
X=X_0\times H.
\end{equation}
\end{lemma}
This relation enables us to study the group $X$ in terms of the already familiar groups $X_0$ and $H$.
For instance,}

Using the language of bi-infinite sequences \eqref{eq:di} we can write
\begin{equation}\label{eq:li}
\begin{array}{l}
x=y+h, \quad y\in X_0, h\in H\\
y=(z_1,z_2,\dots),\quad z_j\in X_{j-1}/X_j, j=1,2,\dots\\
h=(h_1,h_2,\dots),\quad h_j=z_{-j+1}\in X_{-j}/X_{-j+1}.
\end{array}
\end{equation}
The {\em Haar measure} $\mu$ on $X$ can be defined as follows.
\index{Haar measure}
Note that
the cosets $\{X_0+h,h\in H\}$ form a partition the group $X$:
    \begin{equation}\label{eq:XH}
X=\bigcup_{h\in H} (X_0+h); \quad (X_0+h_1)\cap(X_0+h_2)=\emptyset,\;\; h_1\ne h_2.
    \end{equation}
For any Borel set $\cE\subset X$ put
\begin{align}\label{eq:mE}
\mu(\cE)&=\sum_{h\in H} \mu(\cE\cap(X_0+h))=\sum_{h\in H} \mu_0((\cE-h)\cap X_0)
\end{align}
where $\mu_0$ is the Haar measure on the compact subgroup $X_0.$
Noting that the total measure of $X$ is infinite, let us normalize the measure by the condition
\begin{equation}\label{eq:norm}
\mu_0(X_0)=1.
\end{equation}
Finally note that the choice of $X_0$ above is arbitrary: instead of $X_0$ this construction can rely on any
other subgroup $X_i$ \eqref{eq:dc}.
\remove{This gives rise to a series of factorizations of the group $X$ of the form
\begin{equation}\label{eq:seq}
X=X_i\times H^{(i)}, \quad\text{where }H^{(i)}=X/X_i, i\in \integers,
\end{equation}
where $X_i$ is compact and $H^{(i)}$ is countable, discrete, and periodic.}

\vspace*{.1in}\emph{The dual group:} The dual group $\hat X$ of a locally compact uncountable group $X$ is also locally compact
and contains a bi-infinite chain of annihilators $X_j^\bot\subset \hat X$ of the subgroups $X_j\subset X:$
\begin{equation}\label{eq:ca}
\{1\}\subset\dots\subset X_{j-1}^\bot\subset X_j^\bot\subset\dots\subset \hat X, \quad j\in\integers,
\end{equation}
where $\cup_{j\in\integers} X_j^\bot=\hat X,\cap_{j\in\integers} X_j^\bot=\{1\},$ and $|X_j^\bot/X_{j-1}^\bot|=|X_{j-1}/X_j|=n_j;$
see \eqref{eq:sam}. Note that $\{0\}^\bot=\hat X$ and $X^\bot=\{1\}.$

The subgroups $X_j^\bot, j\in\integers$ form the base of neighborhoods of the unit character. This topology is also metrizable,
and the corresponding discrete valuation $\hat \nu$ and metric $\hat\rho$ have the form
\begin{align}
\hat\nu(\phi)&=\max\{-j: \phi\in X_j^\bot\},\\
\hat\rho(\phi)&=2^{-\hat\nu(\phi)}, \quad \phi\in \hat X.\label{eq:disd}
\end{align}
As before, the subgroups $\hat X_j$ are the balls in the metric $\hat \rho.$ Note that the nesting in \eqref{eq:dc} and \eqref{eq:ca}
is in opposite directions because the larger the subgroup $X_j$ the smaller its annihilator $\hat X_i.$

The topological space $\hat X$ can be identified with the space of all bi-infinite sequences
\begin{equation}\label{eq:cs}
\phi=(\dots, \pi_{j-1}, \pi_j,\dots), \quad \pi_j\in X_j^\bot/X_{j-1}^\bot, \quad j\in\integers
\end{equation}
such that $ \pi_j=1$ for $j> \hat\nu(\phi).$ (Here as before we assume that the elements $ \pi_j, j\in\integers$ are chosen from a fixed system
of coset representatives $X_j^\bot/X_{j-1}^\bot$ contained in $X_j.$)

It is convenient to have explicit expressions for the groups considered in terms of coset representatives. In particular, we have a set of relations that is dual to
\eqref{eq:li}:
\begin{equation*}
\begin{array}{l}
\phi=\psi\cdot\xi, \quad\psi\in X_0^\bot=\hat X, \xi\in\hat X/X_0^\bot=\hat X_0\\
\psi=(\zeta_1,\zeta_2,\dots), \quad\zeta_j=\pi_{-j+1}\in X^\bot_{-j+1}/X^\bot_{-j}\\
\xi=(\pi_1,\pi_2,\dots), \quad \pi_j\in X_j^\bot/X_{j-1}^\bot, j=1,2,\dots.
\end{array}
\end{equation*}
\remove{As before, we can also obtain a sequence of expansions dual to \eqref{eq:seq}:
$$
\hat X=\hat X_i\times \hat H^{(i)}, \quad i\in \integers
$$
where $\hat X_i=\hat X/X_i^\bot$ are countable, discrete, and periodic groups and $\hat H^{(i)}$ are compact zero-dimensional groups.}
These relations enable us to define the Haar measure on $\hat X$ as follows. Note that the
cosets $\{X_0^\bot\xi,\xi\in\hat X_0\}$ form a partition of the group $\hat X$:
  $$
  \hat X=\cup_{\xi\in\hat X_0}(X_0^\bot\xi); \quad (X_0^\bot\xi_1)\cap (X_0^\bot\xi_2), \;\;\xi_1\ne\xi_2.
  $$
For any Borel set $\cE\in\hat X$
put
$$
\hat\mu(\cE)=\sum_{\xi\in\hat X_0} \hat\mu(\cE\cap X_0^\bot\xi)=\sum_{\xi\in\hat X_0}
\hat\mu_0(\cE\xi^{-1}\cap X_0^\bot),
$$
where $\mu_0$ is the Haar measure on the compact subgroup $X_0^\bot.$
Finally, we normalize the measure by $\hat\mu(X_0^\bot)=1.$ Together with the normalizations \eqref{eq:norm}
this implies the Parseval relations \eqref{eq:Parseval}, \eqref{eq:ip}.

\vspace*{.1in}
{\em Self-dual groups:} \index{group!self-dual} Examples of
self-dual locally compact Abelian groups can be easily constructed. Let $X_0$ be an arbitrary compact Abelian group and let $\hat X_0$ be its dual group.
Consider a locally compact Abelian group
\begin{equation}\label{eq:selfdual}
X=X_0\times\hat X_0.
\end{equation}
By Pontryagin's duality, $\Hat{\Hat{X_0}}\cong X_0,$ so
\begin{equation}\label{eq:sd1}
\hat X=\hat X_0\times \Hat{\Hat{X_0}}\cong X.
\end{equation}
Thus all the groups of the form \eqref{eq:selfdual} are self-dual (see more on self-dual groups in \cite[p.~I.422]{hew6370}). Below we use self-dual groups to construct a large class of
examples of self-dual association schemes.

\subsection{Dual pairs of association schemes}\label{sect:dualpairs} In this section we present a construction 
of dual pairs of translation schemes starting with a pair $(X,\hat X)$ where $X$ is a locally compact Abelian zero-dimensional group. 
The argument proceeds by partitioning $X$ and $\hat X$
into spheres, thereby constructing a pair of spectrally dual partitions. To prove duality, we will need some results about the
Fourier transforms of functions that are constant on spheres.

\subsubsection{ Fourier transforms} Let $X$ be a compact or locally compact Abelian group and let $D\subset X$ be a compact subgroup. Assume that $D$ is both open
and closed. Then the annihilator $D^\bot\subset \hat X$ is also a compact subgroup of $\hat X$ and is also both open and closed.
Clearly $\mu(D)>0$ and $\hat\mu(D^\bot)>0$ since $D$ and $D^\bot$ are open, and $\mu(D)<\infty,\mu(D^\bot)<\infty$ since they
both are compact.

Let $\chi[D;x]$ and $\chi[D^\bot;\xi]$ be the indicator functions of $D$ and $D^\bot$. We will need explicit 
expressions for their Fourier transforms. We remind the reader this result \cite[pp.81-82]{aga81}
whose short proof is included for completeness\footnote{\label{note:D}There is a certain notational ambiguity in the expressions below: namely,
the letter $D$ in $\chi[D;\cdot]$ refers to the domain while the same letter in $\tilde \chi[D;\cdot]$ is simply a label of the function.
This convention will be used throughout.}.
\begin{lemma}
\begin{align}
\tilde \chi[D;\xi]&=\mu(D)\chi[D^\bot;\xi] \label{eq:D}\\
\chi^\natural[D^\bot;x]&=\hat\mu(D^\bot)\chi[D;x]. \label{eq:Db}
\end{align}
\end{lemma}
\begin{proof} By \eqref{eq:ft} we have
$$
\tilde \chi[D;\xi]=\int_D\xi(x) d\mu(x).
$$
If $\xi\in D^\bot,$ then the result is obvious. Otherwise, let $x_0\in D$ be such that $\xi(x_0)\ne 1.$ Since the Haar measure
is invariant and $D$ is a subgroup, we obtain
\begin{align*}
\int_D \xi(x) d\mu(x)&=\int_{D+x_0}\xi(x+x_0) d\mu(x)=\int_D\xi(x+x_0)d\mu(x)\\
&=\xi(x_0)\int_D \xi(x)d\mu(x)
\end{align*}
so $\tilde \chi[D;\xi]=0,$ which proves \eqref{eq:D}. The proof of \eqref{eq:Db} is the same if one takes into account that
$(D^\bot)^\bot=D$ by the duality theorem.
\end{proof}

We assume that the measures are normalized so that the Parseval identities \eqref{eq:Parseval},\eqref{eq:ip} are
satisfied. Then the transforms $\cF^\natural$ and $(\cF^\sim)^{-1}$ are inverse of each other. Applying $\cF^\natural$ to
\eqref{eq:D} and $\cF^\sim$ to \eqref{eq:Db}, we obtain the following dual relations:
\begin{align}
\chi[D;x]&=\mu(D)\chi^\natural[D^\bot; x] \label{eq:D1}\\
\chi[D^\bot;x]&=\hat\mu(D^\bot)\tilde\chi[D;\xi]. \label{eq:Db1}
\end{align}
Comparing the first of these equalities with \eqref{eq:Db}, or the second with \eqref{eq:D}, we obtain an important relation:
\begin{equation}\label{eq:PW}
\mu(D)\hat\mu(D^\bot)=1.
\end{equation}
On account of it, the pair of relations \eqref{eq:D1},\eqref{eq:Db1} is equivalent to the formulas \eqref{eq:D},\eqref{eq:Db}.

\subsubsection{Zero-dimensional groups and the uncertainty principle}\label{sect:uncertainty}
\index{uncertainty principle}
Observe that equalities \eqref{eq:D},~\eqref{eq:Db} express a rather nontrivial fact: 
in the topological spaces considered, there exist compactly supported functions $\chi[D;x]$ and $\chi[D^\bot;\xi]$
whose Fourier transforms are also compactly supported. This fact has an interesting interpretation in the context of
the ``uncertainty principle'' of harmonic analysis that deserves a more detailed discussion. 
The uncertainty principle is a general 
statement that a function $f$ on an Abelian group $X$ and its Fourier transform $\tilde f$ on the dual group $\hat X$ cannot 
both be ``well localized.'' For instance, in the case of $X=\reals$ this fact constitutes the statement of the 
Paley-Wiener theorem.
A similar obstruction exists for any connected Abelian group. Namely, such a group is topologically isomorphic either to a torus $(\reals/\integers)^l, l>0$
(if $X$ is compact) or to a direct product of $\reals^k,k>0$ and a torus $(\reals/\integers)^l, l>0$ (if $X$ is locally compact).
Such groups do not contain open subgroups which makes relations of the form \eqref{eq:D},~\eqref{eq:Db} impossible.

The uncertainty principle can be formalized in a number of ways, see, e.g., \cite{Folland97}. For totally disconnected groups the
following form of this principle is of interest:

{\em Consider the measure space $(X,\mu),$ where $X$ is an Abelian group and $\mu$ is the Haar measure.
Let $f\in L_2(X), f\ne 0$ and let $\tilde f(\xi)$ be the Fourier transform of $f$ \eqref{eq:ft}. 
Then 
   \begin{equation}\label{eq:mm}
   \mu(\supp f)\hat\mu(\supp \tilde f)\ge 1.
   \end{equation}
where $\supp f=\{x: f(x)\ne 0\}$ and $\hat X$ is the dual group.
}

Note that the function $f$ as an element of $L_2(X)$ is a class of functions that can differ on a subset of measure 0,
so the quantities $\mu(\supp f)$ and $\hat \mu(\supp\tilde f)$ are well defined.
For finite groups inequality \eqref{eq:mm} was pointed out in \cite{Donoho89}; see also \cite[Ch.14, Thm.1]{Terras99}.
The general version of this inequality affords a short simple proof which we include for reader's convenience. Indeed,
we have
   \begin{equation}\label{eq:2i}
   \|f\|_2^2=\int_X|f(x)|^2d\mu(x)\le \|f\|_\infty^2\,\mu(\supp f),
   \end{equation}
where $\|f\|_\infty=\text{ess\,sup}|f(x)|.$ 
At the same time, using \eqref{eq:ft} and the Cauchy-Schwartz inequality, we have
   \begin{align*}
   \|f\|_\infty&\le \int_{\hat X}|\tilde f(\xi)|\le \Big(\int_{\hat X}|\tilde f(\xi)|^2 d\hat\mu(\xi)\Big)^\half
     \Big(\hat\mu(\supp \tilde f)\Big)^\half
     \\
     &=\|\tilde f\|_2\Big(\hat\mu(\supp \tilde f)\Big)^\half
     \end{align*}
Substituting this inequality into \eqref{eq:2i} and noting that (by the Parseval identity \eqref{eq:Parseval})
$\|f\|_2=\|\tilde f\|_2\ne 0,$ we obtain \eqref{eq:mm}.

Observe that for zero-dimensional groups there exist functions that ``optimize'' the uncertainty principle: namely, the inequality\eqref{eq:mm}
holds with equality. Examples of such functions include indicators of subgroups \eqref{eq:PW} as well as piecewise-constant wavelets 
introduced below in \eqref{eq:psi}.


\subsubsection{Balls and spheres: Spectrally dual partitions} Let us return to the main subject of this section and write
relations \eqref{eq:D} and \eqref{eq:Db} for the subgroups $X_j\subset X$ and $X_j^\bot \subset \hat X$ in the chains \eqref{eq:chain}, \eqref{eq:rc} and \eqref{eq:dc}, \eqref{eq:ca}. We have
\begin{align}
\tilde \chi[X_j;\xi]&=\mu(X_j)\chi[X_j^\bot;\xi]\label{eq:f1}\\
\chi^\natural[X_j^\bot;x]&=\hat\mu(X_j^\bot)\chi[X_j;x]\label{eq:f2}\\
\mu(X_j)\hat\mu(&X_j^\bot)=1.\label{eq:f3}
\end{align}
As remarked earlier, see e.g., \eqref{eq:ball},~\eqref{eq:balls1}, the subgroups $X_j$ and $X_j^\bot$ form balls in the corresponding non-Archimedean metrics $\rho$ and $\hat\rho.$
Let us number these balls by their radii. Let
\begin{align}
\gls{Br}&=\{x:\rho(x)\le r\}, \quad r\in\I  \label{eq:b}\\
\hat B(t)&=\{\xi:\hat\rho(\xi)\le t\}, \quad t\in \hat \I. \label{eq:^b}
\end{align}
As above, we use the notation
\begin{align*}
\I_0=\{r\in\I:\mu(B(r))>0\}, \quad \hat\I_0=\{t\in\I: \hat\mu(\hat B(t))>0\}.
\end{align*}

Note that the use of notation $\I,\I_0,$ etc. is consistent with earlier use because the balls will be used below
to form the blocks of the partitions.
Let us describe the sets $\I,\I_0$ and $\hat\I,\hat\I_0$ for different topologies of the groups considered.

$(i)$ Let $X$ be infinite compact and $\hat X$ be a countable discrete group. The metrics $\rho$ and $\hat\rho$ are given
by Eqns.~\eqref{eq:metric}, \eqref{eq:dmd}, and the sets of radii have the form
\begin{equation}\label{eq:rr1}
\begin{matrix}
\I=\I_0\cup\{0\}, \quad \I_0=\{2^{-j},\, j\in \naturals_0\}\\[.05in]
\hat \I=\hat\I_0=\{j,\, j\in \naturals_0\}.
\end{matrix}
\end{equation}

$(ii)$ Let both $X$ and $\hat X$ be locally compact. Then the metrics $\rho$ and $\hat\rho$ are given by \eqref{eq:dis} and \eqref{eq:disd},
respectively, and the radii take the values
\begin{equation}\label{eq:rr2}
\begin{matrix}
\I=\I_0\cup\{0\}, \quad \I_0=\{2^{-j},\, j\in \integers\}\\[.05in]
\hat \I=\hat\I_0\cup\{0\}, \quad \hat\I_0=\{2^{j},\, j\in \integers\}.
\end{matrix}
\end{equation}

$(iii)$ For completeness, let us discuss the case of finite Abelian groups $X$ and $\hat X$ with nested chains of subgroups of length $d$:
\begin{equation}\label{eq:NRT}
\begin{array}{c}
X=X_0\supset X_1\supset\dots\supset X_j\supset\dots\supset X_d=\{0\}\\
\{0\}=X_0^\bot\subset X_1^\bot\subset\dots\subset X_j^\bot\subset\dots\subset X_d^\bot=\hat X.
\end{array}
\end{equation}
The metrics on the groups $X$ and $\hat X$ are given by the expressions
\begin{align*}
\rho(x)&=d-\min\{j=0,1,\dots,d: x\in X_j\}=\max\{i=0,1,\dots,d: x\in X_{d-i}\}\\
\hat\rho(\phi)&=\max\{j=0,1,\dots,d:\phi\in X_j^\bot\}
\end{align*}
and the sets of radii are
\begin{equation}\label{eq:rr3}
\I=\I_0=\hat\I=\hat\I_0=\{0,1,\dots,d\},
\end{equation}
so the dual radii are given by
   \begin{equation}\label{eq:dr}
r(j)=d-j,\;\; \hat r(j)=j, \quad j=0,1,\dots, d.
  \end{equation}
Note that the case of finite groups $X$ and $\hat X$ is far from trivial. The corresponding distances are known
in coding theory as the Rosenbloom-Tsfasman metrics \cite{ros97}. The association scheme for this case was introduced in
\cite{mar99} and studied in detail in \cite{bar09b}. Combinatorial problems for the Rosenbloom-Tsfasman and other related metric
spaces have been the subject of significant literature in the last decade, see e.g., references in \cite{bar09b}. However, in this paper we do not
devote special attention to these questions because our main goal is to study schemes on infinite sets.

\vspace*{.1in}
Define a pair of mutually inverse bijections on the sets $\I_0$ and $\hat\I_0,$
  \begin{equation}\label{eq:rtr}
\I_0\ni r \to \tilde r\in \hat\I_0, \quad \hat\I_0\ni t\to t^\natural\in \I_0
  \end{equation}
defined by the relations
\begin{equation}\label{eq:BB}
B(r)^\bot=\hat B(\tilde r), \quad B(t^\natural)^\bot=\hat B(t). 
\end{equation}
For the examples mentioned above, these bijections have the following form:\\[.05in]
$(i)$ From \eqref{eq:rr1} we obtain
$
\tilde r=-\log_2(r), \;t^\natural=2^{-t}; \quad r=2^{-j},\, j,t\in\naturals_0.
$

\nd $(ii)$ From \eqref{eq:rr2} we obtain
$
\tilde r=r^{-1},\; t^\natural=t^{-1};\quad r=2^{-j},t=2^j,\, j\in \integers.
$

\nd $(iii)$ In the case of finite groups,
$
\tilde r=d-r, \;t^\natural=d-t; \quad r,t\in\{0,1,\dots,d\}.
$

\vspace*{.05in}
For any $r_1,r_2\in \I_0$
$$
B(r_1)\subset B(r_2) \;\Leftrightarrow \;\hat B(\tilde r_1)\supset \hat B(\tilde r_2)
$$
and for any $t_1,t_2\in\hat\I_0$
$$
B(t_1^\natural)\supset B(t_2^\natural) \;\Leftrightarrow\; \hat B(t_1)\subset \hat B(t_2).
$$
In other words,
\begin{equation}
\begin{matrix}
\tilde r_1>\tilde r_2 \quad{\text{iff }}\quad r_1<r_2\\[.05in]
t_1^\natural>t_2^\natural \quad{\text{iff }}\quad t_1<t_2.
\end{matrix}
\label{eq:gl}
\end{equation}

\vspace*{.1in}
In the compact case let
  \begin{equation}\label{eq:br}
  \gls{barr}=\max\{r: r\in\I_0\}.
  \end{equation}
  denote the maximum value of the radius.
Given a value of the radius $s\in \I$ or $\hat\I,$ let
\begin{equation}\label{eq:pn}
\gls{tau-}(s)=\max\{r: r<s\}, \;\; \gls{tau+}(s)=\min\{r: r>s\}.
\end{equation}
Note that $\tau_-(s)$ is undefined if $s=0$ and $\tau_+(s)$ is undefined if $s=\bar r.$
We have $\tau_-(s)<s<\tau_+(s)$,
with no values in between.

Recalling that all the balls are subgroups, let us introduce the notation for the
subgroup index:
   \begin{equation}\label{eq:si}
n(r)=|B(r)/B(\tau_-(r))|.
   \end{equation}

\begin{proposition}\label{prop:+-}
(a) We have 
  \begin{equation}\label{eq:+-}
\widetilde{\tau_-(r)}=\tau_+(\tilde r), \quad \tau_-(t)^\natural=\tau_+(t^\natural)
   \end{equation}
   In words, a one-step move in
the ``time domain" corresponds to a one-step move in the opposite direction in the ``frequency domain."
  
   (b) We have
   \begin{equation}\label{eq:nr}
   n(r)=\hat n(\tau_+(\tilde r))
   \end{equation}
   \end{proposition}
\begin{proof} (a) Let us prove the first equality, the second follows in the same way.
Since $\tau_-(r)<r,$ \eqref{eq:gl} implies that $\widetilde{\tau_-(r)}>\tilde r.$ If $\widetilde{\tau_-(r)}\ne \tau_+(\tilde r),$
then there is a radius $\tilde a$ such that $\widetilde{\tau_-(r)}>\tilde a>\tilde r,$ i.e., $\tau_-(r)<a<r,$ which is a contradiction.

(b) Using duality theory, we observe that if $G_1\subset G_2\subset G$ are closed subgroups in a topological Abelian group 
$G$, then $\widehat{G_2/G_2}=G_1^\bot/G_2^\bot,$ where $\bot$ denotes the annihilator subgroup (see \cite{aga81}, Ch.3, \S2,
or Lemma 24.5 in \cite{hew6370}). 
Therefore the dual of the finite group $B(r)/B(\tau_-(r))$ is 
the group $B(\tau_-(r))^\bot/B(r)^\bot=\hat B(\tau_+(\tilde r))/\hat B(\tilde r),$ see \eqref{eq:BB} and Part (a). Equation 
\eqref{eq:nr} expresses the fact that the order of a finite Abelian group equals the order of the its dual group. 
\end{proof}

We can rewrite equations \eqref{eq:f1}-\eqref{eq:f3} as follows:
    \begin{equation}\label{eq:mu1}
     \left.\begin{array}{c}
\tilde\chi[B(r);\xi]=\mu(B_r)\chi[\hat B(\tilde r);\xi] \quad r\in\I_0 \\
\chi^\natural[\hat B(t);x]=\hat\mu(\hat B(t))\chi[B(t^\natural);x] \quad t\in\hat\I_0\\
\mu(B(r))\hat\mu(\hat B(\tilde r))=1, \quad \hat\mu(\hat B(t))\mu(B(t^\natural))=1.
      \end{array}\right\}
    \end{equation}

Now consider spheres in the groups $X$ and $\hat X$:
\begin{align*}
\gls{S}(r)&=\{x\in X: \rho(x)=r\}. \quad r\in\I,\\
\hat S(t)&=\{\xi\in\hat X: \hat\rho(\xi)=t\}, \quad t\in \hat\I.
\end{align*}
We have
\begin{gather}
S(0)=B(0),\;\;\hat S(0)=\hat B(0) \nonumber\\
S(r)=B(r)\backslash B(\tau_-(r)),\;\; \hat S(t)=\hat B(t)\backslash \hat B(\tau_-(t)),\label{eq:Sx}
\end{gather}
so
\begin{align}
\chi[S(r);x]&=\chi[B(r);x] -\chi[B(\tau_-(r));x] \label{eq:St}\\
\chi[\hat S(t);\xi]&=\chi[\hat B(t);\xi]-\chi[\hat B(\tau_-(t));\xi].\label{eq:Sf}
\end{align}
We will extend these relations to hold for $r=0, t=0$ as well, assuming that $B(\tau_-(0))=\hat B(\tau_-(0))=\emptyset .$
Expressions for the Fourier transforms of these functions follow immediately (note the use of \eqref{eq:+-}).
\begin{lemma} \label{lemma:fc}
\begin{align*}
\tilde\chi[S(r);\xi]&=\mu(B(r))\chi[\hat B(\tilde r);\xi]-\mu(B(\tau_-(r)))\chi[\hat B(\tau_+(\tilde r));\xi]\\
\chi^\natural[\hat S(t);x]&=\hat\mu(\hat B(t))\chi[B(t^\natural);x]-\hat\mu(\hat B(\tau_-(t)))\chi[B(\tau_+(t^\natural));x].
\end{align*}
\end{lemma}

Consider partitions of the groups $X$ and $\hat X$ into spheres:
\begin{equation}\label{eq:part}
X=\bigcup_{r\in \I} S(r), \quad \hat X=\bigcup_{t\in\hat\I} \hat S(t).
\end{equation}
We have
\begin{align*}
\chi[X;x]&=\sum_{r\in \I} \chi[S(r);x]=1 \text{ for all } x\in X\\
\chi[\hat X;\xi]&=\sum_{t\in \hat\I} \chi[\hat S(t);\xi]=1 \text{ for all }\xi\in \hat X.
\end{align*}
In the next theorem, which is the main result in this part, we establish the key properties of these partitions.
\begin{theorem}\label{thm:m1}
Partitions \eqref{eq:part} of the groups $X$ and $\hat X$ are symmetric spectrally dual in the sense of Definition \ref{def:sdp}. We have
\begin{align}
\tilde \chi[S(r);\xi]&=\sum_{b\in \hat \I} p_r(b)\chi[\hat S(b);\xi]\label{eq:sdf}\\
\chi^\natural [\hat S(t); x]&=\sum_{a\in \I} q_t(a) \chi[S(a);\xi]\label{eq:sdfb}
\end{align}
where
\begin{gather}
p_r(b)=\begin{cases} 0 &\text{if } b>\tau_+(\tilde r),\\
-\mu(B(\tau_-(r))) &\text{if } b=\tau_+(\tilde r),\\
\mu(B(r))-\mu(B(\tau_-(r))) &\text{if } 0\le b\le \tilde r
\end{cases} \label{eq:pik3}\\[.1in]
q_t(a)=\begin{cases} 0 &\text{if } a>\tau_+(t^\natural),\\
-\hat\mu(\hat B(\tau_-(t))) &\text{if } a=\tau_+(t^\natural)\\
\hat\mu(\hat B(t))-\hat\mu(\hat B(\tau_-(t))) &\text{if } 0\le a\le t^\natural.
\end{cases}\label{eq:qik3}
\end{gather}
(cf. \eqref{eq:pik1},~\eqref{eq:qik1}). Relations \eqref{eq:sdf} and \eqref{eq:sdfb} hold pointwise for all $\xi\in \hat X$
and $x\in X$.
\end{theorem}
\begin{proof} We have
$$
B(r)=\bigcup_{a\in \I: a\le r} S(a), \;\;\hat B(t)=\bigcup_{b\in \hat\I: b\le t} \hat S(b).
$$
Writing these relations in terms of the indicator functions, we obtain
\begin{align*}
\chi[B(r);x] &= \sum_{ a\in \I: a\le r} \chi[S(a);x]\\
\chi[\hat B(t);\xi] &=\sum_{ b\in \hat\I: b\le t} \chi[\hat S(b);\xi].
\end{align*}
Now \eqref{eq:sdf}-\eqref{eq:qik3} follow by using these expressions in Lemma~\ref{lemma:fc}. Consequently, the functions
constant on the blocks of the partition of $X$ into spheres \eqref{eq:part} have Fourier transforms that are constant on
the blocks of the dual partition. This implies that the partitions \eqref{eq:part} are spectrally dual. The symmetry is obvious
because $x$ and $-x$, or $\xi$ and $\xi^{-1}$ are contained in the same spheres in $X$ and $\hat X,$ respectively.
\end{proof}

In the next lemma we collect several useful properties of the eigenvalues $p_r(b)$ and $q_t(a)$.
\begin{lemma}\label{lemma:orth}
\begin{gather}
\sum_{b\in \hat\I_0} p_r(b)\hat\mu(\hat S(b))=0, \quad r>0 \label{eq:pb}\\
\sum_{b\in \hat\I_0}p_0(b)\hat\mu(\hat S(b))=\begin{cases} 0&\text{if }\mu(S(0))=0,\\
1&\text{if }\mu(S(0))>0
\end{cases} \label{eq:p0}\\
\sum_{a\in \I_0} q_t(a)\mu(S(a))=0, \quad t>0 \label{eq:qb}\\
\sum_{a\in \I_0} q_0(a)\mu(S(a))=\begin{cases} 0 &\text{if }\hat\mu(\hat S(0))=0\\
1&\text{if }\hat\mu(\hat S(0))>0
\end{cases}\label{eq:q0}
\end{gather}
\begin{align}
\sum_{b\in \hat\I_0} p_{r_1}(b)p_{r_2}(b)\hat\mu(\hat S(b))=\delta_{r_1,r_2}\mu(S(r_1)) \label{eq:o11}\\
\sum_{a\in\I_0} q_{t_1}(a)q_{t_2}(a)\mu(S(a))=\delta_{t_1,t_2}\hat\mu(\hat S(t_1)) \label{eq:o12}
\end{align}
\end{lemma}
\begin{proof} To prove \eqref{eq:pb}, we use \eqref{eq:pik3} as follows:
\begin{align*}
\sum_{b\in\hat\I_0} p_r(b)\hat\mu(\hat S(b))&=-\mu(B(\tau_-(r)))\hat\mu (\hat S(\tau_+(\tilde r)))\\[-.05in]
&\hspace*{1in}+
\Big(\mu(B(r))-\mu(B(\tau_-(r)))\Big)\sum_{b\in \hat\I_0: b\le \tilde r}\hat\mu(\hat S(b))\\
&=-\mu\big(B(\tau_-(r))\big)\big(\hat\mu(\hat B(\tau_+(\tilde r)))-\hat\mu(\hat B(\tilde r))\big)\\
&\hspace*{1in}+\Big(\mu(B(r))-\mu(B(\tau_-(r)))\Big)\hat\mu(\hat B(\tilde r))\\
&=-\mu\big(B(\tau_-(r)\big)\hat\mu\big(\hat B(\tau_+(\tilde r)))\big)+\mu(B(r))\hat\mu(\hat B(\tilde r))\\
&\stackrel{\text{(a)}}{=}-\mu\big(B(\tau_-(r))\big)\hat\mu\big(\hat B(\tilde \tau_-(r))\big)+\mu(B(r))\hat\mu(\hat B(\tilde r))\\
&\stackrel{\text{(b)}}{=}-1+1=0
\end{align*}
where for (a) we used \eqref{eq:+-} and for (b) Eq.~\eqref{eq:mu1}.
To prove \eqref{eq:p0}, observe that \eqref{eq:pik3} implies that $p_0(b)=\mu(S(0))=\mu(\{0\})$ for all $b\in \I.$ This
implies the first case in \eqref{eq:p0}. Further, if $\mu(\{0\})>0,$ then $X$ is discrete and $\hat X$ is compact, so
the sum in \eqref{eq:p0} equals $\mu(\{0\})\hat\mu(\hat X)=1$ because of \eqref{eq:PW}. This proves \eqref{eq:p0}.
The proof of \eqref{eq:qb} and \eqref{eq:q0} is completely analogous and will be omitted.

Equations \eqref{eq:o11},\,\eqref{eq:o12} form a special case of the orthogonality relations \eqref{eq:pipj},~\eqref{eq:qiqj}.
We will give an independent proof that relies on the explicit formulas \eqref{eq:pik3},~\eqref{eq:qik3}. Let $r_1\ne r_2$ and $0\le r_1<\tau_+(\tilde r_2)\le \tilde r_1.$
Then \eqref{eq:gl} implies that $\tilde r_1>\tilde r_2$, so the sum in \eqref{eq:o11} extends only to the region $0\le b\le \tilde r_2$
where $p_{r_1}(b)$ is independent of $b.$ Therefore, by \eqref{eq:pb}
$$
\sum_{b\in\hat\I_0} p_{r_1}(b)p_{r_2}(b)\hat\mu(\hat S(b))=
\Big(\mu(B(r_1))-\mu(B(\tau_-(r)))\Big)\sum_{b\in \hat\I_0} p_{r_2}(b)\hat\mu(\hat S(b))=0.
$$
Now let $r_1=r_2=r,$ then
\begin{align*}
\sum_{b\in \hat\I_0} p_r(b)^2\hat\mu(\hat S(b))&= \mu(B(\tau_-(r)))^2\hat\mu(\hat S(\tau_+(\tilde r)))\\[-.1in]
&\hspace*{1in} +\Big(\mu(B(r))-\mu(B(\tau_-(r)))\Big)^2\sum_{0\le b\le r}\hat\mu(\hat S(b))\\
&=\mu(B(\tau_-(r)))^2 \Big( \hat\mu(\hat B(\tau_+(\tilde r)))-\hat\mu(\hat B(\tilde r))\Big)\\
&\hspace*{1in}+\Big(\mu(B(r))-\mu(B(\tau_-(r)))\Big)^2\hat\mu(\hat B(\tilde r))\\
&= \mu(B(\tau_-(r)))^2 \hat\mu(\hat B(\tau_+(\tilde r)))-\mu(B(\tau_-(r)))^2\hat\mu(\hat B(\tilde r))
+\mu(B(r))^2\hat\mu(\hat B(\tilde r))\\
&\hspace*{.2in}+\mu(B(\tau_-(r)))^2\hat\mu(\hat B(\tilde r))
-2\mu(B(r))\mu(B(\tau_-(r)))\hat\mu(\hat B(\tilde r))\\
&=\mu(B(\tau_-(r)))+\mu(B(r))-2\mu(B(\tau_-(r)))\\
&=\mu(B(r))-\mu(B(\tau_-(r)))=\mu(S(r)).
\end{align*}
where we have used \eqref{eq:+-} and \eqref{eq:mu1}. This proves \eqref{eq:o11}. The proof of \eqref{eq:o12}
is essentially the same.
\end{proof}

In the following theorem, which is one of the main results of the paper, we give a construction of dual translation schemes and compute their
spectral parameters and intersection numbers.
\begin{theorem}\label{thm:m2} Let $X$ be a zero-dimensional compact or locally compact Abelian group and let $\hat X$ be its dual group.
Let $\cR=\{R_r, r\in \I\}, \hat\cR=\{\hat R_t,t\in \hat \I\},$ where
\begin{equation}\label{eq:Nr}
R_r=\{(x,y)\in X\times X: x-y\in S(r)\}, \quad r\in \I
\end{equation}
\begin{equation}\label{eq:Nrd}
\hat R_t=\{(\phi,\xi)\in \hat X\times\hat X: \phi\xi^{-1}\in \hat S(t)\}, \quad t\in \hat\I.
\end{equation}
Then $\cX=(X,\mu,\cR)$ and $\widehat\cX=(\hat X,\hat\mu,\hat\cR)$
form a pair of symmetric, mutually dual translation association schemes in the sense of Def.~\ref{def1}. The spectral parameters of these schemes are given
by \eqref{eq:pik3} and \eqref{eq:qik3}. The intersection numbers of $\cX$ are given by
\begin{equation}\label{eq:pijk1}
p_{r_1,r_2}^{r_3}=\begin{cases}0 &\text{if }\lambda=1\\[.05in]
\mu(S(r^\ast)) &\text{if }\lambda=2\\[.05in]
\Big[\,|B(r^\ast)/B(\tau_-(r^\ast))|-2\Big]\,\mu\big(B(\tau_-(r^\ast))\big) &\text{if }\lambda=3
\end{cases}
\end{equation}
where $r^\ast:=\min(r_1,r_2,r_3),$ and $\lambda$ denotes the number of times $\max(r_1,r_2,r_3)$ appears among $\{r_1,r_2,r_3\}.$
Likewise, the intersection numbers of $\widehat\cX$ are given by
\begin{equation}\label{eq:pijk2}
\hat p_{t_1,t_2}^{t_3}=\begin{cases} 0 &\text{if }\lambda=1\\[.05in]
\hat \mu(\hat S(t^\ast)) &\text{if }\lambda=2\\[.05in]
\Big[\,|\hat B(t^\ast)/\hat B(\tau_-(t^\ast))|-2\Big]\,\hat \mu\big(\hat B(\tau_-(t^\ast))\big) &\text{if }\lambda=3
\end{cases}
\end{equation}
where $t^\ast$ and $\lambda$ are defined analogously.
\end{theorem}
\begin{proof} Everything except the expressions for the intersection numbers follows immediately from
Theorems \ref{thm:SDP} and \ref{thm:m1}.
We will compute $p_{r_1,r_2}^{r_3}$ starting from \eqref{eq:ser1}.
We have
\begin{equation}\label{eq:r1r2r3}
p_{r_1,r_2}^{r_3}=\frac 1{\mu(S(r_3))}\sigma(r_1,r_2,r_3),
\end{equation}
where
\begin{equation}\label{eq:rrr}
\sigma(r_1,r_2,r_3)=\sum_{b\in \hat\I_0} p_{r_1}(b)p_{r_2}(b)p_{r_3}(b)\hat\mu(\hat S(b))
\end{equation}
(cf. \eqref{eq:sigma}). Since the value of $\sigma$ does not depend on the order of the arguments, let us assume that
$0\le r_1\le r_2\le r_3.$ We will consider the following three cases grouped according to the multiplicity of the largest radius:

(i) $0\le r_1\le r_2<r_3$

(ii) $0\le r_1<r_2=r_3=r$

(iii) $0<r_1=r_2=r_3:=r$ (if $r=0$ then obviously $p_{00}^0=0$).

In Case (i) we have $\tilde r_3<\tilde r_2\le \tilde r_1$, so the sum in \eqref{eq:rrr} extends to the region $0\le b\le \tau_+(\tilde r_3)\le\tilde r_2$.
From \eqref{eq:pik3}, in this region the coefficients $p_{r_1}(b)$ and $p_{r_2}(b)$ are constant, so we obtain
\begin{align*}
\sigma(r_1,r_2,r_3)&=\Big[ \mu(B(r_1))-\mu(B(\tau_-(r_1)))\Big]\Big[\mu(B(r_2))-\mu(B(\tau_-(r_2)))\Big]\sum_{b\in\hat \I_0}
p_{r_3}(b)\hat\mu(\hat S(b))\\
&\stackrel{\eqref{eq:pb}}=0
\end{align*}

In Case (ii) we have $\tilde r<\tilde r_1,$ so the sum in \eqref{eq:rrr} extends to the region 
$0\le b\le \tau_+(\tilde r)\le \tilde r_1$ in which
$p_{r_1}(b)$ is independent of $b$. Therefore, on account of \eqref{eq:o11} we obtain
\begin{align*}
\sigma(r_1,r,r)=\Big[\mu(B(r_1))-\mu(B(\tau_-(r_1)))\Big]\sum_{b\in\hat \I_0}
p_{r}(b)^2\hat\mu(\hat S(b))=\mu(S(r_1))\mu(S(r)).
\end{align*}

In Case (iii) we have
\begin{align*}
\sigma(r,r,r)&=\sum_{b\in\hat\I_0} p_r(b)^3\hat\mu(\hat S(b))=-\mu(B(\tau_-(r)))^3\Big[\hat\mu(\hat B(\tau_+(\tilde r)))
-\hat\mu(\hat B(\tilde r))\Big]\\[-.05in]
&\hspace*{.5in}
+\Big[\mu(B(r))-\mu(B(\tau_-(r)))\Big]^3\hat\mu(\hat B(\tilde r))\\
&\stackrel{\text{(a)}}=-\mu(B(\tau_-(r)))^3 \hat\mu(\hat B(\widetilde{\tau_-(r)}))+\mu(B(r))^3\hat\mu(\hat B(\tilde r))\\
&\hspace*{.5in} -3\mu(B(r))^2\mu(B(\tau_-(r)))\hat\mu(\hat B(\tilde r))+3\mu(B(r))\mu(B(\tau_-(r)))^2\hat\mu(\hat B(\tilde r))\\
&\stackrel{\text{(b)}}=-\mu(B(\tau_-(r)))^2+\mu(B(r))^2-3\mu(B(r)\mu(B(\tau_-(r)))+3\mu(B(\tau_-(r)))^2\\
&=\Big[\mu(B(r))-2\mu(B(\tau_-(r)))\Big]\Big[\mu(B(r))-\mu(B(\tau_-(r)))\Big]\\
&=\Big[\mu(B(r))-2\mu(B(\tau_-(r)))\Big]\mu(S(r))\\
&=\Big[|B(r)/B(\tau_-(r))|-2\Big]\mu(B(\tau_-(r))\mu(S(r)),
\end{align*}
where $|B(r)/B(\tau_-(r))|$ is the index of the subgroup $B(\tau_-(r))$ in $B(r).$ Here (a) relies on \eqref{eq:+-}
and (b) uses \eqref{eq:mu1}.

Together with \eqref{eq:r1r2r3},\,\eqref{eq:rrr}, these calculations establish the claimed expression for $p_{r_1,r_2}^{r_3}.$
The expression for $\hat p_{t_1,t_2}^{t_3}$ is proved by analogous arguments
starting from \eqref{eq:ser2}.
\end{proof}

\vspace*{.05in} {\em Example:} Following \cite{mar99,skr01}, consider a particular case of schemes over finite groups. Namely, suppose that $X=\integers_q^n$ is the group of $n$-strings over the additive
group of integers mod $q$, and the subgroups $X_r, r=1,\dots,n$ are formed of the strings with $r$ first coordinates equal to 0.
In particular, $X_n=\{(00\dots0)\}.$ 
We observe that $\tau_-(r)=r-1, r>0; \tau_+(r)=r+1, r<n,$ and for $r>0,$ $|S(r)|=(q-1)q^{r-1},$ $|B_{r}|=q^r,$ while 
$|S(0)|=|B(0)|=1.$ 

Let us use Theorem \ref{thm:m2} to compute the intersection numbers (see \cite{mar99}, Lemma 1.1).
 If $i,j,k\ne 0,$ then using \eqref{eq:pijk1} we obtain
  \begin{equation*}
p_{ij}^{k}=\begin{cases}0 &\text{if }i\ne j,j\ne k,i\ne k\\[.05in]
 (q-1)q^{r^\ast-1}&\text{if }i<j=k\text{ or }j<i=k\text{ or }k<i=j\\[.05in]
(q-2)q^{i-1}&\text{if }i=j=k.
\end{cases}
\end{equation*}
We also find directly that $p_{00}^0=p_{0,i}^i=p_{i,0}^i=1,$ and $p^0_{ij}=\delta_{ij}(q-1)q^{i-1},$
where $p^0_{ii}=\mu_i$ is the $i$th valency of $\cX.$ 

Using \eqref{eq:pik3} and \eqref{eq:dr}, we can compute the eigenvalues of this scheme. We obtain
    \begin{equation*}
    p_i(j)=\begin{cases} 0 &\text{if } j>n-i+1\\
         -q^{i-1} &\text{if } j=n-i+1\\
         q^i-q^{i-1} &\text{if } 0\le j\le n-i
         \end{cases}, \quad i>0
         \end{equation*}
recovering a result in \cite{mar99}, Lemma 1.4 (see also \cite{dou02}, Theorem 3.1). 
Again we must separately consider the boundary case $i=0$, but it is
easily seen from \eqref{eq:pij} and \eqref{eq:I} that $p_0(j)=1$ for all $j.$

Finally, since the association scheme $\cX$ is self-dual, i.e., $\cX\cong \widehat\cX,$ 
we have that $p_i(j)=q_i(j)$ for all $i,j,$ and $\hat p_{ij}^k=p_{ij}^k$ for all $i,j,k.$

%
\section{Metric schemes}\label{sect:metric}
Association schemes constructed in Theorem \ref{thm:m2} belong to an important class of the so-called metric schemes for which the classes $R_i$ are formed of pairs of points separated by the same distance. In the case of finite sets such schemes are well known \cite{del73a,bro89}. For infinite sets we need to make some adjustments; in particular, it
will turn out that our metric schemes are non-polynomial; see Sect.~\ref{sect:npol}.

\subsection{Geometric view}
Let $X$ be a metric space with metric $\rho$. The base of the corresponding metric topology consists of all open metric balls.
As before, we assume that the topology satisfies the second countability axiom which in this case is the same as the separability of $X$
\cite[p.120]{Kelley55}. As before, we assume that the measure on $X$ is countably additive and is defined on Borel subsets of the topological
space $X$. If $X$ is an Abelian group then we assume that the metric is invariant, i.e., $\rho(x,y)=\rho(x-y)$, and that $\mu$
is the Haar measure on $X$ (as before, we call both $\rho(x,y)$ and $\rho(x)=\rho(x,0)$ a metric; cf. \eqref{eq:vl}).

Consider the following partition $\cR=\{R_r,r\in\I\}$ of $X\times X$:
\begin{gather}
X\times X=\cup_{r\in \I} R_r\label{eq:X}\\
R_r=\{(x,y)\in X\times X: \rho(x,y)=r\}, \quad r\in \I, \label{eq:Rr}
\end{gather}
where
\begin{equation}
\I=\{r\ge 0: r=\rho(x,y), x,y\in X\}.
\end{equation}
Clearly, for every $x_0\in X$ this partition gives a partition of $X$ into spheres with center at $x_0$:
\begin{gather}
X=\cup_{r\in\I(x_0)} S_{x_0}(r)\\
S_{x_0}(r)=\{y\in X: \rho(x_0,y)=r\}, \quad r\in \I(x_0),
\end{gather}
where $\I(x_0)=\{r\ge 0: \rho(x_0,y)=r, y\in X\}\subset \I.$ If there $X$ is a homogeneous space of some group, for instance, 
an Abelian group acting on itself, then $\I(x_0)=\I$ for all $x_0\in X.$

We can give the following definition: If the partition \eqref{eq:X} forms an association scheme in the sense of Def.~\ref{def1},
then $(X,\mu,\cR)$ is called a {metric scheme}. 
However this definition is too general to be useful: for instance, in Def.~\ref{def1}
we have additionally assumed that $\I$ is at most countably infinite. Therefore, let us adopt

{\em Condition C:} The set of values of the metric $\rho$ is closed and at most countably infinite.

\vspace*{.05in}This condition is rather strong. For instance, it implies that $X$ is zero-dimensional. Indeed, for every $x_0$
consider the function
$f_{x_0}: X\to \I$ given by $f_{x_0}(x)=\rho(x,x_0).$ This function is piecewise constant and continuous in the metric topology of $X.$
This implies that $\rho(x_0,x)$ is constant on the connected component of $x_0,$ and therefore is equal to zero on it.
Since $\rho$ is a metric, the connected component consists just of $x_0$, so $X$ is zero-dimensional.

On account of the above discussion, we define a {\em metric scheme} \index{association scheme!metric} as a triple $\cX=(X,\mu,\cR),$ where $X$ and $\mu$ are as above, $\cR$ is defined by
\eqref{eq:Rr}, and the value
\begin{equation}\label{eq:ms}
\gls{pr1r2r3}=\mu\big\{z\in X: \rho(z,x)=r_1, \rho(z,y)=r_2; \;\rho(x,y)=r_3\big\}
\end{equation}
depends only on $r_1,r_2,r_3$ but not on the choice of $x,y\in X.$ The other conditions of Def.~\ref{def1} are satisfied since $\rho$
is a metric: for instance, a metric scheme is always symmetric.

If $X$ is an Abelian group, then we can write \eqref{eq:ms} as
\begin{align}
p_{r_1,r_2}^{r_3}&=\mu\big\{z\in X:\rho(z)=r_1,\rho(z-y)=r_2;\; \rho(y)=r_3\big\} \nonumber\\
&=\mu\big\{z\in S(r_1): \rho(z-y)=r_2;\; y\in S(r_3)\big\}\label{eq:SS}
\end{align}
where $S(r)$ is a sphere of radius $r$ around 0. Suppose that
$$
p_{r_1,r_2}^{r_3}>0,
$$
where w.l.o.g. we can assume that
\begin{equation}\label{eq:iq}
0\le r_1\le r_2\le r_3
\end{equation}
(see \eqref{eq:iij}, \,\eqref{eq:sigma}). This means that the space $X$
contains triangles with sides $r_1,r_2,r_3,$ and so
\begin{equation}\label{eq:triangle}
r_2-r_1\le r_3\le r_2+r_1.
\end{equation}
These inequalities form necessary conditions for the positivity to hold. They are valid for any metric scheme and are well known
for schemes on finite sets \cite[p.58]{bro89}.

For non-Archimedean metrics, Eq.\eqref{eq:triangle} together with the ultrametric triangle inequality \eqref{eq:um1} implies that
\begin{equation}\label{eq:na}
0\le r_1\le r_2=r_3.
\end{equation}
These relations (assuming \eqref{eq:iq}) form necessary conditions for the positivity of
intersection numbers of a metric scheme on the space with a non-Archimedean norm. They are well known in non-Archimedean
geometry where they say that all triangles are isosceles (or equilateral), with at most one short side; e.g. \cite[p.71]{Robert00}.

In this section we make several observations implied by the definition of the metric scheme. Let us begin with a
simple geometric proof of the expressions for the intersection numbers which were
earlier obtained by a direct calculation.
\begin{proposition} \label{prop:inm}
The intersection numbers of an Abelian metric scheme $\cX$ and its dual scheme $\widehat\cX$ are given by \eqref{eq:pijk1},
\eqref{eq:pijk2}.
\end{proposition}
\begin{proof}
The first equality in \eqref{eq:pijk1},\,\eqref{eq:pijk2} follows directly from \eqref{eq:na}.
To prove the second equality, observe that
\begin{equation}\label{eq:rr}
p_{r_1,r}^r=\mu\big\{z\in X: \rho(z,x)=r_1,\rho(z,y)=r; \;\rho(x,y)=r \big\}.
\end{equation}
From the triangle inequality \eqref{eq:um1} we obtain 
$$
\rho(z,y)=\max\{\rho(z,x),\rho(y,x)\} =\rho(y,x)=r
$$
Thus, the condition $\rho(z,y)=r$ holds for all $z$ with $\rho(x,z)=r_1$, and then the condition $\rho(x,y)=r$
places no constraints on $z$. Thus, we can rewrite \eqref{eq:rr} as follows:
$$
p_{r_1,r}^r=\mu\{z\in X:\rho(z,x)=r_1\}=\mu(S_x(r_1))
$$
where $\mu(S_x(r_1))$ is a sphere of radius $r_1$ with center at $x.$ Since $X$ is an Abelian group and $\mu$ is an
invariant measure, we obtain the second equality in \eqref{eq:pijk1}. The proof of the second case of \eqref{eq:pijk2} is entirely similar.

To prove the third case, we need to consider in detail the structure of $\rho$-spheres in $X$.
Again using the invariance of $\mu$, write \eqref{eq:SS} as
\begin{equation}\label{eq:prrr}
p_{r,r}^r=\mu\big\{z\in S(r): \rho(z-y)=r, y\in S(r)\big\}.
\end{equation}
Here the sphere $S(r)=S_0(r)$ can be written as \eqref{eq:Sx}
\begin{equation}\label{eq:Sx1}
S(r)=B(r)\backslash B(\tau_-(r)),
\end{equation}
where $B(r)$ is a metric ball of radius $r$ centered at 0 and $B(\tau_-(r))$ is a concentric ball of radius that directly precedes
$r$ in the natural ordering of $\I.$ 

The subgroup $B(r)$ can be written as a union of disjoint cosets of $B(\tau_-(r))$:
\begin{align}
B(r)&=\bigcup_{0\le i\le n(r)-1} \Phi_i(r),\label{eq:Phi2}\\
\Phi_i(r)&=B(\tau_-(r))+z_{i,r}, \quad z_{i,r}\in B(r)/B(\tau_-(r)), \label{eq:Phi1}
\end{align}
where $z_{i,r}, i=0,1,\dots,n(r)-1$ is a complete system of representatives of the cosets and $n(r)$ is defined in \eqref{eq:nr}. We will assume that $z_{0,r}=0.$ 
Expressions \eqref{eq:Sx1}-\eqref{eq:Phi1} imply the following partition of
the sphere into cosets of the group $B(\tau_-(r)):$
  \begin{equation}\label{eq:sph}
S(r)=\bigcup_{1\le i\le n(r)-1} \Phi_i(r).
  \end{equation}
We claim that if $z\in \Phi_i(r), y\in \Phi_j(r), 1\le i,j \le n(r)-1,$ then
\begin{equation}
\begin{array}{l}
\rho(z-y)=r \quad\text{if }i\ne j \\
\rho(z-y)\le \tau_-(r)<r \quad\text{if }i=j. \end{array}\label{eq:z-y}
\end{equation}
Indeed, the element $z-y$ is contained in the coset $B(\tau_-(r))+z_{i,r}-z_{j,r}$, and $z_{i,r}-z_{j,r}=0$ if and only if $i=j$, while
if $i\ne j$, then $z_{i,r}-z_{j,r}=z_{l,r}$ for some coset representative $z_{l,r}, 1\le l\le n(r)-1.$

Now return to \eqref{eq:prrr} and note that $y\in S(r)$ implies that $y\in \Phi_l(r)$ for some $l\in\{1,\dots,n(r)-1\},$
and the same is true for $z,$ namely $z\in \Phi_i(r)$ for some $i\in\{1,\dots,n(r)-1\}.$ On account of \eqref{eq:z-y}, we can write
$$
\big\{z\in S(r): \rho(z-y)=r, y\in S(r)\big\}=\bigcup_{\begin{substack}{0\le i\le n(r)-1\\i\ne l}\end{substack}}\Phi_i(r).
$$
Finally, since the measure of each coset is the same and equals $\mu(B(\tau_-(r))),$ we obtain
  \begin{equation}\label{eq:cm}
p_{r,r}^r=(n(r)-2)\mu(B(\tau_-(r))),
  \end{equation}
which is exactly the third case of \eqref{eq:pijk1}.
Again the proof of the corresponding case in \eqref{eq:pijk2} is entirely similar.
\end{proof}

Note that the above proof, in particular, the arguments related to \eqref{eq:cm}, enable one to state several claims about
spheres in a group with a non-Archimedean metric which may be of independent interest.
\begin{proposition} (a) The group $X$ contains an equilateral triangle with side $r>0$ if and only if the index $n(r)$ of the
group $B(\tau_-(r))$ in $B(r)$ is greater than 2. If $x$ and $y$ are the two fixed vertices of such a triangle, then the
Haar measure of the set of third vertices equals $p_{r,r}^r$ given in \eqref{eq:cm}.

(b) The diameter of the sphere $S(r)\subset X$ of radius $r$ equals
$$
\text{diam\,} S(r)=
\begin{cases} r &\text{if } n(r)>2\\ \tau_-(r) &\text{if }n(r)=2.
\end{cases}
$$
This implies that the diameter of the sphere is strictly less than its radius if and only if the index $n(r)=2.$
\end{proposition}

\subsection{On non-polynomiality of metric schemes on zero-dimensional groups}\label{sect:npol}
In the finite case polynomial schemes are well-studied \cite{del73a,ban84,bro89,god93}; in particular, it is a standard fact that finite metric schemes are $P$-polynomial. In this section we address the question of polynomiality for the metric schemes on zero-dimensional groups.

First let $\cX$ be a symmetric scheme on $X$ with a finite number of classes $R_0,R_1,\dots,R_d,$ intersection
numbers $p_{ij}^k, i,j,k=0,1,\dots,d$ and adjacency matrices $A_0,A_1,\dots, A_d.$ The scheme is called $P$-polynomial
if there exist polynomials $v_i(z)$ of degree $i$ such that $A_i=v_i(A_1), i=0,1,\dots,d.$ Let $\rho(x,y),x,y\in X$ 
be defined by 
  $$
  \rho(x,y)=i\quad\text{if}\quad(x,y)\in R_i,\;i=0,1,\dots,d.
  $$
  It is clear that $\rho(x,y)$ is symmetric and $\rho(x,y)=0$ if and only if $(x,y)\in R_0.$
If in addition the function $\rho(x)$ satisfies the triangle inequality, then it forms a metric on $X$, and the scheme $\cX$ is called 
{\em metric}.

The triangle inequality implies that if the intersection numbers $p_{ij}^k\ne 0$ then $|i-j|\le k\le i+j$ (cf. also \eqref{eq:triangle}).
The metric is called {\em nondegenerate} if 
  \begin{equation}\label{eq:ndm}
  p_{1,i}^{i+1}\ne 0 \quad\text{for }i=0,1,\dots d-1.
  \end{equation}

\begin{theorem}[Delsarte; see {\cite[Theorem 5.6]{del73a}}, {\cite[Prop.~2.7.1]{bro89}}]  A symmetric scheme $\cX$ with
a finite number of classes is $P$-polynomial if and only if it is metric with a nondegenerate metric $\rho.$
\label{thm:Pp}\end{theorem}

\vspace*{.05in} Geometrically conditions \eqref{eq:ndm} mean that $X$ contains triangles with sides $1,i,i+1$ for $i=0,1,\dots,d-1;$
see the definition of the intersection numbers \eqref{eq:pijk}.
We can say that the metric $\rho$ is strictly Archimedean, while all the non-Archimedean metrics are degenerate because for them
all the triangles are isosceles; see \eqref{eq:na}. Therefore, the Delsarte theorem implies that all the finite metric schemes with a
non-Archimedean metric are non-polynomial.

Now let us consider schemes on zero-dimensional Abelian groups. Let $X$ be a countable discrete Abelian group with a countable chain
of nested subgroups
  \begin{equation}\label{eq:lcc}
  X\supset\dots\supset X_n\supset\dots\supset X_1\supset X_0=\{0\},
  \end{equation}
and $X=\cup_{j\ge 0}X_j.$ Define a metric $\rho(x,y)=\rho(x-y)$ on $X$ by the formula
  \begin{equation}\label{eq:mrm}
  \rho(x)=\min\{j\in\naturals_0: x\in X_j\},
  \end{equation}
cf. \eqref{eq:dmd}. This metric is clearly non-Archimedean.
Spheres in this metric define on $X$ a symmetric association scheme $\cX$ according to \eqref{eq:X},\eqref{eq:Rr}. This scheme
has countably many classes $R_i,i\in\naturals_0$ and intersection numbers given in Proposition \ref{prop:inm}. Let $A_i,i\in\naturals_0$
be the corresponding sequence of adjacency matrices.

Terminating the chain \eqref{eq:lcc} at some $n\ge 2$, we obtain a finite chain of nested subgroups
  $$
  X_n\supset\dots\supset X_1\supset X_0=\{0\}
  $$
of a finite Abelian group $X_n.$ It is easy to see that spheres in the metric $\rho$ restricted to $X_n$ define a subscheme $\cX_n\subset \cX$ with 
finitely many classes $R_0,R_1,\dots,R_n$ and adjacency matrices $A_0,A_1,\dots,A_n.$ 

Now let us assume that the metric scheme $\cX$ on $X$ is $P$-polynomial, i.e., there exists a countable sequence of polynomials
$v_i(z), i\in\naturals_0$ such that $\deg v_i=i$ for all $i.$ Then all the finite subschemes $\cX_n,n\ge 2$ with the non-Archimedean 
metric \eqref{eq:mrm} must be polynomial since $A_i=v_i(A_1),i=0,1,\dots,n.$ However, as we saw earlier, this contradicts Theorem \ref{thm:Pp}, and so the scheme $\cX$ on a discrete group $X$ with the metric \eqref{eq:mrm} is non-polynomial.

Finally, as far as non-discrete groups are concerned, it makes no sense to address the question of polynomiality because in this
case the set of radii $\I_0$ has an accumulation point $r=0,$ and there does not exist a sphere $S(r)$ and a class $R_r$
that immediately follow the sphere $S(0)$ and the corresponding class $R_0.$

%
\section{Nonmetric schemes on zero-dimensional Abelian groups}\label{sect:nonmetric}
In this section we present a construction of schemes on locally compact Abelian zero-dimensional groups based on 
partitions with blocks indexed by parameters other than the distance. 
Recall that the topology of the group $X$ is defined by a chain of nested subgroups \eqref{eq:chain},~\eqref{eq:dc}, 
and that the subgroups in these chains form metric balls $B(r)$ \eqref{eq:b}, \eqref{eq:^b}, where the set of values of $r$ 
is determined by the metric (see \eqref{eq:rr1},~\eqref{eq:rr2}). In this section we switch to notation $\gls{frakR}, \hat\R$
for the sets of radii because their values no longer index the classes of the scheme (the classes are indexed by two parameters as
explained below after \eqref{eq:^Phi}). As before, introduce the notation
$\R_0=\{r:\mu(B(r))>0\}$ and the analogous notation $\hat \R_0.$

Our starting
point is Eqns.~\eqref{eq:part} and \eqref{eq:sph} that jointly describe a {\em partition} of the group {\em into metric balls} $\Phi_i(r), r\in\R, 1\le i\le n(r)-1:$ \index{group!zero-dimensional!partition into balls}
  \begin{equation}\label{eq:nmp}
  X=\bigcup_{r\in \R}\bigcup_{1\le i\le n(r)-1} \Phi_i(r),
  \end{equation}
where 
\begin{align}
   \Phi_i(r)=B(\tau_-(r))+z_{i,r}, \quad &z_{i,r}\in B(r)/B(\tau_-(r))\label{eq:Phi}\\
  0\le i\le &n(r)-1, \quad n(t)=|B(r)/B(\tau_-(r))| \nonumber
 \end{align}  
where the complete set of coset representatives $(z_{i,r})$ is chosen so that $z_{i,r}=0.$ Note that $i$ in \eqref{eq:nmp}
varies from $1$ to $n(r)-1$ because $i=0$ in \eqref{eq:Phi1} corresponds to the subgroup $B(\tau_-(r))$ which is not a part
of the sphere $S(r).$

In a similar way, let us introduce a partition of $\hat X$ as follows:
  \begin{equation}\label{eq:^nmp}
  \hat X=\bigcup_{t\in \hat\R}\bigcup_{1\le i\le \hat n(t)-1} \hat\Phi_i(t),
  \end{equation}
where
  \begin{align}
  \hat\Phi_i(t)=\hat B(\tau_-(t))\theta_{i,t}, \quad &\theta_{i,t}\in \hat B(t)/\hat B(\tau_-(t)), \label{eq:^Phi}\\
  0\le i\le &\hat n(t)-1, \quad \hat n(t)=|\hat B(t)/\hat B(\tau_-(t))| \nonumber
  \end{align}
and the complete set of coset representatives $(\theta_{i,t})$ is chosen so that $\theta_{0,t}=1$.
We will show that the partitions \eqref{eq:nmp} and \eqref{eq:^nmp} are spectrally dual and therefore give rise to a pair of dual translation schemes on $X$ and $\hat X.$ Accordingly, the classes of the schemes are indexed by pairs of the form $(i,r),$ where $r\in\R$ and $i\in
\{0,\dots,n(r)-1\}.$

\vspace*{.1in}
The group $B(r)/B(\tau_-(r))$ can be identified with the set $z_{i,r}, 0\le i\le n(r)$ (see \eqref{eq:Phi}) and the group
$B(r)/B(\tau_-(r))$ with the set $\theta_{i,\tau_+(\tilde r)}, 0\le i\le \hat n(\tau_+(\tilde r))-1.$ Denote by $\omega_{ij}$
the value of the character $\overline\theta_{j,\tau_+(\tilde r)}$ on $z_{i,r}:$
   \begin{equation}\label{eq:R}
   \gls{omegaij}=\overline\theta_{j,\tau_+(\tilde r)}(z_{i,r}), \quad r\in \R
   \end{equation}
(complex conjugation is added for notational convenience in the calculations below). Orthogonality of characters for finite groups
in this case takes the following form:
   \begin{align}
   \sum_{i=0}^{n(r)-1}\omega_{ij}(r)\overline\omega_{ik}(r)=n(r)\delta_{jk}\quad \text{and}\quad
   \sum_{j=0}^{\hat n(\tau_+(\tilde r))-1}\omega_{ij}(r)\overline\omega_{kj}(r)=\hat n(\tau_+(\tilde r))\delta_{ik}.\label{eq:oo1}
   \end{align}
In other words, the square matrix $\omega(r)=(\omega_{ij})$ satisfies the relations $\omega(r)\omega^\ast(r)=n(r)I$, i.e.,  the matrix $n(r)^{-\half}\omega(r)$ is unitary.

The dual picture is analogous. We have $(\hat B(t)/\hat B(\tau_-(t)))^\wedge=B(\tau_+(t^\natural))/B(t^\natural)$
(again see \eqref{eq:BB} and \eqref{eq:+-}) and $\hat n(t)=n(\tau_+(t^\natural)).$ Identify the group 
$\hat B(t)/\hat B(\tau_-(t))$ with the set $\theta_{i,t}, 0\le i\le \hat n(t)-1$ and the group 
$B(\tau_+(t^\natural))/B(t^\natural)$ with the set $z_{i,\tau_+(t^\natural)}, 0\le i\le n(\tau_+(t^\natural))-1.$
Letting $\hat\omega_{ij}(t)$ be the value of the character $\theta_{i,t}$ on the element $z_{j,\tau_+(t^\natural)},$
   \begin{equation}\label{eq:^R}
     \hat\omega_{ij}(t)=\theta_{i,t}(z_{j,\tau_+(t^\natural)}), \quad t\in\hat\R,
  \end{equation}
we obtain orthogonality relations
  \begin{align}
  &\sum_{i=0}^{\hat n(t)-1}  \hat\omega_{ji}(t)\overline{\hat\omega_{ki}(t)}=\hat n(t)\delta_{jk}\quad\text{and}\quad
  \sum_{j=0}^{n(\tau_+(t^\natural))-1} \hat\omega_{ji}(t)\overline{\hat\omega_{jk}(t)}=n_{\tau_+}(t^\natural)\delta_{ik}.
    \label{eq:oo}
  \end{align}
In other words, the $\hat n(t)\times \hat n(t)$ matrix $\hat\omega(t)=(\hat\omega_{ij}(t))$ satisfies the relations
$\hat\omega(t)\hat\omega(t)^\ast=\hat\omega(t)^\ast\hat\omega(t)=\hat n(t) I.$ 
Note also the following relations:
  \begin{align*}
  \hat\omega(t)&=\omega(\tau_+(t^\natural))^\ast=\omega(\tau_-(t)^\natural)^\ast, \quad t\in\hat \R\\\
  \omega(r)&=\hat\omega(\tau_+(\tilde r))^\ast=\hat\omega(\widetilde{\tau_-(r)})^\ast, \quad r\in\R,
  \end{align*}
obtained by combining \eqref{eq:R} and \eqref{eq:^R} with \eqref{eq:+-}.

The Fourier transforms of the indicators of the balls are found in the following lemma. 
\begin{lemma}\label{lemma:cfs} Let $\Phi_i(r), 0\le i\le n(r)-1$ and $\hat\Phi_i(t), 0\le i\le \hat n(t)-1$ be defined by \eqref{eq:Phi} and \eqref{eq:^Phi}, respectively.
Then
  \begin{align}\label{eq:cP}
   \tilde\chi[\Phi_i(r);\xi]=\sum_{t\in \hat \R}\sum_{j=1}^{\hat n(t)-1} p_{r,i}(t,j)\chi[\hat \Phi_j(t);\xi]
  \end{align}
where
  \begin{equation}\label{eq:pri}
   p_{r,j}(t,j)=\begin{cases} 0&\text{if }t>\tau_+(\tilde r)\\
        \mu(B(\tau_-(r)))\omega_{ij}(r) &\text{if }t=\tau_+(\tilde r)\\
        \mu(B(\tau_-(r))) &\text{if } t\le \tilde r,
     \end{cases}
  \end{equation}
and
  \begin{equation}\label{eq:^cP}
  \chi^\natural[\hat\Phi_j(t);x]=\sum_{r\in\R}\sum_{i=1}^{n(r)-1} q_{t,j}(r,i)\chi[\Phi_i(r);x],
  \end{equation}
where
  \begin{equation}\label{eq:qtj}
    q_{t,j}(r,i)=\begin{cases} 0 &\text{if } r>\tau_+(t^\natural)\\
   \hat\mu(\hat B(\tau_-(t)))\hat\omega_{ji}(t) &\text{if }r=\tau_+(t^\natural)\\
  \hat\mu(\hat B(\tau_-(t)))&\text{if }r\le t^\natural.
    \end{cases}
  \end{equation}
\end{lemma}
\begin{proof} Using \eqref{eq:Phi} and \eqref{eq:f1}, we obtain
   \begin{align}
  \tilde\chi[\Phi_i(r);\xi]&=\int_X\xi(x)\chi[B(\tau_-(r))+z_{i,r};x]d\mu(x)\nonumber
                  \\&=
               \overline{\xi(z_{i,r})}\tilde\chi[B(\tau_-(r));\xi]\nonumber\\
               &=\mu(B(\tau_-(r)))\overline{\xi(z_{i,r})}\chi[B(\tau_-(r))^\bot;\xi]\nonumber\\
              &=\mu(B(\tau_-(r)))\overline{\xi(z_{i,r})}\chi[\hat B(\tau_+(r));\xi], \label{eq:ir}
   \end{align}
where in the last step we also used \eqref{eq:BB} and \eqref{eq:+-}.
Now using partition \eqref{eq:^nmp}-\eqref{eq:^Phi}, we can write
   \begin{align}\label{eq:eg}
 \chi[\hat B(\tau_+(\tilde r));\xi]&=\sum_{i=0}^{\hat n(\tau_+(\tilde r))-1} 
          \chi\big[\hat B(\tilde r)\cdot\theta_{i,\tau_+(\tilde r)};\xi\big]
=\sum_{i=0}^{\hat n(\tau_+(\tilde r))-1} \chi[\hat\Phi_i(\tau_+(\tilde r));\xi]
  \end{align}
and
  $$
  \chi[\hat\Phi_0(\tau_+(\tilde r));\xi]=\sum_{t\in\R:~t\le \tilde r}\sum_{i=1}^{\hat n(t)-1}\chi[\hat\Phi_i(t);\xi].
  $$
Since $\hat B(\tilde r)=B(r)^\bot$ \eqref{eq:BB}, we immediately obtain that
  $$
  \xi(z_{i,r})=\theta_{j,\tau_+(\tilde r)}(z_{i,r})=\overline{\omega_{ij}(r)}, \quad \xi\in\hat\Phi_i(\tau_+(\tilde r))
  $$
and
  \begin{equation}\label{eq:const}
  \overline{\xi(z_{i,r})}\chi[\hat\Phi_i(\tau_+(\tilde r));\xi]=\omega_{ij}(r)\chi[\hat\Phi_i(\tau_+(\tilde r));\xi].
  \end{equation}
In other words, the value of the character $\xi(z_{i,r})$ equals the constant $\omega_{ij}(r)$ on each coset 
$\hat\Phi_j(\tau_+(\tilde r)).$

Now the claimed expressions \eqref{eq:cP}-\eqref{eq:pri} follow on substituting \eqref{eq:eg} and \eqref{eq:const} into
\eqref{eq:ir}. Expressions \eqref{eq:^cP}-\eqref{eq:qtj} are obtained by analogous calculations.
\end{proof}

\begin{theorem}\label{thm:ptj} Partitions $\cN$ and $\widehat\cN$ of the groups $X$ and $\hat X$ into balls defined in 
\eqref{eq:nmp} and \eqref{eq:^nmp} are spectrally dual. The spectral parameters
$p_{r,j}(t,j),q_{t,j}(r,i), r\in\R,t\in\hat\R,1\le i\le n(r), 1\le j\le \hat n(t)$ are given by
\eqref{eq:pri}, \eqref{eq:qtj}. These partitions are symmetric if and only if $n(r)=2$ for all $r\in\R, r>0$ or
equivalently, if and only if $\hat n(t)=2$ for all $t\in \hat\R,t>0.$ In the particular case isolated by these conditions,
the partitions $\cN$ and $\widehat\cN$ coincide with the partitions into spheres \eqref{eq:part}.
\end{theorem}
\begin{proof} We only need to prove the claims about symmetry. Let us introduce the following permutations
on the sets of cosets $z_{i,r},1\le i\le n(r)-1,r\in \R,r>0$ and $\theta_j(t), 1\le j\le \hat n(t)-1, t\in\hat\R, t>0:$
  \begin{equation}\label{eq:prm}
  i\to i':\;-z_{i,r}=z_{i',r}, \qquad j\to j':\;\theta_{j,t}^{-1}=\theta_{j',t}.
  \end{equation}
We obtain $-\Phi_i(r)=\Phi_{i'}(r)$ and $\hat\Phi_j(t)^{-1}=\hat\Phi_{j'}(t).$ 

Permutation $i\to i'$ is an identity if and only if $n(r):=|B(r)/B(\tau_-(r))|=2,$ and $j\to j'$ is an identity
if and only if $\hat n(t):=|\hat B(t)/\hat B(\tau_-(t))|=2.$ In these cases
  \begin{align*}
  \mu(B(\tau_-(r)))=\mu(B(r))-\mu(B(\tau_-(r)))\\
  \hat\mu(\hat B(\tau_-(t)))=\hat\mu(\hat B(t))-\hat\mu(\hat B(\tau_-(t)))
  \end{align*}
and 
  $$
  \omega_{11}(r)=-1, \qquad\hat\omega_{11}(t)=-1.
  $$
Therefore in these cases expressions \eqref{eq:pri} and \eqref{eq:qtj} turn into \eqref{eq:pik3} and
\eqref{eq:qik3}, respectively. We conclude that in these cases, partitions into balls \eqref{eq:nmp}, \eqref{eq:^nmp}
coincide with the partitions into spheres \eqref{eq:part}.
\end{proof}

The coefficients $p_{r,j}(t,j)$ and $q_{t,j}(r,i)$ satisfy orthogonality relations \eqref{eq:pipj},\eqref{eq:qiqj}
which take the following form:
   \begin{align}
  \sum_{t\in \hat\R_0}\sum_{j=1}^{\hat n(t)-1} p_{r_1,i_1}(t,j)\overline{p_{r_2,i_2}(t,j)}\hat\mu(\hat\Phi_j(t))
    &=\delta_{r_1,r_2}\delta_{i_1,i_2}\mu(\Phi_{i_1}(r_1))\label{eq:po}\\
  \sum_{r\in\R_0}\sum_{i=1}^{n(t)-1}q_{t_1,j_1}(r,i)\overline{q_{t_2,j_2}(r,i)}\mu(\Phi_i(r))
     &=\delta_{t_1,t_2}\delta_{j_1,j_2}\hat\mu(\hat\Phi_{j_1}(r_1)),\label{eq:qo}
  \end{align}
where $\R_0:=\{r:\mu(\Phi_i(r))>0\}\subset \R$ and $\hat \R_0:=\{t: \hat\mu(\hat\Phi_j(t))>0\}\subset \hat \R.$
Note that $\mu(\Phi_i(r))$ and $\hat\mu(\hat\Phi_j(t))$ do not depend on the values of $i$ and $j.$ Relations
\eqref{eq:po}, \eqref{eq:qo} in this case can be established directly from \eqref{eq:pri}, \eqref{eq:qtj} using
orthogonality of characters. This verification is much simpler than the calculations for the case of partitions
into spheres (cf. Lemma~\ref{lemma:orth}) and will be left to the reader.

Now let us consider association schemes defined by the partitions \eqref{eq:nmp} and \eqref{eq:^nmp}. Our main
results about them are summarized as follows.
\begin{theorem}\label{thm:tnm} Let $X$ be a zero-dimensional compact or locally compact Abelian group and let $\hat X$ be its dual group.
Let $\cR=\{R_{(r,i)}, r\in \R, 1\le i\le n(r)-1\}$ and $\hat\cR=\{\hat R_{(t,j)}, t\in\hat\R, 1\le j\le \hat n(t)-1\}$
be partitions of the sets of pairs whose blocks are given by
  \begin{align}
    R_{(r,i)}=\{(x,y)\in X\times X: x-y\in\Phi_i(r)\}\\
\hat R_{(t,j)}=\{(\phi,\xi)\in\hat X\times\hat X: \phi\xi^{-1}\in\hat \Phi_j(t)\}.
  \end{align}
Then $\cX=(X,\mu,\cR)$ and $\widehat\cX=(\hat X,\hat\mu,\hat\cR)$ form a pair of mutually dual, nonmetric translation 
association schemes in the sense of Def.~\ref{def1}. The spectral parameters of these schemes are given by
\eqref{eq:pri} and \eqref{eq:qtj} and their intersection numbers are as follows:
  \begin{align}
    p_{(r_1,i_1),(r_2,i_2)}^{(r_3,i_3)}=\begin{cases} 
           0&\text{if }0\le r_1\le r_2<r_3\\
       \mu(B(\tau_-(r)))\delta_{i_2,i_3} &\text{if }0\le r_1<r_2=r_3=r\\
      \mu(B(\tau_-(r)))\delta_{i_1,i_2}^{i_3} &\text{if }0\le r_1=r_2=r_3=r\label{eq:prrrr}
   \end{cases}\\
   \hat p_{(t_i,j_1),(t_2,j_2)}^{(t_3,j_3)}=\begin{cases} 
     0&\text{if }0\le t_1\le t_2<t_3\\
    \hat\mu(\hat B(\tau_-(t)))\delta_{j_2,j_3} &\text{if }0\le t_1<t_2=t_3=t\\
    \hat\mu(\hat B(\tau_-(t)))\hat\delta_{j_2,j_2}^{j_3} &\text{if }0\le t_1=t_2=t_3=t,
   \end{cases}
  \end{align}
where
   $\delta_{i_1,i_2}^{i_3}=\mathbbold{1}\{z_{i_1,r}+z_{i_2,r}=z_{i_3,r}\}$ and 
$\hat\delta_{j_1,j_2}^{j_3}=\mathbbold{1}\{\theta_{j_1,t}\theta_{j_2,t}=\theta_{j_3,t}\}.$
The schemes $\cX$ and $\hat \cX$ are symmetric if and only if $n(r)=2$ for all $r\in\R, r>0$ or
equivalently, if and only if $\hat n(t)=2$ for all $t\in \hat\R,t>0.$
\end{theorem}
{\em Proof:}
On account of Theorems \ref{thm:m2} and \ref{thm:ptj} we only need to verify the expressions for the
intersection numbers.  We begin with \eqref{eq:ser1}, \eqref{eq:ser2}:
  \begin{align}
  p_{(r_1,i_1),(r_2,i_2)}^{(r_3,i_3)}&=\frac1{\mu(B(\tau_-(r_3)))}\sum_{t\in \hat\R_0}\sum_{j=1}^{\hat n(t)-1}
   p_{r_1,i_1}(t,j)p_{r_2,i_2}(t,j)\overline{p_{r_3,i_3}(t,j)}\hat\mu(\hat\Phi_j(t)) \label{eq:prrii}\\
   \hat p_{(t_1,j_1),(t_2,j_2)}^{(t_3,j_3)}&=\frac1{\hat\mu(\hat B(\tau_-(t_3)))}\sum_{r\in \R_0}\sum_{i=1}^{n(r)-1}
     q_{t_1,j_1}(r,i)q_{t_2,j_2}(r,i)\overline{q_{t_3,j_3}(r,i)}\mu(\Phi_i(r)).
   \end{align}
Here we used the fact that $\mu(\Phi_i(r))=\mu(B(\tau_-(r))),\hat\mu(\hat\Phi_j(t))=\hat\mu(\hat B(\tau_-(t)))$ 
implied by \eqref{eq:nmp}, \eqref{eq:^nmp}. We will use these equalities in the form
  \begin{equation}\label{eq:cB}
   \mu(\Phi_i(\tau_+(r)))=\mu(B(r)), \quad\hat\mu(\hat\Phi_j(\tau_+(r)))=\hat\mu(\hat B(t)).
  \end{equation}
The following relations are obvious:
  \begin{align}
  \sum_{t\in\hat \R_0:\,t\le t_0}\sum_{j=1}^{\hat n(t)-1}\hat\mu(\hat \Phi_j(t))&=\hat\mu(\hat B(t_0)), \quad t\in\hat\R_0
    \label{eq:ss}\\
  \sum_{r\in \R_0:\,r\le r_0}\sum_{i=1}^{n(r)-1}\mu(\Phi_i(r))&=\mu(B(r_0)), \quad r_0\in \R_0.
  \end{align}
The symmetry relations \eqref{eq:iij} have the form
  $$
   \mu(B(\tau_-(r_3))) p_{(r_1,i_1),(r_2,i_2)}^{(r_3,i_3)}=\mu(B(\tau_-(r_1)))p_{(r_1,i_1),(r_2,i_2')}^{(r_1,i_1)}
    =\mu(B(\tau_-(r_2)))p_{(r_1,i_1'),(r_3,i_3)}^{(r_2,i_2)}
  $$
  $$
   \hat\mu(\hat B(\tau_-(t_3)))\hat p_{(t_1,j_1),(t_2,j_2)}^{(t_3,j_3)}=\hat\mu(\hat B(\tau_-(t_3)))
       \hat p_{(t_3,j_3),(t_2,j_2')}^{(t_1,j_1)}=
    \hat\mu(B(\tau_-(t_2)))\hat p_{(t_1,j_1'),(t_3,j_3)}^{(t_2,j_2)},
  $$
where the bijections $i\to i',j\to j'$ are given in \eqref{eq:prm}. Also, clearly,
   $$
   p_{(r_1,i_1),(r_2,i_2)}^{(r_3,i_3)}=p_{(r_2,i_2),(r_1,i_1)}^{(r_3,i_3)}; \quad
     \hat p_{(t_1,j_1),(t_2,j_2)}^{(t_3,j_3)}=\hat p_{(t_2,j_2),(t_1,j_1)}^{(t_3,j_3)}.
  $$
Because of these symmetries, it suffices to compute the intersection numbers under the conditions
  $$
  0\le r_1\le r_2\le r_3 \quad\text{and}\quad 0\le t_1\le t_2\le t_3.
  $$
Let us find the $p$'s; the corresponding calculations for the $\hat p$'s are completely analogous.
Consider the following 3 cases:
\begin{enumerate}
\item[(i)] The largest radius is unique, i.e., $0\le r_1\le r_2<r_3;$
\item
[(ii)] There are exactly two largest radii, i.e., $0\le r_1<r_2=r_3=r;$
\item
[(iii)] All the 3 radii are equal: $0\le r_1=r_2=r_3=r.$
\end{enumerate}
In the case (i) $\tilde r_3<\tilde r_2\le \tilde r_1$ (see \eqref{eq:gl}), so \eqref{eq:pri} implies that
the summation in \eqref{eq:prrii} extends to the region $0\le t\le \tau_+(\tilde r_3)\le\tilde r_2\le \tilde r_1.$
For such $t$ the coefficients $p_{r_1,i_1}(t,j)$ and $p_{r_2,i_2}(t,j)$ do not depend on $t$. Therefore 
\eqref{eq:prrii} takes the form
    $$
   \mu(B(\tau_-(r_3)))p_{(r_1,i_1),(r_2,i_2)}^{(r_3,i_3)}=\mu(B(\tau_-(r_1)))\mu(B(\tau_-(r_2)))\Sigma_1,
   $$
where we have denoted
  $$
 \Sigma_1=\sum_{0\le t\le \tau_+(\tilde r_3)}\sum_{j=1}^{\hat n(t)-1}\overline{p_{r_3,i_3}(t,j)}\hat\mu(\hat \Phi_j(t)).
  $$
Using relations \eqref{eq:pri}, \eqref{eq:cB}, and \eqref{eq:ss}, we can write
   \begin{align*}
  \Sigma_1&=\mu(B(\tau_-(r_3)))\Big\{\hat\mu(\hat B(\tilde r_3))\sum_{j=1}^{\hat n(\tau_+(\tilde r_3))-1}
      \overline{\omega_{i_3,j}(r_3)}+ \sum_{0\le t\le\tilde r_3}\sum_{j=1}^{\hat n(t)-1}\hat\mu(\hat\Phi_j(t))\Big\}\\
            &=\mu(B(\tau_-(r_3)))\hat\mu(\hat B(\tilde r_3))\sum_{j=0}^{\hat n(\tau_+(\tilde r_3))-1}
      \overline{\omega_{i_3,j}(r_3)}\\
     &\stackrel{\eqref{eq:oo1}}=0.
  \end{align*}
This exhausts the first case in \eqref{eq:prrrr}.

In case (ii) the summation in \eqref{eq:prrii} in effect is over the region $0\le t\le\tau_+(\tilde r)\le\tilde r_1,$
and in this region the coefficient $p_{r_1,i_1}(t,j)$ is independent of $t$. Therefore, \eqref{eq:prrii} takes the form
    \begin{equation}\label{eq:cci}
   \mu(B(\tau_-(r)))p_{(r_1,i_1),(r,i_2)}^{(r,i_3)}=\mu(B(\tau_-(r_1)))\Sigma_2,
   \end{equation}
where we have denoted
  $$
  \Sigma_2=\sum_{0\le t\le \tau_+(\tilde r)}\sum_{j=1}^{\hat n(t)-1}p_{r,i_2}(t,j)\overline{p_{r,i_3}(t,j)}\hat\mu(\hat \Phi_j(t))
  $$
As in the previous case, use \eqref{eq:pri}, \eqref{eq:cB}, and \eqref{eq:ss} to write 
  \begin{align}
  \Sigma_2&=\mu(B(\tau_-(r)))^2\Big\{\hat\mu(\hat B(\tilde r))\sum_{j=1}^{\hat n(\tau_+(\tilde r))-1}\nonumber
     \omega_{i_2,j}(r)\overline{\omega_{i_3,j}(r)}+\sum_{0\le t\le\tilde r}\sum_{j=1}^{\hat n(t)-1}\hat\mu(\hat\Phi_j(t))\Big\}\nonumber\\
     &=\mu(B(\tau_-(r)))^2\hat\mu(\hat B(\tilde r))\sum_{j=0}^{\hat n(\tau_+(\tilde r))-1}\omega_{i_2,j}(r)\overline{\omega_{i_3,j}(r)}\nonumber\\
     &=\mu(B(\tau_-(r)))^2\hat\mu(\hat B(\tilde r))\hat n(\tau_+(\tilde r))\delta_{i_2,i_3}\nonumber\\
     &=\mu(B(\tau_-(r)))^2\hat\mu(\hat B(\tau_+(\tilde r)))\delta_{i_2,i_3}\nonumber\\
     &=\mu(B(\tau_-(r)))\delta_{i_2,i_3},\label{eq:cii}
  \end{align}
where we have used orthogonality of characters \eqref{eq:oo1} and the following simple relations:
  \begin{gather*}
  \hat\mu(\hat B(\tilde r))\hat n(\tau_+(\tilde r))=\hat\mu(\hat B(\tau_+(\tilde r)))\\
  \mu(B(\tau_-(r)))\hat\mu(\hat B(\tau_+(\tilde r)))=\mu(B(\tau_-(r)))\hat\mu(\hat B(\widetilde{\tau_-(r)}))=1.
  \end{gather*}
  Now the second expression in \eqref{eq:prrrr} follows from \eqref{eq:cii} and \eqref{eq:cci}.
  
Finally, in case (iii) we have
  \begin{align*}
  \mu(B(\tau_-(r)))&p_{(r,i_1),(r,i_2)}^{(r,i_3)}=\sum_{t\in\R_0}\sum_{j=1}^{\hat n(t)-1} p_{r,i_1}(t,j)
  p_{r,i_2}(t,j)\overline{p_{r,i_3}(t,j)}\hat\mu(\Phi_j(t))\\
  &=\mu(B(\tau_-(r)))^3\bigg\{\sum_{j=1}^{\hat n(\tau_+(\tilde r))-1}\omega_{i_1,j}(r)\omega_{i_2,j}(r)\overline{\omega_{i_3,j}(r)}
      \hat\mu(\Phi_j(\tau_+(\tilde r)))\\
      &\hspace*{.7in}+\sum_{0\le t\le\tilde r}\sum_{j=1}^{\hat n(t)-1}\hat\mu(\hat \Phi_j(t))\bigg\}\\
    &=  \mu(B(\tau_-(r)))^3\hat\mu(\hat B(\tilde r))\sum_{j=0}^{\hat n(\tau_+(\tilde r))-1}
     \omega_{i_1,j}(r)\omega_{i_2,j}(r)\overline{\omega_{i_3,j}(r)}.
  \end{align*}
  Now note that, from \eqref{eq:R}, the product $  \omega_{i_1,j}(r)\omega_{i_2,j}(r)$ is a character $\bar\theta_{j,\tau_+(\tilde r)}$ evaluated at $z_{i_1,r}+z_{i_2,r}.$ Therefore, invoking orthogonality, we obtain
  $$
  \sum_{j=0}^{\hat n(\tau_+(\tilde r))-1}
     \omega_{i_1,j}(r)\omega_{i_2,j}(r)\overline{\omega_{i_3,j}(r)}=\hat n(\tau_+(\tilde r))\delta_{i_1,i_2}^{i_3}.
  $$
Summarizing, we obtain
  \begin{align*}
  p_{(r,i_1),(r,i_2)}^{(r,i_3)}&=\mu(B(\tau_-(r)))^2\hat\mu(\hat B(\tilde r))\hat n(\tau_+(\tilde r))\delta_{i_1,i_2}^{i_3}\\
    &=\mu(B(\tau_-(r)))^2\hat\mu(\hat B(\tau_+(\tilde r)))\delta_{i_1,i_2}^{i_3}\\
    &=\mu(B(\tau_-(r)))\delta_{i_1,i_2}^{i_3}
  \end{align*}
The theorem is proved. \qed

%
%

\section{Adjacency algebras (Schur rings)}\label{sect:aa}
In this section we develop a formal approach to adjacency algebras of association schemes on zero-dimensional groups introduced earlier
in this paper\footnote{cf. the remark made in the end of Sect.~\ref{sect:or}.}. We study the case of compact groups in detail and briefly describe modifications needed to cover the case of locally compact groups. Since the association schemes have the translation
property, i.e., are invariant under the group operation, the function algebras that we consider are defined on the group $X$ rather
than on the Cartesian square $X\times X$. In algebraic combinatorics such algebras are known as $S$-rings or Schur rings \cite{muz09a}.

%
%
\subsection{Compact groups} Let $X$ be a compact infinite zero-dimensional Abelian group; accordingly, the group $\hat X$ is discrete, countable, and periodic.
As before, we assume that the Haar measures on $X$ and $\hat X$ are normalized by the conditions $\mu(X)=\hat\mu(\{1\})=1.$

\subsubsection{Notation} Consider functions on $X$ of the form
   \begin{equation}\label{eq:fx}
   f(x)=\sum_{r\in\I}\sum_{i=1}^{n(r)-1} c_{r,i}\chi[\Phi_i(r);x],
   \end{equation}
where the notation is defined in \eqref{eq:si} and \eqref{eq:Phi1}.
These functions are well defined for any choice of complex coefficients $c_{r,i}$. 

To describe the adjacency algebra ${\mathfrak A}(X)$ we will construct a linear space of functions that contains all the 
finite sums of the form \eqref{eq:fx} which is closed with respect to the usual product and convolution of functions. Clearly
it is not enough to consider the space that contains only the finite sums of the form \eqref{eq:fx}: even though such a space is closed with respect to the usual product, it is not closed under convolution. Indeed, the definition of intersection
numbers in Lemma~\ref{lemma:inter} as well as expressions \eqref{eq:pijk1}, \eqref{eq:prrrr} show that already the convolution
of any two functions $\chi[\Phi_{i_1}(r_1);x]$ and $\chi[\Phi_{i_2}(r_2);x], r_1,r_2\in\I$ is not a finite sum of the form 
\eqref{eq:fx}. Therefore, for the algebra to be well defined, we need to enlarge the space of finite sums in order to
have the needed closure property. 
The key observation here is that the infinite sums involved in these expressions are of a very special form. 

We will need the following obvious formulas:
  \begin{align}\label{eq:cnr}
  &\sum_{r\in\I}\sum_{i=1}^{n(r)-1}\chi[\Phi_i(r);x]=1 \quad\text{for all }x\in X\\
  &\sum_{0\le r\le a} \sum_{i=1}^{n(r)-1}\chi[\Phi_i(r);x]=1-\sum_{a<r\le\bar r}\sum_{i=1}^{n(r)-1}\chi[\Phi_i(r);x]\label{eq:1-},\quad a\in \I_0,
  \end{align}
where $\bar r$ is the maximum radius \eqref{eq:br}.
Importantly, \eqref{eq:1-} enables us to transform infinite sums on the left into finite sums.

Let us introduce some notation. Let us number the values of the radius $r\in \I_0$ as follows:
   \begin{equation}\label{eq:rn}
  \bar r=r_1>r_2>r_3>\dots .
\end{equation}
Note that this numbering is slightly different from the numbering used earlier. Number the radii $t\in\hat \I$ as before:
  \begin{equation}\label{eq:tn}
  \{1\}=t_0<t_1<t_2<\dots .
  \end{equation}
Using this notation, the bijections \eqref{eq:rtr} and the operations $\tau_+,\tau_-$ take the following form:
  \begin{equation}
\left.
\begin{array}{l}
   \tilde r_l=t_{l-1}, \;l\ge 1,\qquad\quad t_k^\natural=r_{k+1}, \; k\ge 0\\
  \tau_-(r_l)=r_{l+1},\; l\ge 1,\quad \tau_+(r_l)=r_{l-1},\;l\ge 2\\
   \tau_-(t_k)=t_{k-1},\, k\ge 1,\quad \tau_+(t_k)=t_{k+1}, k\ge 0.
\end{array}\;\right\}
\label{eq:nn}
\end{equation}
In particular, this implies that
  \begin{equation}
\left.\begin{array}{ll}
  \tau_+(\tilde r_l)=\tau_+(t_{l-1})=t_l,&l\ge 1\\
  \tau_+(t_k^\natural)=\tau_+(r_{k+1})=r_k,&k\ge 1.
\end{array}\right\}
      \label{eq:kl}
  \end{equation}
Let 
  \begin{equation}
  \left.\begin{array}{l}
   \gls{alpha0}=\chi[B(\bar r);x]=1 \quad\text{for all } x\in X\\
   \gls{alphali}=\chi[\Phi_i(r_l);x],\;i=1,\dots,n_l-1,\; n_l=n(r_l),\;l\ge 1.
  \end{array}\right\}\label{eq:aal}
  \end{equation}
 Let $\gls{cA}$ be the set of all functions of the form
   \begin{equation}\label{eq:fA}
  f(x)=c_0\alpha_0(x)+\sum_{l\in\naturals}\sum_{i=1}^{n_l-1}c_{l,i}\alpha_{l,i}(x), \quad x\in X
     \end{equation}
where only finitely many coefficients $c_0,c_{l,i}$ are nonzero. Clearly, $\cA$ is a countably dimensional complex linear space.
Note that a function $f\in\cA$ can be written in the form \eqref{eq:fx}, where the coefficients $c_{r,i}$ in \eqref{eq:fx}
are independent of $r$ and $i$ as long as $r\le a$ for some radius $a=a(f)\in\I_0.$

Dually, define
   \begin{equation}
   \left.\begin{array}{l}
    \beta_0(\xi)=\chi[\hat B(0);\xi], \quad \hat B(0)=\{1\}\\
     \beta_{k,j}(\xi)=\chi[\hat \Phi_j(t_k);\xi],\;j=1,\dots,\hat n_k-1,\;\hat n_k=\hat n(t_k), k\ge 1
   \end{array}\right\}
  \label{eq:bbk}
   \end{equation}
and denote by $\widehat \cA$ a countably dimensional complex vector space formed by all functions of the form
  \begin{equation}\label{eq:gA}
  g(\xi)=c_0'\beta_0(\xi)+\sum_{k\in\naturals}\sum_{j=1}^{\hat n_k-1} c_{k,j}'\beta_{k,j}(\xi), \quad \xi\in\hat X,
  \end{equation}
where only finitely many coefficients $c_0',c_{k,j}'$ are nonzero. 
 
\subsubsection{Function algebras $\cA$ and $\widehat \cA$} The goal of this section is to show that the sets $\cA$ and $\widehat{\cA}$ form algebras
(Schur rings) that are closed with respect to multiplication and convolution of functions. It is exactly these algebras that
should be considered as adjacency algebras of the translation schemes $\cR$ and $\hat \cR$ constructed in Sect.~\ref{sect:nonmetric}, Theorem \ref{thm:tnm}.
Later in Sect.~\ref{sect:subA} we also identify subalgebras of these algebras that form adjacency algebras of the 
dual pairs of metric schemes constructed in Section~\ref{sect:vil}.

The products of the basis functions satisfy the following obvious relations:
   \begin{equation}\label{eq:pa}
    \begin{array}{l}
    \alpha_0(x)\alpha_0(x)=\alpha_0(x), \;\;\alpha_0(x)\alpha_{l,i}(x)=\alpha_{l,i}(x)\\
    \alpha_{l_1,i_1}(x)\alpha_{l_2,i_2}(x)=\delta_{l_1,l_2}\delta_{i_1,i_2}\alpha_{l_1,i_1}(x)
    \end{array}
    \end{equation}
    and
  \begin{equation}\label{eq:pb01}
    \begin{array}{l}
    \beta_0(\xi)\beta_0(\xi)=\beta_0(\xi), \;\;\beta_0(\xi)\beta_{k,j}(\xi)=\beta_{k,j}(\xi)\\
    \beta_{k_1,j_1}(\xi)\beta_{k_2,j_2}(\xi)=\delta_{k_1,j_2}\delta_{k_1,j_2}\beta_{k_1,j_1}(\xi).
    \end{array}
    \end{equation}  

\begin{proposition} The Fourier transforms of the basis functions \eqref{eq:aal} and \eqref{eq:bbk} have the following form:
    \begin{equation}\label{eq:Fal}
   \tilde\alpha_{l,i}(\xi)=\pi_{l,i}(0)\beta_0(\xi)+\sum_{k\in\naturals}\sum_{j=1}^{\hat n_k-1}\pi_{l,i}(k,j)\beta_{k,j}(\xi),\;l\ge 1
    \end{equation}
where we use the notation
   \begin{equation}
   \left.
   \begin{array}{l}
   \pi_0(0)=1,\;\;\pi_0(k,j)=0, \;k\ge 1,\\
   \pi_{l,i}(0)=\mu(B(r_{l+1})), \;l\ge 1,\\
   \pi_{l,i}(k,j)=\mu(B(r_{l+1})) \quad\text{if }\;1\le k<l, \;l\ge 1,\\
   \pi_{l,i}(l,j)=\mu(B(r_{l+1}))\omega_{ij}(r_l), \quad l\ge 1\\
   \pi_{l,i}(k,j)=0 \quad\text{if }\;k>l,\;l\ge 1
    \end{array}\right\}\label{eq:pi}
  \end{equation}
(see also \eqref{eq:R});
  \begin{equation}\label{eq:Fbe}
   \beta^\natural_{k,j}(x)=\kappa_{k,j}(0)\alpha_0(x)+\sum_{l\in\naturals}\sum_{i=1}^{n_l-1}
                     \kappa_{k,j}(l,i)\alpha_{l,i}(x),
  \end{equation}
where we use the notation
   \begin{equation}
   \left.\begin{array}{l}
   \kappa_0(0)=1,\;\; \kappa_0(l,i)=0, \;l\ge 1\\
    \kappa_{k,j}(0)=\hat\mu(\hat B(t_{k-1})),\;k\ge 1\\
   \kappa_{k,j}(l,i)=-\hat\mu(\hat B(t_{k-1})) \quad\text{if }\;1\le l<k, \;k\ge 1\\
   \kappa_{k,j}(k,i)=\hat\mu(\hat B(t_{k-1}))(\hat\omega_{ji}(t_k)-1), \;k\ge 1\\
    \kappa_{k,j}(l,i)=0, \;l>k,\;k\ge 1
  \end{array}\right\}\label{eq:kappa}
   \end{equation}
(see also \eqref{eq:^R}).
Finally,
   \begin{equation}\label{eq:a0b0}
   \tilde\alpha_0(\xi)=\beta_0(\xi),\quad \beta_0^\natural(x)=\alpha_0(x).
  \end{equation}
 \end{proposition}  
\begin{proof}
Expression \eqref{eq:Fal} is readily obtained from \eqref{eq:aal} on using \eqref{eq:ft} and \eqref{eq:cP}-\eqref{eq:pri}.
Turning to \eqref{eq:Fbe}, let us use \eqref{eq:^cP} and \eqref{eq:qtj} to write
  \begin{align*} 
   \beta^\natural_{k,j}(x)&=\sum_{l\in\naturals}\sum_{i=1}^{n_l-1} q_{t_k,j}(r_l,j)\alpha_{l,i}(x)\\
 & =\hat\mu(\hat B(t_{k-1}))\biggl[\sum_{l>k}\sum_{i=1}^{n_l-1}\alpha_{l,i}(x)+\sum_{i=1}^{n_k-1}
    \hat\omega_{ji}(t_k)\alpha_{k,i}(x)\biggr],\quad k\ge 1.
  \end{align*}
Now use \eqref{eq:1-} to transform the infinite series into a finite sum:
   $$
   \sum_{l>k}\sum_{i=1}^{n_l-1}\alpha_{l,i}(x)=1-\sum_{l=1}^k\sum_{i=1}^{n_l-1}\alpha_{l,i}(x).
   $$
Substituting and using the notation in \eqref{eq:kappa}, we obtain \eqref{eq:Fbe}.
\end{proof}

It is important to observe that the coefficients in \eqref{eq:pi} and \eqref{eq:kappa} form {\em triangular arrays:}
  \begin{equation}
  \begin{array}{ll@{\qquad}ll}
   \pi_0(k,j)=0 &\text{if }k\ge 1; &\kappa_0(l,i)=0&\text{if }l\ge 1,\\
   \pi_{l,i}(k,j)=0 &\text{if }k>l,\;l\ge 1;&\kappa_{k,j}(l,i)=0 &\text{if }k<l,\;l\ge 1.
  \end{array}
  \label{eq:triangular}
  \end{equation}
Our choice of the basis functions in \eqref{eq:aal} is determined precisely by these properties. Namely, we have shown that
the Fourier transforms interchange functions of the form \eqref{eq:fA} and functions of the form \eqref{eq:gA}.
Recall that only a finite number of coefficients is nonzero in the definitions \eqref{eq:fA} and \eqref{eq:gA}. 
We conclude as follows.
\begin{proposition} 
   \begin{equation}
  \cF^\sim \cA=\widehat \cA, \quad \cF^\natural \widehat\cA=\cA.
   \label{eq:FF}
  \end{equation}
The function spaces $\cA$ and $\widehat \cA$ are commutative algebras closed with respect to multiplication and convolution.
\end{proposition}
\begin{proof} The closedness with respect to multiplication is obvious, and the closedness under convolution follows
from \eqref{eq:conv1}-\eqref{eq:conv2}.
\end{proof}

In the next proposition we explicitly compute convolutions of functions in $\cA$ and $\widehat\cA.$
\begin{proposition}
We have
  \begin{equation}
  \alpha_{l_1,i_1}\ast\alpha_{l_2,i_2}(x)=\pi_{(l_1,i_1),(l_2,i_2)}^{(0)}\alpha_0(x)+\sum_{l\in\naturals}\sum_{i=1}^{n_l-1}
  \pi_{(l_1,i_1),(l_2,i_2)}^{(l,i)}\alpha_{l,i}(x),
  \label{eq:ca0}
  \end{equation}
where
    \begin{align}
   \pi_{(l_1.i_1),(l_2,i_2)}^{(0)}&=\pi_{l_1,i_1}(0)\pi_{l_2,i_2}(0)
          +\sum_{k=1}^{\min\{l_1,l_2\}}\sum_{j=1}^{\hat n_k-1}\pi_{l_1,i_1}(k,j)\pi_{l_2,i_2}(k,j)\kappa_{k,j}(0)
                   \label{eq:pii1}\\
   \pi_{(l_1,i_1),(l_2,i_2)}^{(l,i)}&=\sum_{k=l}^{\min\{l_1,l_2\}}\sum_{j=1}^{\hat n_k-1}\pi_{l_1,i_1}(k,j)\pi_{l_2,i_2}(k,j)\kappa_{k,j}(l,i);\label{eq:pii2}
  \end{align}
  \begin{equation}\label{eq:cb0}
  \beta_{k_1,j_1}\ast\beta_{k_2,j_2}(\xi)=\kappa_{(k_1,j_1),(k_2,j_2)}^{(0)}\delta_0(\xi)+\sum_{k\in\naturals}
    \sum_{j=1}^{\hat n_k-1}\kappa_{(k_1,j_1),(k_2,j_2)}^{(k,j)}\beta_{k,j}(\xi)
  \end{equation}
where
   \begin{align}
    \kappa_{(k_1,j_1),(k_2,j_2)}^{(0)}&:=\kappa_{k_1,j_1}(0)\kappa_{k_2,j_2}(0)&&\nonumber\\
                &\hspace*{-.5in}+\sum_{l=1}^{\max\{k_1,k_2\}}\sum_{i=1}^{n_l-1}\Big\{\kappa_{k_1,j_1}(l,i)+\kappa_{k_2,j_2}(l,i)+
           \kappa_{k_1,j_1}(l,i)\kappa_{k_2,j_2}(l,i)   \Big\}\pi_{l,i}(0) \label{eq:kp1}\\
    \kappa_{(k_1,j_1),(k_2,j_2)}^{(k,j)}&:=\sum_{l=k}^{\max\{k_1,k_2\}}\sum_{i=1}^{n_l-1}
     \Big\{\kappa_{k_1,j_1}(l,i)+\kappa_{k_2,j_2}(l,i)+
           \kappa_{k_1,j_1}(l,i)\kappa_{k_2,j_2}(l,i)   \Big\}\pi_{l,i}(k,j).\label{eq:kp2}
    \end{align}
Finally,
   \begin{align}
    \alpha_0\ast\alpha_0(x)&=\pi_{(0),(0)}^{(0)}\alpha_0(x),\quad \alpha_0\ast\alpha_{l,i}(x)=\pi_{(0),(l,i)}^{(0)}\alpha_0(x)
    \label{eq:a0a0}\\
    \beta_0\ast\beta_0(\xi)&=\kappa_{(0),(0)}^{(0)}\beta_0(\xi),\quad \beta_0\ast\beta_{k,j}(\xi)=
\kappa_{(0),(k,j)}^{(0)}\beta_{k,j}(\xi)\label{eq:b0b0}
   \end{align}
where
  \begin{equation}\label{eq:pi0}
   \left.\begin{array}{c}
   \pi_{(0),(0)}^{(0)}=1\\
   \pi_{(0),(l,i)}^{(0)}=\pi_{(l,i),(0)}^{(0)}=\mu(B(r_{l+1})),\;l\ge 1\\
  \kappa_{(0),(k,j)}^{(0)}=\kappa_{(k,j),(0)}^{(0)}=\hat\mu(\hat B(t_{k-1})),\;k\ge 1
     \end{array}\right\}
  \end{equation}
For completeness, we also put
   \begin{equation}\label{eq:kp0}
   \left.\begin{array}{c}
    \kappa_{(0),(0)}^{(0)}=1\\
  \pi_{(0),(l_1,i_1)}^{(l,i)}=\pi_{(l_1,i_1),(0)}^{(l,i)}=0,\;l\ge 1,l_1\ge 1\\
   \kappa_{(0),(k_1,j_1)}^{(k,j)}=\kappa_{(k_1,j_1),(0)}^{(k,j)}=0, \;k\ge1,k_1\ge1
  \end{array}\right\}
   \end{equation}
\end{proposition}
\begin{proof} Let us compute the convolution $\alpha_{l_1,i_1}\!\ast\alpha_{l_2,i_2}.$ Using \eqref{eq:Fal}, \eqref{eq:pb01},
we obtain 
   $$
   \tilde\alpha_{l_1,i_1}(\xi)\tilde\alpha_2(\xi)=\pi_{l_1,i_1}(0)\pi_{l_2,i_2}(0)\beta_0(\xi)+
      \sum_{k\in\naturals}\sum_{j=1}^{\hat n_k-1}\pi_{l_1,i_1}(k,j)\pi_{l_2,i_2}(k,j)\beta_{k,j}(\xi).
   $$
Computing the Fourier transforms on both sides of this equality and using \eqref{eq:Fbe} and \eqref{eq:a0b0}, we obtain
   \begin{multline*}
   \alpha_{l_1,i_1}\ast\alpha_{l_2,i_2}(x)=\pi_{l_1,i_1}(0)\pi_{l_2,i_2}(0)\alpha_0(x)+\sum_{k\in\naturals}
    \sum_{j=1}^{\hat n_k-1}\pi_{l_1,k_1}(k,j)\pi_{l_2,i_2}(k,j)\\
    \times\Big\{\kappa_{k,j}(0)\alpha_(x)+\sum_{l\in\naturals}\sum_{i=1}^{n_l-1}\kappa_{k,j}(l,i)\alpha_{l,i}(x)\Big\}
    \end{multline*}
Note that the sums in this expression are finite because of the conditions \eqref{eq:triangular}. 
Interchanging the order of summation, simplifying, and using notation \eqref{eq:pii1}-\eqref{eq:pii2}, we obtain \eqref{eq:ca0}. The summation range in \eqref{eq:pii1}, \eqref{eq:pii2} follows upon invoking again \eqref{eq:triangular}.

To prove \eqref{eq:cb0}, start with \eqref{eq:Fbe} and compute
  \begin{multline}
  \beta^\natural_{k_1,j_1}\beta^\natural_{k_2,j_2}=\kappa_{k_1,j_1}(0)\kappa_{k_2,j_2}(0)\alpha_0(x)\\
   +\sum_{l\in\naturals} \sum_{i=1}^{n_l-1} \Big\{\kappa_{k_1,j_1}(l,i)+\kappa_{k_2,j_2}(l,i)+
           \kappa_{k_1,j_1}(l,i)\kappa_{k_2,j_2}(l,i)   \Big\}\alpha_{l,i}(x).
  \end{multline}
Computing the Fourier transform on both sides of this equality and using \eqref{eq:Fal}, \eqref{eq:a0b0}, we obtain
  \begin{multline*}
   \beta_{k_1,j_1}\ast\beta_{k_2,j_2}(\xi)=\kappa_{k_1,j_1}(0)\kappa_{k_2,j_2}(0)\beta_0(\xi)\\
           +\sum_{l\in\naturals} \sum_{i=1}^{n_l-1} \Big\{\kappa_{k_1,j_1}(l,i)+\kappa_{k_2,j_2}(l,i)+
           \kappa_{k_1,j_1}(l,i)\kappa_{k_2,j_2}(l,i)   \Big\}\\
                \times\Big\{\pi_{l,i}(0)\beta_0(\xi)+\sum_{k\in\naturals}\sum_{j=1}^{hat n_k-1}\pi_{l,i}(k,j)\beta_{k,j}(\xi)\Big\}.
   \end{multline*}
The sums in this expression are again finite because of \eqref{eq:triangle}, so we can interchange the order of summation. 
Simplifying, we obtain \eqref{eq:cb0}.

Finally, expressions \eqref{eq:a0a0}, \eqref{eq:b0b0}, \eqref{eq:pi0}, \eqref{eq:kp0} are verified directly.
\end{proof}

Note that the coefficients \eqref{eq:pii2} also form a triangular array:
   $$
  \pi_{(l_1,i_1),(l_2,i_2)}^{(l,i)}=0 \quad\text{if } l>\min\{l_1,l_2\}.
   $$
Using \eqref{eq:pi} and \eqref{eq:triangular}, it is possible to express them
via the measures $\mu(B(r))$ and $\hat\mu(\hat B(r))$ and the characters $\omega_{ij}(r),\hat\omega_{ji}(r)$, 
but the resulting expressions are too cumbersome to write out explicitly.

Expressions \eqref{eq:pii1}, \eqref{eq:pii2}, and \eqref{eq:pi0} give a complete set of coefficients for computing 
the convolution of any two functions from the set $\{\alpha_0,\alpha_{l,i},l\in\naturals\}$ and therefore of any two
functions in the algebra $\cA$. The same is true with respect to the expressions \eqref{eq:kp1}, \eqref{eq:kp2}, and
\eqref{eq:kp0} and the functions from the algebra $\widehat\cA.$

%
%
\subsubsection{Adjacency algebras} Here we show that the algebras $\cA$ and $\widehat \cA$ introduced above can be
viewed as adjacency algebras \index{adjacency algebra} of association schemes constructed in Theorem \ref{thm:tnm}.
Let $\cA_m,m\in\naturals_0$ be the set of all functions of the form \eqref{eq:fA} such that $c_{l,i}=0$ for $l>m.$
Clearly $\cA_m$ is a vector space of dimension
    $$
   \dim \cA_m=1+\sum_{l=1}^m(n(r_l)-1), \quad m\ge 1,
   $$
where $n(r)$ is defined in \eqref{eq:si}. For $m=0,$ $\cA_0=\text{Const}_X$ is the one-dimensional space of constant 
functions on $X.$ It is easy to see that
  \begin{equation}\label{eq:cA<}
  \cA_0\subset \cA_1\subset\cA_2\subset\dots\subset\cA,
  \end{equation}
and
   \begin{equation}\label{eq:Aun}
  \cA=\bigcup_{m\in\naturals_0}\cA_m.
  \end{equation}
Analogously, denote by $\widehat\cA_m, m\in\naturals_0$ the set of all functions of the form    \eqref{eq:gA}
such that $c'_{k,j}=0$ for $k>m.$ Clearly, $\widehat\cA_m$ is a vector space of dimension
  $$
  \dim \widehat\cA_m=1+\sum_{l=1}^m(\hat n(r_l)-1), \quad m\ge 1,
  $$
and $\widehat\cA_0$ is a one-dimensional space of functions $c_0'\beta_0(\xi)$ supported on the unit element of the group $\hat X.$
The following embedding is easy to see:
   \begin{equation}\label{^cA<}
   \widehat \cA_0\subset\widehat\cA_1\subset\widehat\cA_2\subset\dots\subset\widehat\cA,
   \end{equation}
and
   \begin{equation}\label{eq:^Aun}
   \widehat A=\bigcup_{m\in\naturals_0}\widehat\cA_m.
   \end{equation}
   Expressions \eqref{eq:cA<}, \eqref{eq:Aun} and \eqref{^cA<}, \eqref{eq:^Aun} support the view of $\cA$ and $\widehat \cA$ as
countably dimensional algebras graded by finite-dimensional algebras $\cA_m,\widehat\cA_m, m\in\naturals_0$. 
Essentially, the algebras $\cA$ and $\widehat \cA$ are inductive (direct) limits of the subalgebras $\cA_m,\widehat\cA_m:$
   $$
   \cA=\lim_{\stackrel{\longrightarrow}{m}}\cA_m,\qquad     \widehat\cA=\lim_{\stackrel{\longrightarrow}{m}}\widehat\cA_m.
   $$
The grading \eqref{eq:cA<} define a natural topology on $\cA$ and $\widehat\cA$, namely, a sequence of functions 
$f_j,j\in \naturals$ converges to a function $f$ if for all sufficiently large $j,$ the functions $f_j$ are contained in
some subalgebra $\cA_m,$ in which the convergence is defined by a usual topology of a finite-dimensional space.
A similar remark applies to $\widehat \cA$ and the grading \eqref{^cA<}.
Clearly, in this topology, $\cA$ and $\widehat\cA$ are closed spaces.
   
Using \eqref{eq:nn} and the correspondence \eqref{eq:nr}, it is easy to check that 
   $$
  \dim \cA_m=\dim\widehat\cA_m,\quad m\in\naturals_0,
  $$
so the vector spaces $\cA_m$ and $\widehat\cA_m$ are isomorphic. However, much more is true, namely that the isomorphism
is given by the Fourier transforms \eqref{eq:ft} and \eqref{eq:ift}.
  \begin{lemma} (a) For all $m\ge 0,$
    \begin{equation}\label{eq:fiso}
     \cF^\sim \cA_m=\widehat\cA_m,\quad\cF^\natural \widehat\cA_m=\cA_m.
    \end{equation}
(b) The spaces $\cA_m$ and $\widehat\cA_m$ are closed with respect to multiplication and convolution of functions.
  \end{lemma}
\begin{proof} (a) Using \eqref{eq:Fal} and the conditions \eqref{eq:triangular}, we conclude that the Fourier transform
of the basis functions $\alpha_0,\alpha_{l,i}, l\le m$ of the space $\cA_m$ expands into a linear combination of the
basis functions $\beta_0,\beta_{k,j}, k\le m$ of the space $\widehat\cA_m,$ which establishes the first relation
in \eqref{eq:fiso}. Likewise, \eqref{eq:Fbe} and \eqref{eq:triangular}
imply the second relation.
\end{proof}

Thus, we have proved that $\cA_m$ and $\cA_m$ are function algebras closed with respect to multiplication and convolution.
These algebras are dual of each other in the sense that the Fourier transform exchanges the convolution and multiplication
operations.
Next we argue that $\cA_m$ and $\widehat\cA_m, m\ge 0$ can be considered as adjacency algebras of finite translation schemes
$\cX^{(m)}$ and $\widehat\cX^{(m)}$ constructed as follows. Consider the following pair of dual finite Abelian groups:
   \begin{equation}\label{eq:fag}
   X^{(m)}=X/B(r_{m+1}), \quad \hat X^{(m)}=B(r_{m+1})^\bot=\hat B(t_m).
   \end{equation}
The finite partitions of the group $X^{(m)}$ into balls
   $$
  B(r_{m+1}), \;\Phi_i(r_l), \;i=1,\dots,n_l, l=1,\dots,m
   $$
and of the group $\hat X^{(m)}$ into balls
   $$
  \hat B(0), \;\hat\Phi_j(t_k), \;j=1,\dots,n_k,k=1,\dots,m
  $$
(see \eqref{eq:Phi}, \eqref{eq:^Phi}, and \eqref{eq:bbk}) are spectrally dual. By Theorem \ref{thm:tnm} (see also \cite{zin09}) these partitions give rise to mutually dual translation association schemes $\cX^{(m)}$ and $\widehat\cX^{(m)}$ with 
$\dim \cA_m=\dim \widehat\cA_m$ classes. The incidence matrices of these schemes are given by
    \begin{align*}
    A_0^{(m)}&=\llbracket \;\chi[B(r_{m+1});x-y]\;\rrbracket=I\\
  A_{l,i}^{(m)}&=\llbracket \;\chi[\Phi_i(r_l);x-y]\;\rrbracket, \quad x,y\in X^{(m)}
   \end{align*}
 and
    \begin{align*}
  B_0^{(m)}&=\llbracket\;\chi[\hat B(0);\phi\xi^{-1}]\;\rrbracket=I\\
  B_{k,j}^{(m)}&=\llbracket\;\chi[\hat\Phi_j(t_k);\phi\xi^{-1}]\;\rrbracket,\quad \phi,\xi\in\hat X^{(m)},
    \end{align*}
    respectively.
    
It is now clear that upon identifying the elements $\beta_0,\beta_{k,j}$ in \eqref{eq:bbk} with the matrices $B_0^{(m)}, B_{k,j}^{(m)},$ we
obtain an isomorphism between $\widehat \cA_m$ and the adjacency algebra of the scheme $\widehat\cX^{(m)}.$
The isomorphism between $\cA_m$ and the adjacency algebra of the scheme $\cX^{(m)}$ can be established in the same way.
However, for reasons that are made clear below we opt for a slightly different mapping.
Namely, let identify
the elements $\alpha_{l,i}$ in \eqref{eq:aal} with the matrices $A_{l,i}^{(m)}.$ At the same time, the element $\alpha_0$ 
is identified not with $A_0^{(m)}$ but rather with the all-one matrix
   $$
   J^{(m)}=A_0^{(m)}+\sum_{l=1}^m\sum_{i=1}^{n_l-1}A_{l,i}^{(m)}
   $$
(cf. \eqref{eq:A}). This identification maps the basis $(\alpha_0,\alpha_{l,i})$ of $\cA_m$ to the basis 
$(J^{(m)},A_{l,i}^{(m)})$ of $\cX^{(m)},$ establishing the claimed 
isomorphism. 

\begin{theorem} The algebra $\cA_m$ ($\widehat \cA_m$) is isomorphic to the adjacency algebra of the scheme $\cX^{(m)}$
(resp., $\widehat\cX^{(m)}$).
\end{theorem}

The reasons for choosing the basis $(J^{(m)},A_{l,i}^{(m)})$ rather than the standard basis $(A_0^{(m)}, A_{l,i}^{(m)})$
are related to the fact that for an infinite uncountable group $X,$ the standard basis ``degenerates'' as $m\to\infty$
because the measure of the ball $\mu(B(r_{m+1}))\to 0$ as $m\to\infty.$ This implies that $A_0^{(m)}\to 0$ in
$L_2(X^{(m)})$ as $m\to\infty.$ At the same time, using our choice of the basis,
transition to the limit is easily accomplished in terms of the graded algebras \eqref{eq:cA<}.
For instance, in the limit, the matrices $A_0^{(m)}$ turn into the operator with the kernel $\chi[B(0);x-y],$ but this
indicator function is not contained in $\cA$ because
   $$
   \chi[B(0);x]=\alpha_0(x)-\sum_{l\in \naturals}\sum_{i=1}^{n_l-1}\alpha_{l,i}(x),
   $$
but the algebra $\cA$ contains only finite linear combinations of the basis elements $\alpha_0,\alpha_{l,i}.$
Essentially, the contents of this section deals with careful formalization of these key observations.

To conclude, let $\cX$ and $\widehat\cX$ be a pair of mutually dual association schemes defined on a zero-dimensional compact group $X$
and its dual group $\hat X$ using construction of Theorem \ref{thm:tnm}. Then the algebras $\cA$ and $\widehat \cA$ defined in \eqref{eq:fA} and \eqref{eq:gA}, respectively, should be viewed as the adjacency algebras of these schemes\footnote{We do not state this and similar later results
as theorems because the adjacency algebras of general infinite association schemes have not been formally defined.}.

%
%
\subsubsection{Adjacency algebras of metric schemes on groups}\label{sect:subA}
In addition to numerous finite-dimensional subalgebras, the algebras $\cA$ and $\widehat \cA$ contain 
mutually dual, countably dimensional subalgebras $\gls{cAsph}\subset \cA$ and $\widehat \cA^{\,\text{(sph)}}\subset \widehat \cA$
related to the spectrally dual partitions of the groups $X$ and $\hat X$ into spheres constructed in Section~\ref{sect:dualpairs}, Thm.~\ref{thm:m1}. In this section we identify these subagebras as adjacency algebras 
of the association schemes related these partitions (see Theorem~\ref{thm:m2}). 

Let 
   $$
  \alpha_l(x)=\sum_{i=1}^{n_l-1}\alpha_{l,i}(x)=\chi[S(r_l);x], \quad l\ge 1,
   $$
(cf. \eqref{eq:St}, \eqref{eq:sph}, \eqref{eq:aal}) and denote by $\cA^{\text{(sph)}}$ the set of all functions
$f:X\to \complexes$ of the form
 \begin{equation}\label{eq:asph}
   f(x)=c_0\alpha_0(x)+\sum_{l\in\naturals} c_l\alpha_l(x),
  \end{equation}
where only finitely many coefficients $c_0,c_l$ are nonzero.

Similarly, let
  $$
  \beta_k(\xi)=\sum_{j=1}^{\hat n_k-1}\beta_{k,j}(\xi)=\chi[\hat S(t_k);\xi], \quad k\ge 1
  $$
and denote by $\widehat \cA^{\,\text{(sph)}}$ the set of all functions $g:\hat X\to\complexes$ of the form
  \begin{equation}\label{eq:bsph}
  g(\xi)=c_0'\beta_0(\xi)+\sum_{k\in\naturals} c_k'\beta_k(\xi),
  \end{equation}
where only finitely many coefficients $c_0',c_k'$ are nonzero.

Denote by $\cA_m^{\text{(sph)}}, m\in \naturals_0$ the set of all functions of the form \eqref{eq:asph}
such that $c_l=0$ for $l>m$ and use the notation $\widehat \cA_m^{\,\text{(sph)}}, m\in \naturals_0$ to refer to the set
of functions of the form \eqref{eq:bsph} such that $c_k'=0$ if $k>m.$

It is obvious that $\cA^{\text{(sph)}}$ and $\widehat \cA^{\,\text{(sph)}}$ are countably dimensional vector spaces
and $\cA_m^{\text{(sph)}}\subset\cA^{\text{(sph)}} $ and $\widehat \cA_m^{\,\text{(sph)}}\subset \widehat \cA^{\,\text{(sph)}}$
are finite-dimensional subspaces of dimension $m+1.$ Moreover, we have
   \begin{equation}\label{eq:AAsph}
   \left.    \begin{array}{c}
   \cA_0^{\text{(sph)}}\subset\cA_1^{\text{(sph)}}\subset \cA_2^{\text{(sph)}}\subset\dots\subset\cA^{\text{(sph)}}\\
   \cA^{\text{(sph)}}=\bigcup\limits_{m\in\naturals_0}\cA_m^{\text{(sph)}}\\
   \widehat \cA_0^{\,\text{(sph)}}\subset\widehat \cA_1^{\,\text{(sph)}}\subset\widehat \cA_2^{\,\text{(sph)}}
     \subset\dots\subset\widehat \cA^{\,\text{(sph)}}\\
      \widehat \cA^{\,\text{(sph)}}=\bigcup\limits_{m\in\naturals}\widehat \cA_m^{\,\text{(sph)}}.
    \end{array}\right\}
  \end{equation}
\begin{lemma}
The Fourier transforms on the spaces considered satisfy the following relations analogous to \eqref{eq:FF} and
\eqref{eq:fiso}:
     \begin{equation}\label{eq:FFAAmm}
     \begin{array}{c}
   \cF^\sim\cA^{\text{(sph)}}=\widehat \cA^{\,\text{(sph)}}, \quad \cF^\natural\widehat \cA^{\,\text{sph}}=\cA^{\text{(sph)}}\\
    \cF^\sim\cA_m^{\text{(sph)}}=\widehat \cA_m^{\,\text{(sph)}}, \quad \cF^\natural\widehat \cA_m^{\,\text{sph}}=\cA_m^{\text{(sph)}}.
    \end{array} 
     \end{equation}
\end{lemma}
{\em Proof:} (outline) Relations \eqref{eq:FFAAmm} are proved by a direct computation using \eqref{eq:Fal}, \eqref{eq:Fbe}
in order to compute the Fourier transforms of the basis functions $\alpha_l$ and $\beta_k$, taking account of the 
orthogonality relations of the characters $\omega_{ij}(r)$ and $\hat\omega_{ji}(t)$ \eqref{eq:oo1}, \eqref{eq:oo}.
Another option is to use directly expressions for the Fourier transforms of the indicator functions of spheres $S(r)$
and $\hat S(r)$ given in Lemma \ref{lemma:fc}. Equation \eqref{eq:1-} needed to transform infinite sums of indicators
into finite ones in this case takes the following form:
   $$
   \sum_{l>k}\chi[S(r_l);x]=1-\sum_{l=1}^k\chi[S(r_l);x].
  $$
Details can be left to the interested reader.

\vspace*{.1in}
Since each of the spaces $\cA^{\text{(sph)}}, \widehat \cA^{\,\text{(sph)}},\cA_m^{\text{(sph)}},
\widehat\cA_m^{\,\text{(sph)}}, m\in \naturals_0$ is closed under the usual multiplication of functions, these relations
imply that these spaces are also closed with respect to convolution. Thus, $\cA^{\text{(sph)}}, \widehat \cA^{\,\text{(sph)}},\cA_m^{\text{(sph)}},
\widehat\cA_m^{\,\text{(sph)}}, m\in \naturals_0$ ara algebras of functions closed with respect to both multiplication
and convolution. By \eqref{eq:AAsph} the countably dimensional algebras $\cA^{\text{(sph)}}$ and $\widehat \cA^{\,\text{(sph)}}$
are graded by the finite-dimensional subalgebras $\cA_m^{\text{(sph)}}, m\in\naturals_0$ and $\widehat\cA_m^{\,\text{(sph)}}, m\in \naturals_0.$ It is easy to check that the algebras $\cA_m^{\text{(sph)}}$ and $\widehat\cA_m^{\,\text{(sph)}}$
are isomorphic  to the adjacency algebras of association schemes constructed from partitions of the finite Abelian groups
\eqref{eq:fag} following Theorem \ref{thm:m2}. 

This enables one to state the main result of this section. Let $X$ be a compact zero-dimensional Abelian group, let $\hat X$ be its dual group, and let $\cX$ and $\widehat\cX$ be metric schemes introduced in Sect.~\ref{sect:dualpairs}, Theorem~\ref{thm:m2}.
Then the algebras $\cA^{\text{(sph)}}$ and $\widehat \cA^{\,\text{(sph)}}$ should be viewed as the adjacency algebras of these schemes.

\subsection{Locally compact groups}
Let us discuss changes in the construction of adjacency algebras needed to cover the case of locally
compact groups. Here we confine ourselves to brief remarks. 
Let $X$ be a locally compact uncountable zero-dimensional group, and let the topology on $X$ be defined by the chain
of nested balls $B(r), r\in \I.$ The dual group $\hat X$ is also locally compact uncountable and zero-dimensional,
with topology defined by the chain of nested balls $\hat B(t), t\in\hat\I.$ 

Consider the spectrally dual partitions \eqref{eq:nmp} and \eqref{eq:^nmp} of $X$ and $\hat X$ into
balls $\Phi_i(r)$ and $\hat\Phi_j(r)$ and the corresponding mutually dual
association schemes $\cR$ and $\hat\cR$ constructed in Theorem \ref{thm:tnm}. The adjacency algebras of these
schemes can be described as follows.

Let us number the radii $r\in \R_0$ in the descending order and the radii $t\in\hat\R_0$ in the ascending order:
   \begin{equation*}\begin{array}{c}
   0<\dots<r_1<r_1<r_0<r_{-1}<r_{-2}<\dots\\
   0<\dots<t_{-2}<t_{-1}<t_0<t_1<t_2<\dots .
   \end{array}
  \end{equation*}
Choose $r_0$ and $t_0$ so that $B(r_0)^\bot=\hat B(t_0),$ i.e., 
    $$
  \tilde r_0=t_0,\quad t_0^\natural=r_0.
    $$
Then $B(r_m)^\bot=\hat B(t_m)$ and
   $$
  \tilde r_m=t_m,\quad t^\natural_m=r_m \quad\text{for all }m\in\integers,
   $$
see \eqref{eq:rtr}, \eqref{eq:BB}.

Let 
   $$
   \begin{array}{l}
   \alpha_0(x)=\chi[B(r_0);x],\\
   \alpha_{l,i}(x)=\chi[\Phi_i(r_l);x], \quad i=1,\dots,n_l-1, n_l=n(r_l), l\in\integers.
   \end{array}
   $$
Denote by $\cA$ the set of all functions $f:X\to\complexes$ of the form 
   \begin{equation}\label{eq:fff}
   f(x)=c_0\alpha_0(x)+\sum_{l\in\integers}\sum_{i=1}^{n_l-1} c_{l,i}\alpha_{l,i}(x),
    \end{equation}
where only finitely many coefficients $c_0,c_{l,i}$ are nonzero.

Likewise, let
   $$ 
   \begin{array}{c}
   \beta_0(\xi)=\chi[\hat B(t_0);\xi]\\
   \beta_{k,j}(\xi)=\chi[\hat\Phi_j(t_k);\xi], \quad j=1,\dots,\hat n_k-1, \hat n_k=\hat n(t_k), k\in\integers
  \end{array}
  $$
and denote by $\widehat\cA$ the set of all functions $g:\hat X\to\complexes$ of the form
   \begin{equation}\label{eq:ggg}
  g(\xi)=c_0'\beta_0(\xi)+\sum_{k\in\integers}\sum_{j=1}^{\hat n_k-1} c'_{k,j}\beta_{k,j}(\xi),
   \end{equation}
where only finitely many coefficients $c_0,c_{k,j}'$ are nonzero.

It is clear that $\cA$ and $\widehat \cA$ are countably dimensional vector spaces. The following lemma is verified by
direct computation.
\begin{lemma} $\cF^\sim\cA=\widehat\cA$ and $\cF^\natural \widehat\cA=\cA.$\end{lemma}
To prove this is suffices to use expressions \eqref{eq:mu1} for the Fourier transforms of the indicator functions of the balls $B(r_0),\hat B(t_0)$
 and expressions \eqref{eq:cP} and \eqref{eq:^cP} for the Fourier transforms of the indicator functions of the balls
$\Phi_i(r_l), \hat\Phi_j(t_k).$ Infinite sums of indicator functions in this case are transformed into finite sums
using the following relations:
  $$
  \sum_{l\ge 0}\sum_{i=1}^{n_l-1}\chi[\Phi_i(r_l);x]=\chi[B(r_0);x]
  $$
implying
  \begin{align*}
  \sum_{l>k}\sum_{i=1}^{n_l-1}\chi[\Phi_i(r_l);x]=\chi[B(r_0);x]-\sum_{0\le l\le k}\sum_{i=1}^{n_l-1} \chi[\Phi_i(r_l);x]
   \quad \text{for }k\ge 0\\
\sum_{l>k}\sum_{i=1}^{n_l-1}\chi[\Phi_i(r_l);x]=\chi[B(r_0);x]+\sum_{k<l\le 0}\sum_{i=1}^{n_l-1} \chi[\Phi_i(r_l);x]
  \quad \text{for }k<0.
  \end{align*}
Similarly,
   $$
  \sum_{k\le 0}\sum_{j=1}^{\hat n_k-1}\chi[\hat\Phi_j(t_k);\xi]=\chi[\hat B(t_0;\xi]
  $$
implying
   \begin{align*}
   \sum_{k<l}\sum_{j=1}^{\hat n_l-1}\chi[\hat\Phi_j(t_k);\xi]=\chi[\hat B(t_0;\xi]
   +\sum_{l\le k\le0}\sum_{j=1}^{n_l-1}\chi[\hat\Phi_j(t_k;\xi] \quad\text{for }l\le 0\\
    \sum_{k<l}\sum_{j=1}^{\hat n_l-1}\chi[\hat\Phi_j(t_k);\xi]=\chi[\hat B(t_0;\xi]
      -\sum_{0\le k<l}\sum_{j=1}^{n_l-1}\chi[\hat\Phi_j(t_k;\xi] \quad\text{for }l> 0.
    \end{align*}  

The spaces $\cA$ and $\widehat \cA$ are closed with respect to the usual product of functions, and therefore, by the
previous lemma, are also closed with respect to convolutions.

The main result of this section is stated as follows. Let $X$ and $\hat X$ be a pair of mutually dual locally compact uncountable zero-dimensional Abelian groups
and let $\cX$ and $\widehat \cX$ be the association schemes constructed in Theorem \ref{thm:tnm}. The function algebras
$\cA$ and $\hat\cA$ should be viewed as the adjacency algebras of these schemes.

Note that the algebras $\cA$ and $\hat \cA$ can also be graded by finite-dimensional subalgebras. For instance, it is
sufficient to introduce subalgebras $\cA_m,m\in\naturals_0$ formed of the functions \eqref{eq:fff} with coefficients $c_{l,i}=0$
for $|l|>m$ and subalgebras $\widehat\cA_m,m\in\naturals_0$ formed of the functions \eqref{eq:ggg} with coefficients
$c'_{k,j}=0$ for $|k|>m.$ We do not go into further details here, hoping to cover this range of questions in a separate
publication.

%
\section{Eigenvalues and harmonic analysis}\label{sect:eigenvalues}
One of the first questions that arise in the study of a new class of association schemes is whether these schemes are polynomial, i.e., whether the first and the second eigenvalues \eqref{eq:AE}, \eqref{eq:EA} coincide with values
of some orthogonal polynomials of one discrete variable. 
In particular, a finite scheme is metric (i.e., a distance-regular graph) if and only if
it is $P$-polynomial. The $Q$-polynomiality property is much more intricate and is discussed in detail in \cite{ban84} as
well as in a large number of more recent papers. The theory of orthogonal polynomials provides tools for a detailed study 
of polynomial schemes such as the classical Hamming and Johnson schemes in coding theory \cite{del73a,ban84,del98}. 

{Turning to the case of schemes on zero-dimensional groups, we have observed in 
Sect.~\ref{sect:npol} that they cannot be polynomial because even their finite subschemes on Abelian groups 
with a non-Archimedean metric are not polynomial.
At the same time, the eigenvalues satisfy orthogonality relations \eqref{eq:pipj}, \eqref{eq:qiqj}, and therefore we are faced with the 
question of characterizing the class of orthogonal functions whose evaluations coincide with the eigenvalues.
} In this section we show that eigenvalues of the metric schemes defined in Sections \ref{sect:dualpairs} and \ref{sect:metric} have a natural interpretation in the framework of basic harmonic analysis, more specifically, the Littlewood-Paley theory, and outline its link with the theory of martingales.
 
In order to study $p$- and $q$-coefficients of nonmetric schemes constructed by partitioning the groups into balls,
see Sect.~\ref{sect:nonmetric}, we introduce a new class of orthogonal systems of functions, calling them {\em Haar-like
bases}. These systems have all the remarkable properties of wavelets, except that generally they are not self-similar.
Self-similarity is preserved only for special zero-dimensional groups with self-similar chains of nested subgroups, in which
case the bases that we define coincide with piecewise constant wavelets.

Of course, our paper is not a specialized study in harmonic analysis, so we do not discuss a number of important questions
related to the function systems. In particular, we study the $p$- and $q$-coefficients in the Hilbert space $L_2,$
while possible generalization to the spaces $L_\alpha, 1\le \alpha<\infty$ is mentioned only very briefly.

\subsection{Eigenvalues of metric schemes and Littlewood-Paley theory}
We use the notation of Section \ref{sect:vil}, starting with the zero-dimensional topological Abelian group $X.$ 
We also identify the subgroups $X_j\in X$
in the chains \eqref{eq:chain} and \eqref{eq:dc} with the metric balls in the corresponding non-Archimedean metric,
see \eqref{eq:b}, \eqref{eq:^b}. We have
  \begin{equation}\label{eq:BB1}
  \begin{array}{l}   B(r_1)\subset B(r_2), \quad r_1<r_2, \;r_1,r_2\in\I_0\\
     \bigcap_{r\in\I_0} B(r)=\{0\}, \quad \bigcup_{r\in \I_0}B(r)=X,
  \end{array}
  \end{equation}
where $\I_0=\{r:\mu(B(r))>0\}.$ The set $\I_0$ is at most countable and can have only one accumulation point $r=0.$
The quotient groups $B(r_2)/B(r_1), r_1<r_2$ are finite, and the group $X/B(r)$ is finite if $X$ is compact and is countable
and discrete if $X$ is locally compact. Analogous notation from Section \ref{sect:vil} is used for the dual group $\hat X.$

\subsubsection{Littlewood-Paley theory} For every $r\in\I_0$ define a partition of the group $X$ into balls $B(r)$:
  \begin{equation}\label{eq:pb1}
  K(r)=\{B(r)+z, z\in X/B(r)\}
  \end{equation}
  (cf. \eqref{eq:d} and \eqref{eq:XH}). 
The nesting \eqref{eq:BB1} implies that for $r_1<r_2$ the partition $K(r_1)$ is a refinement of the partition $K(r_2).$
In this case we write $K(r_1)\prec K(r_2), r_1<r_2.$
Our general idea in this section is to study nested chains of functional spaces formed by functions constant on the partitions of
the form \eqref{eq:pb1}. 

Let $f$ be a locally summable function $X$. Consider its  approximation $\gls{cErf}$ by a function piecewise constant on the 
partition \eqref{eq:pb1}:
    \begin{equation}\label{eq:appr}
  \cE_rf(x):=\mu(B(r))^{-1}\int_{B(r)+z} f(y)d\mu(y), \quad x\in B(r)+z.
    \end{equation}
This expression can be written as an integral operator 
  \begin{equation}\label{eq:io}
  \cE_rf(x)=\int_X \cE_r(x,y)f(y)d\mu(y)
  \end{equation} 
with kernel
   \begin{equation}\label{eq:kernel}
  \cE_r(x,y)=\mu(B(r))^{-1}\sum_{z\in X/B(r)} \chi[B(r)+z;x]\chi[B(r)+z;y].
  \end{equation}
\begin{lemma}
Operator $\cE_r$ is a convolution
  \begin{equation}\label{eq:cE}
  \cE_rf(x)=\mu(B(r))^{-1} \int_X\chi[B(r);x-y]f(y)d\mu(y)
  \end{equation}
\end{lemma}
\begin{proof}
  Since $K(r)$ is a partition of $X$ into cosets of the subgroup $B(r),$ for a given $y\in B(r)+z$ for some $z\in X/B(r)$
we have $x-y\in B(r)$ if and only if $x\in B(r)+z$ for the same $z$. This means that\footnote{Eq.~\eqref{eq:cBxy} will be useful on its own in our calculations for nonmetric schemes in the next section.}
  \begin{equation}\label{eq:cBxy}
  \chi[B(r);x-y]=\sum_{z\in X/B(r)} \chi[B(r)+z;x]\chi[B(r)+z;y].
  \end{equation}
Replacing the right-hand side of \eqref{eq:kernel} with $\chi[B(r);x-y]$, we obtain the claimed expression.
Finally note that the series and integrals are well defined because $K(r)$ is a partition into compact subsets of
finite measure.
\end{proof}

{\em Remark:} Equation \eqref{eq:cBxy} will be useful on its own in our definition of Haar-like wavelets on zero-dimensional
groups in Sect.~\ref{sect:nH} below.

\vspace*{.05in}
Expression \eqref{eq:cE} implies the following important {\em martingale property} of the approximations $\cE_r:$
\index{martingale property}
   \begin{equation}\label{eq:mp}
  \cE_{r_1}\cE_{r_2}=\cE_{\max\{r_1,r_2\}}.
   \end{equation}
In particular, the mappings $\cE_{r_1}$ and $\cE_{r_2}$ commute for all $r_1,r_2\in\I_0$ and are idempotent:
   $$
  \cE_{r_1}\cE_{r_2}=\cE_{r_2}\cE_{r_1}, \quad \cE_{r}^2=\cE_r.
   $$

We have defined the mappings $\cE_r,r\in\I_0$ on the class of all locally summable functions. Now consider them
as operators in the space $L_\alpha(X), 1\le \alpha<\infty.$ The Young inequality \eqref{eq:Young} implies that in this space
the operators $\cE_r,r\in\I_0$ are linear and bounded. The operators $\cE_r$ are projectors on the space 
$\cL_\alpha[K(r)]\subset L_\alpha(X)$ of functions piecewise constant on the partition $K(r):$
   \begin{equation}\label{eq:Lalpha}
  \cL_\alpha[K(r)]=\Big\{g\in L_\alpha(X): g(x)=\sum_{z\in X/B(r)} c_z\chi[B(r)+z;x]\Big\},
   \end{equation}
where $c_z\in\complexes$ and the $L_\alpha$-norm of $g\in\cL_\alpha$ is defined by the expression
   $$
  \|g\|_\alpha=\Big(\mu(B(r))\sum_{z\in X/B(r)}|c_z|^\alpha\Big)^{1/\alpha}.
   $$
In the Hilbert space $L_2(X)$ the operators $\cE_r$ are orthogonal projectors on the subspace $\cL_2[K(r)].$

\vspace*{.1in}\nd{\em Properties of the spaces $\cL_\alpha$.}

(i) On account of  \eqref{eq:mp}, the spaces $\cL_\alpha[K(r)], r\in\I_0$ form a chain of nested subspaces, viz.
  \begin{equation}\label{eq:mra0}
  \cL_\alpha[K(r_1)]\supset \cL_\alpha[K(r_2)] \quad\text{if }r_1<r_2.
  \end{equation}

(ii) \begin{equation}\label{eq:mra1}
  \bigcap_{r\in \I_0}\cL_\alpha[K(r)]=\begin{cases} \text{Const}_X &\text{if $X$ compact}\\
    \{0\} &\text{if $X$ is locally compact.}
   \end{cases}
  \end{equation}
Here $\text{Const}_X$ denotes a one-dimensional space of functions constant on $X$.
Indeed, functions contained in all the spaces $\cL_\alpha[K(r)], r\in\I_0$ are necessarily identically
constant on $X$, but such functions are in $L_\alpha$ only if $X$ is compact.

(iii) The Banach spaces $L_\alpha(X), 1\le \alpha<\infty$ are separable, and the union of spaces \eqref{eq:Lalpha}
is dense in each of them: namely,
  \begin{equation}\label{eq:dense}
  \overline{\bigcup_{r\in\I_0}\cL_\alpha[K(r)]}=L_\alpha(X).
  \end{equation}
  where the overbar means closure.
This follows from our assumptions of $X$ being a second-countable space whose topology is given by the chain of nested subgroups
(see \cite{aga81}, Ch.2 for details).

\begin{lemma}\label{lemma:01}
(a) 
The operators $\cE_r$ strongly converge to identity as $r\to 0$, i.e., 
for any function $f\in L_\alpha(X),1\le \alpha<\infty$
   $$
   \cE_rf\to f \quad\text{if }r\to 0
   $$
in the metric of $L_\alpha(X).$ 

(b) If $X$ is compact, then $\cE_{\bar r}$ is a projector in $L_\alpha(X)$ on $\text{Const}_X.$ If $X$ is locally compact
and $\bar r=\infty$ (see \eqref{eq:br}), then the operators $\cE_r$ strongly converge to $0$ in $L_\alpha(X),$ i.e., for any function 
$f\in L_\alpha(X), 1<\alpha<\infty$
    $$
    \cE_rf\to 0 \quad\text{if }r\to \infty
    $$
    (here it is essential that $\alpha>1$).
\end{lemma}
\begin{proof}
(a) All the operators $\cE_r, r\in \I_0$ are uniformly bounded, so it suffices
to prove convergence on a dense subset in $L_\alpha(X).$ Let $f\in \cL[K(r_0)],$ then \eqref{eq:mp} implies that
$\cE_r f=f$ for all $r<r_0$, which together with \eqref{eq:dense} proves the convergence. If $X$ is discrete, then 
$r=0\in \I_0,$ and we can simply write $\cE_0=I.$

(b) The first claim is obvious since $\bar r<\infty.$ Turning the second claim, it also suffices to prove it on a dense subset of functions in $L_\alpha(X).$
Consider the set of functions on $X$ with compact support. Suppose that a function $f\in L_\alpha(X)$ is supported on a ball
$B(r_0)$ of some radius $r_0,$ then for $r>r_0$ Eq.~\eqref{eq:appr} implies
  $$
  \cE_rf(x)=\begin{cases} \displaystyle{\mu(B(r))^{-1}\int_{B(r_0)} f(y) d\mu(y)} &\text{if } x\in B(r)\\
     0&\text{if }x\not\in B(r).
     \end{cases}
  $$
At the same time, using H{\"o}lder's inequality, we obtain
  $$
  \Big|\int_{B(r_0)} f(y)d\mu(y)\Big|\le \mu(B(r_0))^{1-\frac 1\alpha}\|f\|_\alpha.
  $$
The last two equations imply that
  $$
  \|\cE_rf\|_\alpha\le \Big(\frac{\mu(B(r_0))}{\mu(B(r))}\Big)^{1-\frac 1\alpha}\|f\|_\alpha\to 0
  $$
if $r\to \infty$ and $\alpha>1.$ The lemma is proved.
\end{proof}

\vspace*{.1in}
The next definition is important for our goal of characterizing the eigenvalues of the scheme $\cX.$ Consider 
the ``increment'' of the operator $\cE_r$ when the value of the radius changes from $r$ to the next value:
  \begin{equation}\label{eq:Delta}
 \Delta_r=\cE_r-\cE_{\tau_+(r)}, \quad r\in\I_0
  \end{equation}
(cf. \eqref{eq:pn}). 
If $X$ is compact, and $r=\bar r,$ then define $\Delta_{\bar r}=\cE_{\bar r}.$
Another way to write $\Delta_r$ follows if in \eqref{eq:Delta} we take into account \eqref{eq:io}-\eqref{eq:kernel}:
we see that it is an integral operator
  \begin{equation}\label{eq:io1}
  \Delta_rf(x)=\int_X\Delta_r(x,y)f(y)d\mu(y)
  \end{equation}
with kernel $\Delta_r(x,y)=\Delta_r(x-y),$ where
  \begin{equation}\label{eq:kernel1}
  \Delta_r(z)=\mu(B(r))^{-1}\chi[B(r);z]-\mu(B(\tau_+(r)))^{-1}\chi[B(\tau_+(r));z].
  \end{equation}
Importantly, the kernels $\Delta_r(x,y)$ are real symmetric.  

Since $\cL_2[K(\tau_+(r))]\subset\cL_2[K(r)],$ we can write 
  \begin{equation}\label{eq:Wr}
   \cL_2[K(r)]=\cL_2[K(\tau_+(r))]\oplus \cW_r,
   \end{equation}
where $\cW_r$ is the orthogonal complement of $\cL_2[K(\tau_+(r))]$ in $\cL_2[K(r)].$

\vspace*{.1in}
\begin{lemma}\label{lemma:delta} 
(a) The operators $\Delta_r$ are  commuting orthogonal projectors from $L_2(X)$ to $\cW_r,$
and
  \begin{equation}\label{eq:D0}
    \Delta_r\cE_{\tau_+(r)}=0, \quad r\in\I_0.
  \end{equation}
(b) The operators $\Delta_r,r\in\I_0$ form a complete system of orthogonal projectors, i.e., any function $f\in L_2(X)$ can be
written as 
  \begin{equation}\label{eq:LP1}
  f=\sum_{r\in\I_0} \Delta_rf,
  \end{equation}
  where the equality is understood in the $L_2$ sense.
\end{lemma}
\begin{proof} 
(a) First we prove that the $\Delta_r$ are pairwise orthogonal
idempotents:
  \begin{equation}
  \Delta_{r_1}\Delta_{r_2}=\delta_{r_1,r_2}\Delta_{r_1}
  \end{equation}
Suppose that $r_1=r_2=r,$ then using \eqref{eq:mp} we find
   $$
  \Delta^2_r=\cE_r+\cE_{\tau_+(r)}-2\cE_r\cE_{\tau_+(r)}=\cE_r-\cE_{\tau_+(r)}=\Delta_r.
   $$
Now suppose that $r_1<r_2,$ then
   \begin{align*}
  \Delta_{r_1}\Delta_{r_2}&=\cE_{r_1}\cE_{r_2}-\cE_{r_1}\cE_{\tau_+(r_2)}-\cE_{\tau_+(r_1)}\cE_{r_2}+\cE_{\tau_+(r_1)}
   \cE_{\tau_+(r_2)}\\
   &=\cE_{r_2}-\cE_{\tau_+(r_2)}-\cE_{r_2}+\cE_{\tau_+(r_2)}=0.
    \end{align*}
Relation \eqref{eq:D0} is proved analogously to the above calculation. Of course, in the compact case we must 
put $\cE_{\tau_+(\bar r)}=0$ by definition.

Since the operators $\Delta_r$ are obviously self-adjoint, this implies that they are (commuting) orthogonal projectors,
and \eqref{eq:D0} implies that the range of $\Delta_r$ is $\cW_r.$

(b) Let us start with the compact case. Let us number the radii in $\I_0$ in decreasing order: $\bar r=r_0>r_1>\dots$.
Consider the finite sum 
  $$
  \sum_{0\le j\le J} \Delta_{r_j} f=\Big(\cE_{r_0}+\sum_{i=1}^{J}(\cE_{r_i}-\cE_{r_{i-1}})\Big) f\to f \quad 
\text{if }r_{_{\!J}}\to 0
  $$
where the convergence follows from Lemma \ref{lemma:01}(a).

Now let $X$ be locally compact. Again number the radii in decreasing order: $\dots r_j>r_{j+1}>\dots,$ and note that
$\lim_{j\to\infty} r_j=0$ and $\lim_{j\to-\infty}r_j=\infty.$ Consider the finite sum
  $$
  \sum_{J_1\le j\le J_2} \Delta_{r_j} f=\Big(\sum_{i=J_2}^{J_1}(\cE_{r_{i}}-\cE_{r_{i+1}})\Big) f=-\cE_{r_{_{\!J_2+1}}}f+\cE_{r_{_{\!J_1}}}f
  \to f
  $$
where the convergence for $r_{_{\!J_1}}\to 0$ and $r_{_{\!J_2}}\to\infty$ is implied by Lemma \ref{lemma:01}. The lemma is proved.
\end{proof}

\remove{ Using \eqref{eq:pri} with $r_1=r$ and $ r_2=\tau_+(r),$
we can write
   $$
   \cL_2[K(r)]=\cL_2[K(\tau_+(r))] \stackrel{\cdot}{+}\cW_r
   $$}
Formula \eqref{eq:LP1} is known as the Littlewood-Paley expansion, and the collection of functions \eqref{eq:kernel1} forms
the {\em Littlewod-Paley} basis of wavelets \cite[p.115]{Daub92}.

From part (a) of Lemma \ref{lemma:delta} we conclude (by the Pythagorean theorem) that
  \begin{equation}\label{eq:rI0}
  \sum_{r\in \I_0}\|\Delta_r f\|_2^2=\|f\|_2^2.
  \end{equation}
For $f\in L_2(X)$ define a {\em quadratic function}
   \begin{equation}\label{eq:SI}
   S_\Delta f(x)=\Big(\sum_{r\in \I_0} |\Delta_r f(x)|^2\Big)^{\half},
   \end{equation}
so that
  \begin{equation}\label{eq:J1}
  \|S_\Delta f\|_2=\|f\|_2.
  \end{equation}
Observe that the sum of the series in \eqref{eq:SI} is well defined because the terms are nonnegative. Of course, it can be infinite,
but the last equality shows that it is finite almost everywhere on $X.$

While the above results are very simple,
one of the basic facts of harmonic analysis asserts that the Littlewood-Paley expansion \eqref{eq:LP1} is preserved if the 
Hilbert space $L_2(X)$ is replaced with a Banach space $L_\alpha(X), 1<\alpha<\infty$ (the condition $\alpha>1$ is important
here). Namely, for any function $f\in L_\alpha(X), 1<\alpha<\infty$ we have the {\em Littlewood-Paley inequality}
\index{inequality!Littlewood-Paley}
  \begin{equation}\label{eq:LP}
  c_1(\alpha)\|S_\Delta f\|_\alpha\le \|f\|_\alpha\le c_2(\alpha)\|S_\Delta f\|_\alpha,
  \end{equation}
where $c_1(\alpha),c_2(\alpha)$ are positive constant that do not depend on $f$. This inequality goes back to a 
1932 work of Paley \cite{Paley32} where it was proved for the zero-dimensional dyadic group. This result was later
generalized to large classes of topological Abelian groups. A detailed exposition of the Littlewood-Paley theory 
as well as relevant references are found in the monograph \cite{edw77}.

\vspace*{.1in} These considerations take us to the main result of this section, a characterization of the eigenvalues of metric schemes on 
zero-dimensional groups.
\begin{theorem}\label{thm:LP} Consider a metric scheme $\cX=(X,\mu,\cR)$ on a zero-dimensional group $X$ defined 
in Theorem \ref{thm:m2} and let $q_r(a),a\in \I_0$ be the second eigenvalues of $\cX;$ see \eqref{eq:Ei2} and \eqref{eq:qik3}.  
Suppose that $(x,y)\in R_a, a\in \I_0.$ Then the kernel of the projector 
$\Delta_r, r\in \I_0$ can be written as
   \begin{equation}\label{eq:qDelta}
  \Delta_r(x,y)=q_{\tilde r}(a),
     \end{equation}
where $\tilde r$ is the ``dual'' radius defined in \eqref{eq:BB}.
\end{theorem}
\begin{proof} As established in  \eqref{eq:cE} and \eqref{eq:io1}-\eqref{eq:kernel1}, the $\cE_r$ and $\Delta_r$ are convolution operators.
Recalling the Fourier transform expression for the ball \eqref{eq:mu1} (and \eqref{eq:Fc}), we obtain
   $$
   \tilde{\cE_r f}(\xi)=\mu(B(r))^{-1}\tilde\chi[B(r);\xi]\tilde f(\xi)=\chi[\hat B(\tilde r);\xi]\tilde f(\xi)
   $$
and thus,
  \begin{equation}\label{eq:ddn}
  \tilde{\Delta_r f}(\xi)=\Big(\chi[\hat B(\tilde r);\xi]-\chi[\hat B(\tau_-(\tilde r));\xi]\Big)\tilde f(\xi)=\chi[\hat S(\tilde r);\xi]\tilde f(\xi).
  \end{equation}
Recall that $\hat B(\cdot)$ and $\hat S(\cdot)$ are balls and spheres in the dual group $\hat X$ and note the use of \eqref{eq:+-} and Lemma \ref{lemma:fc}. 
Now compute the inverse Fourier transform \eqref{eq:ift} of the right-hand side of \eqref{eq:ddn} and use \eqref{eq:io1} to claim that
  \begin{equation}\label{eq:639}
  \Delta_r(x,y)=\int_{\hat X}\overline{\xi(x-y)} \chi[\hat S(\tilde r);\xi]d\hat\mu(\xi)=\chi^\natural (\hat S(\tilde r); x-y).
  \end{equation}
The right-hand side of the last expression is familiar from Section~\ref{sect:dualpairs} where it arises in studying partition of $X$ into
spheres. In particular, from \eqref{eq:sdfb} we obtain the following expression:
  $$
  \Delta_r(x,y)=\sum_{a\in \I} q_{\tilde r}(a)\chi[S(a);x-y].
  $$
Of course, there is only one value of $a$ for which the term on the right is nonzero, namely the one with $x-y\in S(a),$ i.e., 
such that $(x,y)\in R_a.$
\end{proof}

\subsubsection{Metric schemes and martingales} Considerations of the previous section permit a useful reformulation in terms of 
martingale theory. We digress briefly from the main subject of the paper to discuss this topic. Note that the connection between the theory developed above
and martingales is apparent already from expressions like \eqref{eq:io} and related nested partitions (more on this below).
Martingale extensions of the Littlewood-Paley theory are associated primarily with the work of Burkholder (\cite{Burk73,Burk91}) 
and are discussed in \cite[Ch.~5]{edw77}. See also \cite[Ch.~11]{ChowTeicher97} as well as the survey \cite{Peshkir95} which offers several 
different perspectives of martingale theory.
 
Let $X$ be a compact group and let the radii of the balls (nested subgroups) be numbered in decreasing order: $\bar r=r_0>r_1>r_2>\dots.$
Suppose that the Haar measure is normalized by the condition $\mu(X)=1$. Let us view $(X,\mu)$ as a probability space
equipped with a filtration of increasingly refined partitions $K_j=K(r_j), j\in \naturals_0$ (see \eqref{eq:pb1}): 
    \begin{equation}\label{eq:partition}
 K_0\prec K_1\prec K_2,\dots.
   \end{equation}
Let us write $\cE_j$ instead of $\cE_{r_j}$, then the martingale property \eqref{eq:mp} takes the form
   \begin{equation}\label{eq:mp1}
   \cE_{j_1}\cE_{j_2}=\cE_{\min\{j_1,j_2\}}.
   \end{equation}
Let $f: X\to \complexes$ be a measurable function (a random variable). The expectation of $f$ equals
   $$
   f_0=\cE_0f=\int_X f(x) d\mu(x)
   $$
and its conditional expectation with respect to the $\sigma$-algebra generated by the blocks of the partition $K_j$ is precisely
   $$
   f_j=\cE_j f,\quad j\in \naturals.
   $$
By \eqref{eq:mp1} the random variables $f_i, i\in \naturals$ have the property 
  \begin{equation}\label{eq:mp2}
  \cE_j f_i=f_j \quad\text{if }i\ge j;
  \end{equation}
in particular, 
   \begin{equation}\label{eq:mp3}
   \cE_j f_{j+1}=f_j, \quad j\in \naturals_0.
   \end{equation}
A sequence of random variables $f=(f_0,f_1,f_2,\dots)$ is called a {\em martingale} on $X$ with respect to the filtration \eqref{eq:partition}
if it satisfies the following conditions:\\[.05in]
(i) for each $j\in\naturals_0$ the random variable $f_j$ is measurable with respect to $K_j,$ i.e., 
$f_j$ is constant on the blocks of the partition;\\
(ii) the random variables $f_j$ satisfy condition \eqref{eq:mp3},\\[.05in]
If the sequence $(f_j)$
is generated by a single function, it is called a {\em Doob martingale} \index{Doob martingale}
(also called a L{\'e}vy martingale). It is precisely Doob martingales that we were 
considering earlier in this paper, calling them piecewise constant functions.

For an arbitrary martingale on $X$ define a {\em martingale difference sequence:}
  \begin{equation}\label{eq:md}
  \Delta f_0=f_0, \quad\Delta f_j=f_j-f_{j-1}, \quad j\in \naturals
  \end{equation}
  and a sequence of {\em quadratic variations:}
  \begin{equation}\label{eq:qv}
  [f]_J=\sum_{i=0}^J |\Delta f_j|^2, \quad j\in \naturals_0.
  \end{equation}
  Assume for simplicity that the random variables $f_j$ are real-valued and $\Delta f_0=f_0=0.$

\begin{lemma} The martingale differences are uncorrelated, i.e.,
   \begin{equation*}\label{eq:ll1}
   \cE_0 \Delta f_i\Delta f_j=0, \quad{i\ne j}
   \end{equation*}
and
  \begin{equation*}\label{eq:ll2}
    \cE_0(\Delta f_j)^2=\cE_0 f_j^2-\cE_0 f_{j-1}.
  \end{equation*}
\end{lemma}
\begin{proof} Let $g$ be an arbitrary function on $X$ and let $f$ be a function piecewise constant on the blocks of the partition $K_i$.
Then 
   \begin{equation}\label{eq:hih}
    \cE_i gh = h\cE_i g.
   \end{equation}
By the martingale property,
  \begin{equation}\label{eq:mp0}
  \cE_0=\cE_0\cE_i, \quad i\in \naturals_0.
  \end{equation}
Using  \eqref{eq:md}, we find
  \begin{equation}\label{eq:56}
  \cE_0\Delta f_i\Delta f_j=\cE_0 f_i f_j-\cE_0 f_i f_{j-1}-\cE_0 f_{i-1}f_j+\cE_0 f_{i-1}f_{j-1}.
  \end{equation}
Suppose that $i>j$. Using \eqref{eq:mp2} together with \eqref{eq:hih} and \eqref{eq:mp0}, we find
  \begin{align*}
  &  \cE_0 f_if_j=\cE_0 \cE_j f_if_j-\cE_0 f_j\cE_j f_i=\cE_0 f_j^2\\
  &\cE_0 f_if_{j-1}=\cE_0\cE_{j-1} f_i f_{j-1}=\cE_0 f_{j-1}\cE_{j-1}f_i=\cE_0 f_{j-1}^2\\
  &\cE_0 f_{i-1}f_j=\cE_0\cE_j f_{i-1}f_j=\cE_0 f_j\cE_j f_{i-1}=\cE_0 f_j^2\\
  &\cE_0 f_{i-1}f_{j-1}=\cE_0\cE_{j-1} f_{i-1}f_{j-1}=\cE_0 f_{j-1}\cE_{j-1} f_{i-1}=\cE_0 f_{j-1}^2.
    \end{align*}
Inserting these formulas into \eqref{eq:56}, we obtain the first part of the Lemma.

The second part is proved analogously. We have 
  \begin{equation}\label{eq:58}
  \cE_0(\Delta f_j)^2=\cE_0 f_j^2- 2\cE_0 f_j f_{j-1}+\cE_0 f_{j-1}^2.
  \end{equation}
Using the martingale property, we find
  $$
  \cE_0 f_jf_{j-1}= \cE_0\cE_{j-1} f_jf_{j-1}=\cE_0 f_{j-1} \cE_{j-1} f_j=\cE_0 f_{j-1}^2.
  $$  
Using this in \eqref{eq:58}, we obtain the second claim of the lemma.
\end{proof}

Summing equations \eqref{eq:qv} on $j=0,1,\dots, J,$ we obtain
     \begin{equation}\label{eq:Levi}
     \cE_0[f]_J=\cE_0 f_J^2, \quad J\in\naturals_0.
     \end{equation}
Letting $f$ to be a Doob martingale and setting $J\to\infty,$ we obtain relations \eqref{eq:rI0} and \eqref{eq:J1} derived earlier
using a different language.

The Littlewood-Paley inequality \eqref{eq:LP} also has an analog in martingale theory. Suppose that $1<\alpha<\infty$ (again it is important
that $\alpha>1$, see \cite{Peshkir95}, \S6) and let $f=(f_0,f_1,f_2,\dots)$  be an arbitrary martingale with $f_0=0$. According to the {\em Burkholder inequality} \index{inequality!Burkholder}
\cite{Burk73}, \cite[Ch.11]{ChowTeicher97},
   \begin{equation}\label{eq:Burk}
   c_1(\alpha)\cE_0[f]_J^{\alpha/2}\le \cE_0|f_J|^\alpha\le c_2(\alpha)\cE_0[f]_J^{\alpha/2},
   \end{equation}
where the positive constants $c_1(\alpha),c_2(\alpha)$ depend neither on $J\in \naturals_0$ nor on the martingale $f$. \remove{For instance, we
 can take
  $$
  c_1(\alpha)=(18 \alpha^{3/2}/(\alpha-1))^{-1}, \;\;c_2(\alpha)=18\alpha^{3/2}(\alpha-1)^{1/2}.
  $$}
It is quite remarkable that a simple martingale property, \eqref{eq:mp3}, gives rise to an inequality as strong as \eqref{eq:Burk}.

It is easy to see that for Doob martingales the Littlewood-Paley and Burkholder inequalities are equivalent: indeed, specializing 
\eqref{eq:LP} for functions piecewise constant on the partition $K_J,$ we obtain \eqref{eq:Burk}. Conversely, starting from 
\eqref{eq:Burk} and letting $J\to\infty,$ we obtain \eqref{eq:LP}. The assumptions of $f_0=0$ and real-valued martingale are
not essential. It is easy to lift them at the expense of somewhat bigger constants in \eqref{eq:Burk}.
It is also possible to generalize these considerations from compact to locally compact 
zero-dimensional Abelian groups.

Summarizing, we observe that in the special situation of zero-dimensional Abelian groups, the Littlewood-Paley theory
and martingale theory form equivalent languages to describe the same set of results. 

%
%
\subsection{Nonmetric association schemes and Haar-like wavelets}\label{sect:nH}
Let us consider nonmetric schemes on zero-dimen\-sional groups introduced in Sect.~\ref{sect:nonmetric}. Recall that these schemes are constructed from spectrally dual partitions of $X$ into balls, see \eqref{eq:nmp}-\eqref{eq:^Phi}, 
that are obtained as refinements of partitions into spheres \eqref{eq:part}. In particular, spheres $\hat S(t),t\in\hat\R_0$
in the dual group $\hat X$ are partitioned into balls as follows: 
   \begin{equation}
   \hat S(t)=\bigcup_{j=1}^{\hat n(t)-1}\hat\Phi_j(t), \quad t>0
   \end{equation}
while for $t=0$ as before we put $\hat S(0)=\hat B(0)=\hat\Phi_0(0)=\{1\},$ where $1$ is the unit character 
(in this section, we again use the notation $\R,\R_0,$ etc. instead of $\I,\I_0$, etc., for reasons explained in the beginning
of Sect.~\ref{sect:nonmetric}).
Writing these partitions in terms of characteristic functions, we obtain
   \begin{align}
  \chi[\hat S(t);\xi]&=\sum_{j=1}^{\hat n(t)-1} \chi[\hat\Phi_j(t),\xi], \quad t>0,
\label{eq:cf}\\
   \chi[\hat S(0);\xi]&=\chi[\hat B(0);\xi]=\chi[\hat\Phi_0(0),\xi]=\1(\xi=1). \label{eq:cf0}
   \end{align}    
Using expression \eqref{eq:639} for the kernel of the projector $\Delta_r$, we obtain the following relation:
  $$
  \Delta_r(x,y)=\sum_{i=1}^{\hat n(\tilde r)-1}\Delta_{r,j}(x,y)
  $$
where
  \begin{align}
\Delta_{r,j}(x,y)&=\int_{\hat X} \overline{\xi(x-y)}\chi[\hat\Phi_j(\tilde r);\xi]d\hat\mu(\xi)\nonumber\\
&=
       \chi^\natural[\hat\Phi_j(\tilde r);x-y].\label{eq:co}
  \end{align}
If $X$ is compact and $\bar r$ denotes the maximum radius \eqref{eq:br}, then
   \begin{equation}\label{eq:1c}
  \Delta_{\bar r,0}(x,y)=\Delta_{\bar r}(x,y)=1 \quad \text{for all } x,y\in X.
   \end{equation}
Denote by $\Delta_{r,j}$ the integral operator with the kernel $\Delta_{r,j}(x,y).$ Eq.~\eqref{eq:co} implies that 
in $L_2(X,\mu)$ these operators are commuting orthogonal projectors:
    \begin{equation}\label{eq:sa}
\Delta_{r,j}^\ast=\Delta_{r,j}, \quad\Delta_{r_1,j_1}\Delta_{r_2,j_2}=\delta_{r_1,r_2}\delta_{j_1,j_2}\Delta_{r_1.j_1}.
    \end{equation}
Moreover, 
   \begin{equation}\label{eq:Ds}
   \Delta_r=\sum_{j=1}^{\hat n(\tilde r)-1}\Delta_{r,j}
   \end{equation}
and
   \begin{equation}\label{eq:sI}
  \sum_{r\in \R_0} \sum_{j=1}^{\hat n(\tilde r)-1}\Delta_{r,j}=I,
   \end{equation}
where $I$ is the identity operator in $L_2(X).$ This implies that the system of projectors $\{\Delta_{r,j}\}$ is complete
in $L_2(X).$
If the group $X$ is compact, then $\Delta_{\bar r,0}=\Delta_{\bar r}$ is a projector on the one-dimensional space of constants.
  
All the above formulas follow immediately from the fact that \eqref{eq:^nmp} is a partition of the group $\hat X$
into balls $\hat \Phi_j(t).$ Note also that we could take account of the relation $\hat n(\tilde r)=n(\tau_+(r))$ that
is dual to \eqref{eq:nr}. 

Comparing formulas \eqref{eq:cP} and \eqref{eq:co}, we obtain the following expression for the kernel of $\Delta_{r,j}:$
  $$
  \Delta_{r,j}(x,y)=\sum_{a\in \R}\sum_{j=1}^{n(a)-1} q_{\tilde r,j}(a,i)\chi[\Phi_i(a);x-y].
  $$

This implies the following theorem which is analogous to Theorem \ref{thm:LP}.
\begin{theorem}\label{thm:knm} Let 
$\cX=(X,\mu,\cR)$ be an association scheme on the zero-dimensional group $X$ constructed in Theorem \ref{thm:tnm}. 
Kernels of the projectors $\Delta_{r,j}, r\in \R_0$ are related to the second eigenvalues of $\cX$ as follows:
  \begin{equation}\label{eq:Dq}
  \Delta_{r,j}(x,y)=q_{\tilde r,j}(a,i) \quad\text{for }(x,y)\in R_{a,i}.
  \end{equation}
\end{theorem}

In the second part of this section we use this theorem to show that the eigenvalues of nonmetric schemes are related
to a class of complete systems of orthogonal functions on $L_2(X)$ which we call {\em Haar-like bases}. As a first step, we derive
another expression for the kernel $\Delta_{r,j}.$
\begin{lemma}\label{lemma:of} We have
   \begin{align}
   \Delta_{r,j}(x,y)&=\mu(B(\tau_+(r)))^{-1} \theta_{j,\tilde r}(x-y)\chi[B(\tau_+(r));x-y]\label{eq:pff}\\
   &=\mu(B(\tau_+(r)))\sum_{z\in X/B(\tau_+(r))}\Delta_{r,j}(x,z)\Delta_{r,j}(z,y)\label{eq:Drj}\\
   &=\mu(B(\tau_+(r)))\sum_{z\in X/B(\tau_+(r))}\Delta_{r,j}(x,z)\overline{\Delta_{r,j}(y,z)}. \label{eq:Dbrj}
   \end{align}
 \end{lemma}
\begin{proof}
First let us compute the
Fourier transform of the indicator function of the ball $\hat\Phi_j(\tilde r)$ in \eqref{eq:co}. 
Proceeding analogously to \eqref{eq:ir} and using \eqref{eq:^Phi} and \eqref{eq:f2}, we obtain
  \begin{align*}
  \chi^\natural[\hat\Phi_j(\tilde r);x-y]&=\int_{\hat X}\xi(x-y)\chi[\hat B(\tau_-(\tilde r))\theta_{j,r};\xi]d\hat\mu(\xi)\\
  &=\theta_{j,\tilde r}(x-y)\chi^\natural[\hat B(\tau_-(\tilde r));x-y]\\
  &=\hat\mu(\hat B(\tau_-(\tilde r))) \theta_{j,\tilde r}(x-y)\chi[\hat B(\tau_-(\tilde r))^\bot;x-y]
  \end{align*}
  where $\theta_{j,\tilde r}$ is the character defined in \eqref{eq:^Phi}. Using
  the equalities
   $$
   \hat B(\tau_-(\tilde r))^\bot=B(\tau_-(\tilde r)^\natural)=B(\tau_+(\tilde r^\natural))=B(\tau_+(r))
   $$
(see \eqref{eq:BB} and \eqref{eq:+-}) and $\hat\mu(\hat B(\tau_-(\tilde r)))\mu(B(\tau_+(r)))=1$ (see \eqref{eq:mu1}),
we obtain \eqref{eq:pff}.
 If the group $X$ is compact, then in this calculation
we also assume that $\tilde r>0,$ i.e., $r<\bar r$, while if $\tilde r=0$ and $r=\bar r,$ we rely on \eqref{eq:1c}. 

Now let us use \eqref{eq:cBxy} to rewrite the characteristic function in \eqref{eq:co}. Together with the multiplicative property
of the characters, this implies \eqref{eq:Drj}. Finally, \eqref{eq:Dbrj} follows by definition of $\Delta_{r,j}.$
\end{proof}
 
Let us remark that for metric schemes considered in this paper, 
relations of the form \eqref{eq:Drj} are generally not valid because translations of the sphere do not partition
$X$ (unless the sphere coincides with the ball), and so equation \eqref{eq:cBxy} does not have a proper analog. 

\vspace*{.1in}Lemma \ref{lemma:of} enables us to introduce a function system on the group $X$. The next definition is the main one in this section.
\begin{definition} {\bf (Haar-like bases)}\label{def:OF}
Define a system of functions on a zero-dimensional group $X$ as follows:
   \begin{align}
   \gls{psi}&=\mu(B(\tau_+(r)))^{-\half}\theta_{j,\tilde r}(x-z)\chi[B(\tau_+(r));x-z]\nonumber\\
   &=\mu(B(\tau_+(r)))^{\half}\Delta_{r,j}(x,z),\label{eq:psi}
   \end{align}
where the parameters satisfy
   $$
   r\in\R_0,\;\; j=1,\dots,n(\tau_+(r)),\;\; z\in X/B(\tau_+(r)).
   $$
 Here $B(r)$ is the ball in $X$ of radius $r$ around zero, and the character $\theta_{j,r}$ is defined in \eqref{eq:^Phi}.
 
   If the group $X$ is compact, then \eqref{eq:psi} applies for $r<\bar r,$ where $\bar r$ is the largest radius \eqref{eq:br},
   while for $r=\bar r$ we put by definition
   $
   \psi_{\bar r}(x)=\psi_{\bar r,0,0}(x)\equiv 1.
   $
\end{definition}
With this definition, Equation \eqref{eq:Drj} takes the following form:
    \begin{equation}
    \Delta_{r,j}(x,y)=\sum_{z\in X/B(\tau_+(r))}\psi_{r,j,z}(x)\overline\psi_{r,j,z}(y). \label{eq:gk}
    \end{equation}

\vspace*{.05in}{\em Remark:} This expression together with the properties of the projectors $\Delta_{r,j}$ 
suggests a link to the theory of {\em zonal spherical kernels} and {\em Gelfand pairs} (e.g., \cite{vanDijk09,vil68}); 
however developing it goes outside the scope of this paper. 

\vspace*{.05in} Using \eqref{eq:co}, we immediately find the Fourier transform of the functions $\psi:$
  \begin{equation}\label{eq:foupsi}
  \tilde\psi_{r,j,z}(\xi)=\mu(B(\tau_+(r)))^{-\half} \xi(z)\chi(\hat\Phi_j(\tilde r),\xi), \quad \xi\in\hat X.  
  \end{equation}
Note that both the functions $\psi$ and $\tilde\psi$ are very well localized: they are supported on the balls whose
radii are optimally correlated in terms of the uncertainty principle \eqref{eq:mm} which holds with equality because of \eqref{eq:mu1}.
  
Properties of the functions $\psi_{r,j,z}$ are summarized in the following theorem.
\begin{theorem}\label{thm:mra}
(i) The function system $\psi_{r,j,z}(x)$ defined in \eqref{eq:psi} forms an orthonormal basis of the space $L_2(X).$

(ii) For each $r\in\R_0$ the subsystem of functions $\psi_{r,j,z}(x)$ forms an orthonormal basis of the space $\Delta_r L_2(X).$

(iii) For each $r\in\R_0$ and $j=1,\dots,n(\tau_+(r))$ the subsystem of functions $\psi_{r,j,z}(x)$ forms an orthonormal basis
in the space $\Delta_{r,j}L_2(X).$
\end{theorem}
\begin{proof}
Let us prove orthogonality:
  \begin{equation}\label{eq:port}
  \int_X\psi_{r_1,j_1,z_1}(x)\overline{\psi_{r_2,j_2,z_2}(x)}d\mu(x)=\delta_{r_1,r_2}\delta_{j_1,j_2}\delta_{z_1,z_2}.
  \end{equation}
Using \eqref{eq:ip} we can write
   $$
   \int_X\psi_{r_1,j_1,z_1}(x)\overline{\psi_{r_2,j_2,z_2}(x)}d\mu(x)=
   \int_{\hat X} \tilde\psi_{r_1,j_1,z_1}(\xi)\overline{\tilde \psi_{r_2,j_2,z_2}(\xi)}d\hat \mu(\xi)
  $$
Substituting \eqref{eq:foupsi} into the right-hand side, we conclude that \eqref{eq:port} holds true if $r_1\ne r_2$ or $j_1\ne j_2$
since in this case $\hat\Phi_{j_1}(\tilde r_1)\cap\hat\Phi_{j_2}(\tilde r_2)=\emptyset.$ Similarly, substituting \eqref{eq:psi}
into the left-hand side, we conclude that \eqref{eq:port} holds true if $z_1\ne z_2$ since in this case
$(B(\tau_+(r))+z_1)\cap(B(\tau_+(r))+z_2)=\emptyset.$

Let us prove that these functions form a basis. Introduce operators $\Psi_{r,j,z}$ given by
  $$
  \Psi_{r,j,z}f(x)=\psi_{r,j,z}(x)\int_X\overline{\psi_{r,j,z}(y)}f(y)d\mu(y)
  $$
These operators are orthogonal projectors in $L_2(X)$ on one-dimensional subspaces spanned by the functions $\psi_{r,j,z}(x).$
Using \eqref{eq:sa} and \eqref{eq:port} we see that 
   \begin{gather*}
   \Psi^\ast_{r,j,z}=\Psi_{r,j,z}\\
   \Psi_{r_1,j_1,z_1}\Psi_{r_2,j_2,z_2}=\delta_{r_1,r_2}\delta_{j_1,j_2}\delta_{z_1,z_2}\Psi_{r_1,j_1,z_1}.
   \end{gather*}
In particular, the operators $\Psi$ are commuting idempotents.   
   
Eq.\eqref{eq:gk} can be written as an operator equality:
  \begin{equation}\label{eq:oe}
  \Delta_{r,j}=\sum_{z\in X/B(\tau_+(r))}\Psi_{r,j,z},
  \end{equation} 
whereupon \eqref{eq:Ds} and \eqref{eq:sI} take the form
  \begin{gather}
  \Delta_{r,j}=\sum_{j=1}^{\hat n(\tilde r)-1}\sum_{z\in X/B(\tau_+(r))}\Psi_{r,j,z}\label{eq:DP}\\
  \sum_{r\in\R_0}\sum_{j=1}^{\hat n(\tilde r)-1}\sum_{z\in X/B(\tau_+(r))}\Psi_{r,j,z}=I.\label{eq:DPI}
  \end{gather}
The operator equalities \eqref{eq:oe}-\eqref{eq:DPI} are understood in the strong sense: they hold true upon being applied
to any function $f\in L_2(X),$ where the corresponding series converge in the metric of the space $L_2(X).$
\end{proof}

On account of this theorem, Part(i), any function $f\in L_2(X)$ has the Fourier series
  \begin{equation}\label{eq:pf}
  f(x)\cong \sum_{r\in\R_0}\sum_{j=1}^{\hat n(\tilde r)-1}\sum_{z\in X/B(\tau_+(r))} a_{r,j,z}(f)\psi_{r,j,z}(x)
  \end{equation}
with coefficients 
  $$
  a_{r,j,z}(f)=\int_X\overline{\psi_{r,j,z}(x)}f(x)d\mu(x),
  $$
and the following relation holds:
  \begin{equation}\label{eq:energy}
  \int_X |f(x)|^2d\mu(x)=
    \sum_{r,j,z}|a_{r,j,z}(f)|^2.
  \end{equation}

Similarly to the Littlewood-Paley theory it turns out that the expansion \eqref{eq:pf} is stable with respect to the norm change. Define the following quadratic function:
  \begin{align*}
  S_\Psi f(x)&=\Big(\sum_{r\in\R_0}\sum_{j=1}^{\hat n(\tilde r)-1}\sum_{z\in X/B(\tau_+(r))} |a_{r,j,z}(f)\psi_{r,j,z}(x)|^2\Big)^\half\\
  &=\Big(\sum_{r\in\R_0}\sum_{j=1}^{\hat n(\tilde r)-1}\sum_{z\in X/B(\tau_+(r))} |a_{r,j,z}(f)|^2
  \mu(B(\tau_+(r)))^{-1}\chi[B(\tau_+(r));x-z]\Big)^\half.
  \end{align*}
Using this notation, we can rewrite \eqref{eq:energy} as follows:
  $$
  \|S_\Psi f\|_2=\|f\|_2.
  $$
It turns out that \eqref{eq:pf} is stable if the Hilbert space $L_2(X)$ is replaced by the Banach space $L_\alpha(X), 1<\alpha<\infty$
(note that condition $\alpha\ne 1$ is essential here). Specifically, for any function $f\in L_\alpha(X)$ we have
   $$
   c_1(\alpha)\|S_\Psi(f)\|_\alpha\le \|f\|_\alpha\le c_2(\alpha)\|S_\Psi f\|_\alpha,
   $$
where the positive constants $c_1,c_2$ do not depend on $f$. 
We do not prove this inequality here because it involves nontrivial aspects of multiplier theory in the spaces $L_\alpha(X)$ \cite{edw77},
which again is outside the scope of this paper.

A natural way to think of functions \eqref{eq:psi} introduced above is by interpreting them as {\em Haar-like wavelets} on the
group $X.$ 
To explain this point of view, recall the definition of wavelets in the context  of {\em multiresolution analysis} (e.g., \cite{Grochenig92}; \cite[Sect.~2.2,\,8.3]{Wojta99}, \cite[Ch.2]{NPS11}). 
\index{multiresolution analysis}
A sequence of monotone increasing closed subspaces $V_j\subset V_{j+1}, j\in \integers$ of $L_2(X)$ 
is called a multiresolution approximation of $L_2(X)$ if $\cap_{j} V_j=\{0\}$ and $\overline{\cup_{j} V_j}$ is dense in $L_2(X),$
if there exists a {\em scaling function} $\phi\in V_0$ whose translations form an orthonormal basis of $V_0$, and if there is a
dilation operator $A$ such that $f(x)\in V_j$ if and only if $f(Ax)\in V_{j+1}.$
In this case the set of all dilations and translations of the scaling function forms a complete orthonormal system in $L_2(G),$
called a wavelet basis.
Because of this, the obtained basis is said to have a {\em self-similarity property}. Note that a general construction of self-similar
wavelets generated by partiotoins of the group $\reals^n$ was defined in \cite{Grochenig92} and is discussed in detail in 
\cite[Sect.2.8]{NPS11}. In \cite{Lukom10} this construction was extended to zero-dimensional Abelian groups.

In our situation, the corresponding sequence of subspaces is given by $\cL_2[K(r)]$, and properties \eqref{eq:mra0}, \eqref{eq:mra1}, and \eqref{eq:dense} ensure that it forms a multiresolution approximation of $L_2(X).$ Moreover, by Theorem \ref{thm:mra}(ii), for
each $r\in\R_0$ the subsystem of functions \eqref{eq:psi} forms an orthonormal basis in the orthogonal complement $\cW_r=\Delta_r L_2(X),$
see \eqref{eq:Wr}. The only difference with the classical definition is that the system of functions \eqref{eq:psi} generally does not
have the self-similarity property.

Non-self-similar wavelets were considered in the literature; see, e.g., \cite{Girardi97,Novikov2012}. In \cite{Girardi97},
Sect.8 they are even called ``second generation wavelets,'' but the approach taken in that paper is so general that the corresponding
theory essentially coincides with martingale theory. Wavelets \eqref{eq:psi} considered here are much more specific in that they fully account for the group structure of the measure space $(X,\mu).$
{Paper \cite{Novikov2012} considers non-self-similar wavelet bases such that the spaces $V_j$ are generated
by characteristic functions of a partition of a general measure space $\Omega.$ This paper identifies general sufficient conditions for a 
partition to form a multiresolution approximation of $L_2(\Omega)$. }

We note that if the sequence of nested balls \eqref{eq:BB1} is self-similar with respect to some expansive (or contractive)
automorphism of the group $X$, then the system of functions \eqref{eq:psi} that arises from this chain will also be self-similar.
Such automorphisms of $X$ are easily defined if, for instance, all the quotient groups $B(\tau_+(r))/B(r), r\in \R_0$ are isomorphic
to the one and the same finite Abelian group. 

Using the mapping of $X\to [0,1]$ \eqref{eq:map} (or $X\to[0,\infty)$ \eqref{eq:map-lc})
it is possible to represent the wavelets on $X$ as functions on the real line. 
For the self-similar case this is done in \cite{Lukom10}, while the general case has not been studied in detail.

We shall limit our discussion in this part to the above brief remarks because wavelet theory {\em per se} is not a subject of the 
present work. Our main goal here is to point out a link between wavelet theory and association schemes and spectrally dual partitions
on zero-dimensional groups introduced in this paper.

%
\section{Subsets in association schemes: Coding theory}\label{sect:coding}
One of the main applications of the classical theory of association schemes relates to coding theory \cite{del73a,del98}. 
In particular, association
schemes provide a natural context for the study of code duality including the celebrated MacWilliams theorem \cite{mac63,mac91}, its numerous extensions and applications. From the perspective of harmonic analysis, the MacWilliams theorem is an instance of the Poisson summation formula.
It it is similarly well known that this theorem is naturally connected with spectrally dual partitions \cite{zin96,Forney98,hgl12}. In this section we derive an extension of these concepts to the general association schemes introduced in this paper.

Let $X$ and $\hat X$ be a pair of dual locally compact Abelian groups and let $\cN=(N_i,i\in\I)$ and $\widehat\cN=(\hat N_i, i\in\I)$ be
finite or countably infinite partitions of $X$ and $\hat X$ respectively (recall that by our assumption, the measure of each 
block of the partition is finite). We assume that the partitions $\cN$ and $\widehat\cN$
are spectrally dual and give rise to association schemes $\cX$ and $\widehat\cX$ according to the result of Theorem \ref{thm:SDP}.

Define a {\em code} $Y$ \index{code}
to be a compact subgroup of $X$. The {\em dual code} \index{code!dual} of $Y$ is the annihilator of $Y$ in $\hat X$:
  $$
  Y^\bot=\{\phi\in \hat X: \phi(y)=1 \text{ for all }y\in Y\}.
  $$
Clearly, $Y^\bot$ is a compact subgroup of $\hat Y,$ and $Y^\bot\cong\widehat{X/Y}.$ By the Poisson summation formula \cite[\S31]{hew6370}
        \begin{equation}\label{eq:poisson}
        \int_Y f(x)d\mu(x)=\int_{Y^\bot} \tilde f(\phi)d\hat\mu(\phi),
       \end{equation}
where $f$ is a function on $Y$ taking values in $\complexes$ or in any finite-dimensional vector space over $\complexes,$ 
and the integrals are assumed to exist. 
A particular form of \eqref{eq:poisson} that is commonly used in coding theory is related to the notion of the weight distribution of the code
(see \cite[Thm.6.3]{del73a},~\cite[p.71]{bro89}).
Define the {\em weight distribution} \index{weight distribution}
of $Y$ as $m=(m_i, i\in\I_0),$ where $m_i=\mu(Y\cap N_i).$ Similarly, 
$\hat m=(\hat m_i,i\in\hat\I_0),$ where $\hat m_i=\hat\mu(Y^\bot\cap \hat N_i),$ is the weight distribution of the dual code $Y^\bot.$ Note that $  \sum_{i\in \I_0} m_i=\mu(Y).$

\begin{theorem} The weight distributions of a code $Y$ and its dual code $Y^\bot$ satisfy the MacWilliams relations
   \begin{equation} \label{eq:mw}
      \hat m_i=\frac1{\mu(Y)}\sum_{k\in \I_0} q_i(k) m_k, \quad m_i=\mu(Y)\sum_{k\in\hat\I_0} p_i(k)\hat m_k.
   \end{equation}
Therefore $\sum_{k\in \I_0} q_i(k) m_k\ge 0$ for all $i\in\hat\I_0.$
   \end{theorem}
{\em Remark:} The summation regions can be extended from $\I_0$ to $\I$ because (assuming that the measures
are complete), $\mu(Y\cap N_i)=0$ if $\mu(N_i)=0.$

\begin{proof}
Clearly,
   $$
   \int_Y \phi(y) d\mu(y)=\mathbbold{1}\{\phi\in Y^\bot\}\cdot \mu(Y).
   $$
Therefore,
  \begin{align}
  \hat m_i&=\hat\mu(Y^\bot\cap \hat N_i)=\frac1{\mu(Y)}\int_{\hat N_i}\int_Y \overline{\phi(y)}d\mu(y)d\hat\mu(\phi)\nonumber\\
    &=\frac1{\mu(Y)^2}\int_{\hat N_i}\Big(\int_{Y}\int_Y \overline{\phi(y)}{\phi(y')}d\mu(y) d\mu(y')\Big)d\hat\mu(\phi)\label{eq:iii}
      \end{align}
At the same time, let $\chi(y)=\mathbbold{1}\{y\in Y\}.$ Using \eqref{eq:Ek} and \eqref{eq:Ei4} we obtain that
  \begin{equation}\label{eq:imp}
  E_i\chi(y)=\int_X \hat \chi_i^\natural(y-y')\chi(y')d\mu(y')
  =\int_X\int_{\hat N_i}\overline{\phi(y)}\phi(y')\chi(y')d\hat\mu(\phi)d\mu(y').
    \end{equation}
Using Fubini's theorem, we conclude that
  \begin{equation}\label{eq:hmi}
  \hat m_i=\frac1{\mu(Y)^2}\int_X\chi(y)(E_i\chi)(y)d\mu(y).
  \end{equation}
On the other hand, using \eqref{eq:Ei2}, we obtain
  \begin{equation}\label{eq:mim}
  \int_X\chi(y)(E_i\chi)(y)d\mu(y)=\sum_{k\in \I_0}q_i(k)\int_X\chi(y)(A_k\chi)(y)d\mu(y).
  \end{equation}
We observe that $A_k\chi(y)=\mu\{y'\in Y: y-y'\in N_k\},$ so the last integral evaluates to $\mu(Y)m_k,$ establishing
the first of the claimed relations. To prove the second, multiply the first one by $p_j(i)$ and sum on $i\in\hat\I_0:$
  $$
  \sum_{i\in\hat\I_0}p_j(i)\hat m_i=\frac1{\mu(Y)}\sum_{k\in\I_0} m_k\sum_{i\in\I_0}p_j(i)q_i(k).
  $$
To interchange the order of summation we need absolute convergence which can be checked similarly to the proof 
of Theorem \ref{thm:int}. On account of \eqref{eq:pipj} and \eqref{eq:mv}, the sum on $i$ on the right equals
$\delta_{jk},$ establishing the second equality in \eqref{eq:mw}.
\end{proof} 
  
We make two remarks that parallel the classical coding theory.

1. ({\em Delsarte inequalities.}) 
\index{inequalities!Delsarte}
Let $\cX$ be an association scheme defined on a topological Abelian group $X$, and let $Y\subset X$ be an arbitrary
subset of finite measure. Define 
   $$   
n_k=\frac 1{\mu(Y)}\mu\{(y,y')\in Y^2:\,(y,y')\in R_k\}, k\in\I_0.
   $$
Then 
  $$
   \sum_{k\in\I_0}q_i(k)n_k\ge 0.
  $$
This inequality follows because the projectors $E_k, k\in \I$ are positive semidefinite. Indeed, let $\chi$ be the 
indicator function of $Y$ in $X$. Arguing as in \eqref{eq:mim}, we observe that 
   $$
     \mu(Y)n_k=\sum_{k\in \I_0}q_i(k)\int_X\chi(y)(A_k\chi)(y)d\mu(y).
   $$
Using \eqref{eq:imp}, the left-hand side of \eqref{eq:mim} can be evaluated, and we obtain
   \begin{align*}
  \mu(Y)\sum_{k\in\I_0} n_kq_i(k)&=\int_{\hat N_i} |\tilde\chi(\phi)|^2 d\hat\mu(\phi)\ge 0.
  \end{align*}

2. Similarly to the classical case, it is also possible to define designs in association schemes. Namely, call a subgroup $Y\in X$ a $T$-{\em design} if $\sum_{k\in\I_0}q_i(k)n_k=0$ for all $i\in T,$ where $T$ is a subset of $\hat\I_0.$
\index{$T$-design}

\vspace*{.05in}
We have defined a code as a compact subgroup of $X$. At the same time, in classical coding theory, a code is a {\em finite} subset
of the ambient metric space or association scheme \cite{mac91,bro89}. Adopting this definition, we observe that for the case of 
non-Archimedean metric many problems of the classical theory become trivial. 

\vspace*{.05in}
{\em Finite codes:} Let $X$ be a homogeneous metric space with distance function $\rho$ and let $Y$ be a finite subset which we call a code.  \index{code!finite}
The value $\rho(Y)=\min\{\rho(x,x'), x,x'\in Y, x\ne x'\}$ is called the minimum distance of the code $Y$.
Denote by $B(x,r)$ a metric ball in $X$ of radius $r$. Let $r$ be a value of the radius such that $B(x,r)\cap B(x',r)=\emptyset$
if $x,x'\in Y, x\ne x'.$ If at the same time, $\cup_{x\in Y} B(x,r)=X$ then the code $Y$ is called {\em perfect}.  
Existence of perfect codes in Hamming spaces is one of the major open problems of coding theory, in which 
the answer is known only if $X$ forms a vector space over a finite field; see \cite{mac91}, Ch.~6. At the same time, metric schemes of 
Sect.~\ref{sect:metric} contain perfect codes for any value of the radius $r$. Indeed, let $X$ be a zero-dimensional topological 
compact group
with distance \eqref{eq:dis} and consider a partition \eqref{eq:pb1} of
$X$ into balls:
   \begin{equation}\label{eq:pb2}
     X=\bigcup_{z\in X/B(r)} B(r)+z,
   \end{equation}
 where $B(r)$ denotes the ball of radius $r$ around the identity element.
Thus, the translations of $B(r)$ form a tight packing of the group $X$. Now form a code $Y$ by taking any one point in each of the 
balls. By a standard fact in non-Archimedean geometry, every point of the ball is its center. Indeed, let $x$ be the chosen center,
let $y$ be such that $\rho(x,y)\le r,$ and let $x'\in B(r), x'\ne x.$ From \eqref{eq:um1}, $\rho(x',y)\le r,$ and so $y\in B(x',r).$
The same argument obviously applies to a translation of $B(r)$ by an element $z$.
Thus the collection of points $Y=\{0\}\cup\{z\in X/B(r)\},$ where the coset representatives are chosen arbitrarily,
forms a perfect code in $X$. The minimum distance of the code is $\rho(Y)=r$. 
Note moreover that every pair of distinct points $y_1,y_2\in Y$ satisfy $\rho(y_1,y_2)=r$, making $Y$ into a ``simplex'' code.
This remark again should be contrasted with the classical case of the Hamming space in which simplex codes exist only for
a very special set of parameters \cite{mac91}, Ch.1.
 
   Observe that this construction also resolves the non-Archimedean version of the main metric problem of coding theory which concerns the cardinality of the
largest packing of the metric space $X$. For instance, if $X$ is the Hamming space, then the best known general results are obtained
by a greedy procedure that packs balls of radius $d-1$ into $X$ ``for as long as possible.'' The cardinality of the resulting code
is said to attain the Gilbert-Varshamov bound. At the same time, upper bounds on $|Y|$ for a given value of distance $\rho(Y)$
diverge from this bound, leaving a gap between the known constructions and the impossibility theorems; see \cite[Ch. 17]{mac91} and 
\cite{del98}. In particular, according to the Hamming bound, any code $Y$ satisfies $|Y|\le |X|/\text{vol\,}(B(\half(d-1))).$ 
In the infinite non-Archimedean case there is no gap 
between the upper and lower sphere-packing bounds on codes (note that all the balls in the partition \eqref{eq:pb2} have equal measure \eqref{eq:Hm}).
 
Finally, in classical coding theory, most attention is devoted to group and linear codes, i.e., finite subgroups of the space $X$.
Examples of such codes in ultrametric spaces of the form \eqref{eq:NRT} were extensively studied in the literature \cite{skr01,mar99}.
In the case of zero-dimensional groups, finite subgroups exist if $X$ is periodic, e.g., a Cantor-type group, and do not exist  
for non-periodic groups such as the additive group of $p$-adic numbers.

\providecommand{\bysame}{\leavevmode\hbox to3em{\hrulefill}\thinspace}
\providecommand{\MR}{\relax\ifhmode\unskip\space\fi MR }
\providecommand{\MRhref}[2]{%
  \href{http://www.ams.org/mathscinet-getitem?mr=#1}{#2}
}
\providecommand{\href}[2]{#2}

\clearpage


\printglossary
{\renewcommand{\thefootnote}{}\footnotetext{
\vspace{-.3in}
We list only objects and functions related to the group $X$, omitting the analogous notions for the dual group $\hat X$.
}}
\renewcommand{\thefootnote}{\arabic{footnote}}

\printindex
\end{document}